\documentclass[12pt]{amsart}
\usepackage{latexsym,amsmath,amsfonts,amscd,amssymb}
\usepackage{stmaryrd}
\usepackage{lscape}
\usepackage{array}   
\usepackage[mathscr]{eucal} 
\usepackage{mathrsfs}
\usepackage{graphicx}
\usepackage{xypic}
\usepackage{calligra}
\usepackage[T1]{fontenc}
\DeclareFontFamily{OT1}{pzc}{}
\DeclareFontShape{OT1}{pzc}{m}{it}{<-> s * [1.10] pzcmi7t}{}
\DeclareMathAlphabet{\mathpzc}{OT1}{pzc}{m}{it}

\DeclareFontFamily{OT1}{rsfs}{}
 \DeclareFontShape{OT1}{rsfs}{n}{it}{<->rsfs10}{}
 \DeclareMathAlphabet{\curly}{OT1}{rsfs}{n}{it}

\theoremstyle{plain}  
\newtheorem{theorem}{Theorem}[section]

\newtheorem*{theorem*}{Theorem}

\newtheorem{lemma}[theorem]{Lemma}
\newtheorem{proposition}[theorem]{Proposition}

\newtheorem{definition}[theorem]{Definition}
\theoremstyle{remark}

\newtheorem{remark}[theorem]{Remark}

\newtheorem*{claim*}{Claim}
\numberwithin{equation}{section}

\renewcommand{\leq}{\leqslant}

\renewcommand{\geq}{\geqslant}

\newcommand{\R}{\mathbb{R}}

\newcommand{\Z}{\mathbb{Z}}
\newcommand{\C}{\mathbb{C}}

\newcommand{\GGG}{\curly{G}}

\newcommand{\UUU}{\curly{U}}

\newcommand{\cM}{\mathcal{M}}

\newcommand{\calR}{\mathcal{R}}

\newcommand{\bE}{{\bf{E}}}

\newcommand{\dbar}{\bar{\partial}}

\newcommand{\lra}{\longrightarrow}
\newcommand{\st}{\;:\;}

\newcommand{\PSL}{\mathrm{PSL}}

\newcommand{\SU}{\mathrm{SU}}

\newcommand{\OO}{\mathrm{O}}
\newcommand{\GL}{\mathrm{GL}}
\newcommand{\SL}{\mathrm{SL}}
\newcommand{\NSL}{\mathrm{NSL}}
\newcommand{\SSS}{\mathrm{S}}
\newcommand{\SO}{\mathrm{SO}}
\newcommand{\Sp}{\mathrm{Sp}}
\newcommand{\Spin}{\mathrm{Spin}}

\DeclareMathOperator{\ad}{ad}
\DeclareMathOperator{\Ad}{Ad}

\DeclareMathOperator{\rank}{rank}

\DeclareMathOperator{\Hom}{Hom}
\DeclareMathOperator{\End}{End}

\DeclareMathOperator{\Id}{Id}

\DeclareMathOperator{\Aut}{Aut}
\DeclareMathOperator{\Int}{Int}
\DeclareMathOperator{\Out}{Out}
\DeclareMathOperator{\Tr}{Tr}
\DeclareMathOperator{\Nm}{Nm}
\DeclareMathOperator{\cl}{\mathpzc{cl}}
\DeclareMathOperator{\conj}{Conj}

\newcommand{\la}{\langle}
\newcommand{\ra}{\rangle}

\newcommand{\Pic}{\operatorname{Pic}}

\renewcommand{\phi}{\varphi}

\newcommand{\liep}{\mathfrak{p}}

\newcommand{\lieg}{\mathfrak{g}}
\newcommand{\liel}{\mathfrak{l}}
\newcommand{\liez}{\mathfrak{z}}
\newcommand{\liegl}{\mathfrak{gl}}
\newcommand{\liesu}{\mathfrak{su}}
\newcommand{\lieso}{\mathfrak{so}}
\newcommand{\liesl}{\mathfrak{sl}}
\newcommand{\lieu}{\mathfrak{u}}

\renewcommand{\phi}{\varphi}
\newcommand{\Mg}{\mathcal{M}^{\operatorname{gauge}}}

\begin{document}

\title[Involutions and higher order automorphisms of Higgs bundle moduli spaces]
{Involutions and higher order automorphisms \\ of Higgs bundle moduli spaces}

\author[Oscar Garc{\'\i}a-Prada]{Oscar Garc{\'\i}a-Prada}
\address{Instituto de Ciencias Matem\'aticas \\
  CSIC-UAM-UC3M-UCM \\ Nicol\'as Cabrera, 13--15 \\ 28049 Madrid \\ Spain}
\email{oscar.garcia-prada@icmat.es}

\author[S. Ramanan]{S. Ramanan}
\address{Chennai Mathematical Institute\\
H1, SIPCOT IT Park, Siruseri\\
Kelambakkam 603103\\
India}
\email{sramanan@cmi.ac.in}

\thanks{
  Partially supported by the Spanish MINECO under the ICMAT Severo Ochoa 
grant No. SEV-2011-0087, and grant No. MTM2013-43963-P and by the European
Commission Marie Curie IRSES  MODULI Programme PIRSES-GA-2013-612534. 
}


\begin{abstract}
We consider the moduli space $\cM(G)$ of 
$G$-Higgs bundles over a compact Riemann surface $X$, 
where $G$ is a complex semisimple Lie group. 
This is a hyperk\"ahler manifold homeomorphic to the 
moduli space $\calR(G)$ of representations of the fundamental
group of $X$ in $G$. In this paper we study
finite order automorphisms of $\cM(G)$ obtained by 
combining the action of an element of order $n$ 
in $H^1(X,Z)\rtimes\Out(G)$, where $Z$ is the centre of 
$G$ and $\Out(G)$ is the group of outer 
automorphisms of $G$, with the multiplication 
of the Higgs field by an $n$th-root of unity,
and describe the subvarieties of fixed points.
We give special attention to the case of involutions, 
defined by the action of an element of order $2$
in $H^1(X,Z)\rtimes\Out(G)$ combined with the  multiplication of the
Higgs field by $\pm 1$. In this
situation, the subvarieties of fixed 
points are hyperk\"ahler submanifolds of $\cM(G)$ in the  (+1)-case,
corresponding to the moduli space of representations of
the fundamental group in  certain reductive complex subgroups of $G$ defined
by holomorphic involutions of $G$; while  in the (-1)-case they
are Lagrangian subvarieties corresponding to the moduli space of  representations of 
the fundamental group of $X$ in real forms 
of $G$ and certain extensions of these. 
We illustrate the general theory with the description of involutions 
for $G=\SL(n,\C)$ and involutions and order three automorphism defined
by triality for $G=\Spin(8,\C)$. 

\end{abstract}

\maketitle

\section{Introduction}

Let 
$G$ be a complex connected semisimple  Lie group with Lie algebra $\lieg$. 
Let $X$ be a smooth projective curve over $\C$, equivalently a compact 
Riemann surface. A  $G$-Higgs bundle over $X$ is a pair $(E,\varphi)$ where
$E$ is  a principal $G$-bundle $E$ over $X$ and
$\varphi$ is a  section of $E(\lieg)\otimes K$, where $E(\lieg)$ is
the bundle associated to $E$ via the adjoint representation of $G$, 
and $K$ is the canonical bundle on $X$.
These objects were introduced by Hitchin \cite{hitchin1987,hitchin:duke}.
There are  notions of (semi)stability and polystability for $G$-Higgs bundles  
which allow to  consider  the   moduli space of polystable
$G$-Higgs bundles $\mathcal{M}(G)$, which  has the structure of a complex 
algebraic variety.  Let $Z$ be the centre of $G$ and let $\Out(G)$ be the group of outer 
automorphisms of $G$. The group $H^1(X,Z)\rtimes\Out(G)\times \C^*$
acts naturally on $\mathcal{M}(G)$. In this paper we study finite
order automorphisms  
defined by elements of this group and describe their fixed points, 
with special emphasis in the case of involutions --- 
the action of $\Aut(X)$ on  $\mathcal{M}(G)$ is considered in \cite{basu-garcia-prada}.  
The simplest involution $(E,\varphi)\mapsto (E,-\varphi)$ was already studied by Hitchin 
when  $G=\SL(2,\C)$ in \cite{hitchin1987}.

An important feature of $G$-Higgs bundles is their relation with 
representations 
of the fundamental group of $X$ in $G$. By  a representation we mean a 
homomorphism $\rho:\pi_1(X)\to G$. We say that $\rho$ is  reductive if
the  Zariski closure of the  image of $\rho$ in $G$ is a  reductive group,
or equivalently, if the composition of $\rho$ with the adjoint representation
of $G$  gives a totally reducible representation of $\pi_1(X)$ in $\lieg$.
The set of reductive representations in $G$ modulo conjugation by $G$, which
we will denote by $\calR(G)$, is called the moduli space of representations of 
$\pi_1(X)$ in $G$, and  has the structure of a complex algebraic variety. 
Non-abelian Hodge theory establishes that $\cM(G)$ and $\calR(G)$ are 
homeomorphic.
This is proved by combining results by Hitchin \cite{hitchin1987} and
Donaldson \cite{donaldson} for $\SL(2,\C)$ and Simpson
\cite{simpson,simpson:1992} and Corlette \cite{corlette} in general.
In this paper we establish the relation of the fixed points of the involutions, and 
higher order automorphisms  of $\cM(G)$ mentioned above, with representations of $\pi_1(X)$ in
$G$.

An important ingredient in the theory of Higgs bundles is the  $\C^*$ action
 on $\cM(G)$ defined  by  $(E,\varphi)\mapsto (E,\lambda \varphi)$ for 
$\lambda \in \C^*$. In particular, multiplication by
$-1$ sending $(E,\varphi)$ to  $(E,-\varphi)$  
defines the simplest non-trivial involution in $\cM(G)$ that we consider in this 
paper. To describe the fixed points, consider the group $\Int(G)$ of inner
automorphisms  of $G$, and an element $\theta\in \Int(G)$ of order $2$.
Then $\theta$ gives an involution  of $\lieg$ which can be decomposed
in terms of $\pm 1$-eigenspaces as $\lieg={\lieg}^+\oplus {\lieg}^-$, where  
clearly ${\lieg}^+$ is the Lie algebra of $G^\theta$, the subgroup of fixed points under 
$\theta$. The restriction of the adjoint representation of $G$ to $G^\theta$
gives the adjoint representation of $G^\theta$ in ${\lieg}^+$ as well as the isotropy
representation  $G^\theta\to \GL(\lieg^-)$. The fixed points of our involution
$(E,\varphi)\mapsto (E,-\varphi)$ are given by pairs $(E,\varphi)$ consisting
of  a $G^\theta$-bundle $E$ and $\varphi\in H^0(X,E(\lieg^-)\otimes K)$, where
$E(\lieg^-)$ is the bundle associated to the isotropy representation.
Two involutions $\theta,\theta'\in \Int_2(G)$ which are equivalent, in the sense
that there is an $\alpha\in \Int(G)$ such that
$\theta'=\alpha\theta\alpha^{-1}$, define the same locus in  the fixed point
set in $\cM(G)$. So the fixed point set is determined basically by the set $\Int_2(G)/\sim$,
where $\Int_2(G)$ is the set of elements of order $2$ in $\Int(G)$ and $\sim$
is the equivalence relation defined above.

In \cite{cartan} Cartan shows that there exists a conjugation $\tau$ of $G$, 
i.e. an
antiholomorphic involution, defining a maximal compact subgroup of $G$,
such that in  each class in  $\Int_2(G)/\sim$ there is a $\theta$ 
commuting with $\tau$. For such $\theta$, we can consider the conjugation of 
$G$ defined by $\sigma:=\theta\tau$. The fixed point subgroup $G^\sigma$ defines 
a real form of $G$, which is inner equivalent to the compact real form, 
what  is known as a real form of Hodge type. The fixed points of the involution
$(E,\varphi)\mapsto (E,-\varphi)$ correspond precisely with the representations
of  $\pi_1(X)$ in $G^\sigma$.
  
Of course, in general, there are other real forms which are not of Hodge type.
In this paper we show that representations of $\pi_1(X)$ in these other real 
forms are also related  to fixed points of appropriate involutions of $\cM(G)$. 
To define these involutions, we consider the group $\Aut(G)$ of holomorphic 
automorphisms of $G$. The natural action of $\Aut(G)$ on $\cM(G)$ decends to
an action of the outer group of automorphisms 
$\Out(G)=\Aut(G)/\Int(G)$. In particular we can consider the set $\Out_2(G)$ of 
elements of order $2$ in $\Out(G)$ and define for $a\in \Out_2(G)$ the 
involutions $\iota(a,\pm):\cM(G)\to \cM(G)$ given by 
$(E,\varphi)\to (a(E),\pm a(\varphi))$. We can extend the equivalence 
relation defined above to the set  $\Aut_2(G)$  of elements of order $2$
of $\Aut(G)$, establishing for  $\theta,\theta'\in \Aut_2(G)$ 
that $\theta\sim \theta'$ if  there is an $\alpha\in \Int(G)$ such that
$\theta'=\alpha\theta\alpha^{-1}$. This equivalence descends to define a 
surjective map  $\cl: \Aut_2(G)/\sim  \to \Out_2(G)$. Now, for 
$\theta\in \Aut_2(G)$, as in the case of $\theta\in \Int_2(G)$,  we have 
the decomposition $\lieg=\lieg^+\oplus \lieg^-$ and the representations 
$G^\theta\to \GL(\lieg^\pm)$.  The fixed points of $\iota(a,\pm)$
are given by $(E,\varphi)$ where  $E$ is a $G^\theta$-bundle,
$\varphi\in H^0(X,E(\lieg^\pm)\otimes K)$, and $[\theta]\in \cl^{-1}(a)$.
As in the previous case, involutions defining the same class $[\theta]$ 
determine
the same locus in $\cM(G)$. To relate to real forms of $G$, we first note
that Cartan's result mentioned above applies in fact more generally to 
$\Aut_2(G)/\sim$ and hence in each class we can find a representative
$\theta$ commuting with the compact conjugation $\tau$, such that 
$\sigma:=\theta\tau$ is a conjugation of $G$ defining
a real form $G^\sigma$. We will say that $\sigma$, and the corresponding 
$\theta$ are in the `clique' $a$. The real forms of
Hodge type are thus the ones corresponding to the trivial clique in $\Out_2(G)$.
Now, the fixed points of $\iota(a,+)$ are given by  $G^\theta$-Higgs bundles,
thus corresponding to representations of $\pi_1(X)$ in $G^\theta$,
where $\theta$ is in the clique $a$, and hence defining hyperk\"ahler subvarieties
of $\cM(G)$. The fixed points of $\iota(a,-)$ 
correspond to representations of $\pi_1(X)$ in $G^\sigma$,
where $\sigma$ is in the clique $a$, and define complex Lagrangian 
subvarieties of $\cM(G)$, with respect to the complex structure of $\cM(G)$
determined  by the complex structure of $X$ and the corresponding holomorphic
symplectic structure.

We can go a step further and consider the action of $H^1(X,Z)$,  where $Z$ is 
the centre of $G$, on $\cM(G)$,  defined by `tensoring' a principal $G$-bundle 
by an element in $H^1(X,Z)$ --- the principal bundle analogue of tensoring a 
vector bundle by a  line bundle. We then consider 
elements $(\alpha,a)\in H^1(X,Z)\rtimes\Out(G)$ of order $2$, where the
semidirect product is with respect to the natural action of $\Out(G)$ on
$H^1(X,Z)$, and define the involutions $\iota(a,\alpha,\pm)$ of $\cM(G)$
given by $(E,\varphi)\mapsto (a(E)\otimes\alpha,\pm a(\varphi))$. 
The fixed points of these involutions are now related to
what we call $(G_\theta,\pm)$-Higgs bundles. These are pairs $(E,\varphi)$
consisting of a $G_\theta$-bundle $E$, where $\theta\in \Aut_2(G)$ is in 
the clique of $a$ and
$$
G_\theta:=\{g\in G\;:\; \theta(g)=c(g)g,\;\mbox{with}\; c(g)\in Z\},
$$
and  $\varphi\in H^0(X,E(\lieg^\pm)\otimes K)$.

The $(G_\theta,+)$-bundles correspond with 
representations of $\pi_1(X)$ in $G_\theta$, while the 
$(G_\theta,-)$-bundles correspond with representations of
$\pi_1(X)$ in the subgroup  
 $$
G_\sigma:=\{g\in G\;:\; \sigma(g)=c(g)g,\;\mbox{with}\; c(g)\in Z\}.
$$
The groups $G_\theta$ and $G_\sigma$ are the normalizers of $G^\theta$ and
$G^\sigma$, respectively, in $G$, and are extensions of $G^\theta$ and
$G^\sigma$, respectively, by a finite group --- the same in both cases.
Again the fixed points subvarieties for $\iota(a,\alpha,+)$ are 
hyperk\"ahler submanifolds, while those for $\iota(a,\alpha,-)$ are Lagrangian 
submanifolds.

We can generalise our study to  automorphisms of $\cM(G)$ defined
by elements of $H^1(X,Z)\rtimes\Out(G)\times \C^*$ of arbitrary finite order. For this, we take an element $(\alpha,a)\in H^1(X,Z)\rtimes\Out(G)$ of order $n$
and $\zeta_k:=\exp(2\pi i\frac{ k}{n})$, with $0\leq k \leq n-1$, and consider 
the automorphism
$\iota(a,\alpha,\zeta_k)$ of $\cM(G)$
given by $(E,\varphi)\mapsto (a(E)\otimes\alpha,\zeta_k a(\varphi))$. 
The fixed points of this automorphism are described by
what we call $(G^\theta,\zeta_k)$-Higgs bundles  (resp. $(G_\theta,\zeta_k)$-Higgs bundles). These are pairs $(E,\varphi)$
consisting of a $G^\theta$-bundle (resp. $G_\theta$-bundle) $E$, where 
$\theta\in \Aut_n(G)$ is in the class defined by $a$ and, as above, 
$$
G_\theta:=\{g\in G\;:\; \theta(g)=c(g)g,\;\mbox{with}\; c(g)\in Z\},
$$
and  $\varphi\in H^0(X,E(\lieg^k)\otimes K)$, 
where $\lieg^k$ is the eigenspace with eigenvalue 
$\zeta_k$
of the  automorphism of $\lieg$ defined by $\theta$.
Except for the case $k=0$, which corresponds to ordinary 
$G^\theta$-Higgs bundles (resp. $G_\theta$-Higgs bundles) since
$\lieg^0$ is the Lie algebra of $G^\theta$ and  $G_\theta$, the other 
cases are not
generally related  to representations of the fundamental group. The fixed points
for $k=0$  define indeed hyperk\"ahler subvarieties of $\cM(G)$.  

The paper is organized as follows. In Section \ref{lie-theory} we give 
the necessary background on involutions and real forms of complex semisimple
Lie algebras
and Lie groups. Somethings are quite standard, but  we include other relevant 
results necessary for our Higgs bundle analysis that we have not found
in the literature. Recently we learned of the paper by J. Adams \cite{adams}
which relates closely to the  Galois cohomology approach to real
forms that we take in this paper. We extend some of the features 
to arbitrary finite order automorphisms of a complex semisimple Lie group
$G$. Since the arguments are basically the same as those for involutions,
we mostly avoid details on the proofs.
In Section \ref{automorphisms} we study finite order automorphisms of a
$G$-bundle and show how these give rise to reductions of structure group of the 
bundle. We consider a version of this, twisted by a finite order automorphism
of $G$ and a subgroup of $Z$, the centre of $G$.  These will play a central role in the 
study of our finite  order automorphisms of the moduli space 
of $G$-Higgs bundles.  

In Section \ref{higgs-bundles} we review the basics of $G$-Higgs bundle theory,
where $G$ is a complex semisimple Lie group, including the notions of stability
and the Hitchin--Kobayashi correspondence, relating to solutions to the Hitchin 
equations, whose  moduli space appears as
a hyperk\"ahler quotient. 
We then describe a natural group of automorphisms of $\cM(G)$,
which in particular will provide with the finite order
automorphisms that we study in this paper.
In Section  \ref{invo-higgs-bundles} we introduce  a class of 
$G$-Higgs bundles with an  extra structure determined by an element  
$\theta\in \Aut_n(G)$. These are the 
$(G^\theta,\zeta_k)$- and  
$(G_\theta,\zeta_k)$-Higgs bundles (which we  call
$(G^\theta,\pm)$- and  
$(G_\theta,\pm)$-Higgs bundles when $\theta\in  \Aut_2(G)$)
mentioned above.  These objects require appropriate stability 
conditions, defining moduli spaces
$\cM(G^\theta,\zeta_k)$ and   $\cM(G_\theta,\zeta_k)$,  
in terms of which we will describe  the fixed points 
of the finite order automorphisms studied later. 

In Section \ref{higher-order-auto} we undertake the study of  
the automorphisms considered  above 
on the moduli space of $G$-Higgs bundles.  
Building upon results in Sections      
\ref{lie-theory}, \ref{automorphisms}, \ref{higgs-bundles} 
and \ref{invo-higgs-bundles}, we describe the fixed point subvarieties.
Our main results are Theorems  \ref{moduli-n-automorphism}
and \ref{moduli-n-automorphism-alpha}.

In Section \ref{higgs-reps} we review first the basics on the moduli 
space $\calR(G)$  of reductive representations of $\pi_1(X)$ in $G$,
and its homeomorphism with the moduli space of $G$-Higgs bundles 
$\cM(G)$.  We then show similar correspondences for 
$\cM(G^\theta,\zeta_0)$ and $\cM(G_\theta,\zeta_0)$ with
$\calR(G^\theta)$ and  $\calR(G_\theta)$ respectively,
and when $\theta$ is of order $2$ for 
 $\cM(G^\theta,-)$ and   $\cM(G_\theta,-)$ with   
 $\calR(G^\sigma)$  and $\calR(G_\sigma)$ respectively.

In Section \ref{section-involutions-higgs-reps} 
we 
specialise the results of Section \ref{higher-order-auto} to the case
of involutions  on $\cM(G)$  and  translate these results to the moduli 
space $\calR(G)$ of
representations of the fundamental group of $X$ in $G$. Here we use the 
non-abelian Hodge theory correspondences given in Section \ref{higgs-reps}. 
Our main results are Theorems \ref{fixed-connected-a},
\ref{fixed-connected}, \ref{fixed-connected-a-alpha},
   \ref{fixed-a-reps},  \ref{fixed-a-lambda-reps}, 
\ref{hyperkahler-lagrangian}, and \ref{hyperkahler-n}. 
In particular, in  Theorem  \ref{hyperkahler-lagrangian}
we establish the hyperk\"ahler or Lagrangian property 
of the fixed point subvarieties according to the type of involution. 
These correspond respectively to $(B,B,B)$ and $(B,A,A)$ branes in the 
terminology used  in the study of Langlands duality and mirror
symmetry for Higgs bundles (\cite{kapustin-witten,hitchin2013,biswas-garcia-prada,baraglia-schaposnik}).

In Section \ref{sln} we ilustrate our main results in Section
\ref{section-involutions-higgs-reps}  in
the case $G=\SL(n,\C)$. It is worth pointing out that,
for $G=\SL(2,\C)$,  the description
of the fixed points involving elements of order $2$ in $H^1(X,Z)$,
 which in this case can be identified with $J_2(X)$ --- the $2$-torsion
elements in the Jacobian of $X$ ---, involves the Prym variety of the \'etale 
cover of $X$ defined by the element in $J_2(X)$. A detailed study of this 
case is carried out in \cite{garcia-prada-ramanan-rank2}. For $\SL(n,\C)$,
with even $n$, there is a similar phenomenon which involves now generalised
Prym varieties in the sense of Narasimhan--Ramanan \cite{narasimhan-ramanan}.
This case and the general Prym construction for arbitrary groups is being 
pursued  somewhere else.

We finish in Section \ref{triality} by applying our results  to the study of
involutions and order $3$ automorphisms of the moduli space of 
$\Spin(8,\C)$-Higgs bundles and exploring the role of triality.




The main results of this paper have been presented in several conferences
and workshops since 2006, and announced in 
\cite{garcia-prada-2007,garcia-prada,ramanan}. We apologize for having taken
so long to produce the final paper. 

\noindent{\bf Acknowledgements}. We wish to thank the following institutions
for hospitality and support:  MSRI, Berkeley; 
Newton Institute for Mathematical 
Sciences, Cambridge; Mathematical Institute, Oxford; NUS, Singapore;
IISc, Bangalore; CMI, Chennai; ICMAT, Madrid.

\section{Automorphisms of complex  Lie groups, involutions and real forms}
\label{lie-theory}

We recall some basic facts about  automorphisms and real forms of  
complex Lie algebras and  complex Lie groups 
(see \cite{helgason,onishchik,onishchik-vinberg}), and give some 
results needed for our analysis that we have not found in the literature.

\subsection{Automorphisms of Lie algebras and Lie groups}\label{automorphisms-g}

Let $\lieg$ be a complex Lie algebra. We consider the linear
algebraic group $\Aut(\lieg)$ of all automorphisms of $\lieg$.
The group $\Int (\lieg)$  of {\bf inner automorphisms} of $\lieg$
is the normal subgroup of $\Aut(\lieg)$ generated by all elements
of the form $\exp(\ad x)$, where $\ad$ is the adjoint representation
of $\lieg$ and $x\in \lieg$. The quotient  
$$
\Out(\lieg)=\Aut(\lieg)/\Int (\lieg) 
$$ 
is called the group of {\bf outer automorphisms} of $\lieg$.
We thus have an extension 
 \begin{equation}\label{outer-extension-lie-algebra}
1 \lra \Int(\lieg)\lra \Aut(\lieg) \lra \Out(\lieg) \lra 1.
\end{equation}

Let  $G$ be  a connected complex Lie group with Lie algebra $\lieg$. Through all
the paper the centre $Z(G)$ of $G$ will be denoted by $Z$. 
Then $\Int(\lieg)=\Ad(G)$, where $\Ad$ is the adjoint representation of $G$
in $\lieg$. Recall that $\Ad(G)\cong G/Z$.

If $\lieg$ is semisimple,  $\Out(\lieg)$ is isomorphic to the group of
automorphisms of its Dynkin diagram. 
From this one may list the groups
$\Out(\mathfrak{g})$ for the simple complex Lie algebras
in Table \ref{outer-groups}.

Let $G$ be a complex  Lie group. Let $\Aut(G)$ be the group of
holomorphic automorphisms of $G$  and   
$\Int(G)$ be the normal subgroup of $\Aut(G)$ given by 
{\bf inner automorphisms}.
The elements of $\Int(G)$ will be denoted by $\Int(g)$ for $g\in G$ and the
action on $G$ is given by 
$$
\Int(g)(h):= ghg^{-1}\;\;\; \mbox{for every}\;\;\; h\in G.
$$

Then if $G$ is connected $\Int(G)\cong\Ad(G)$.  Let $\Out(G):=\Aut(G)/\Int(G)$ be the 
group of {\bf outer automorphisms} of $G$.
It is well-known  that  if $G$ is a connected complex reductive Lie group 
the extension 
\begin{equation}\label{outer-extension-group}
1 \lra \Int(G)\lra \Aut(G) \lra \Out(G)\lra 1.
\end{equation}
splits (see \cite{de-siebenthal}).

Let $\tilde{G}$ be the universal cover of $G$. We clearly have 
$$
\Int(G)\cong \Int(\tilde{G})\cong\Int(\Ad(G))\cong \Ad(G)\cong \Int(\lieg).
$$

We thus observe that, if $\pi_1(G)=\{1\}$ or $Z=\{1\}$, extension 
(\ref{outer-extension-group}) is isomorphic to
(\ref{outer-extension-lie-algebra}). 

\begin{table}[htbp]
\centering
\begin{tabular}{|c|c|c|}
\hline $\mathfrak{g}$ & $\Out(\mathfrak{g})$ & $Z(\tilde{G})$\\
\hline $A_n$, $n>1$ & $\mathbb{Z}/2$ & $\Z/(n+1)$\\
\hline $A_1$ & $\{1\}$ & $\Z/2$\\
\hline $B_n$ & $\{1\}$ & $\Z/2$\\
\hline $C_n$ & $\{1\}$ & $\Z/2$\\
\hline $D_4$ & $S_3$ & $\Z/2\times \Z/2$\\
\hline $D_n$, $n>4$, $n$ even &  $\mathbb{Z}/2$ & $\Z/2\times \Z/2$\\
\hline $D_n$, $n>4$, $n$ odd &  $\mathbb{Z}/2$ & $\Z/4$\\
\hline $E_6$ & $\mathbb{Z}/2$ & $\{1\}$\\
\hline $E_7$ & $\{1\}$ & $\{1\}$\\
\hline $E_8$ & $\{1\}$ & $\{1\}$\\
\hline $F_4$ & $\{1\}$ & $\{1\}$\\
\hline $G_2$ & $\{1\}$ & $\{1\}$\\
\hline
\end{tabular}
\vspace{12pt}
\caption{Group of outer automorphisms and centres} 
\label{outer-groups}

\end{table}


Observe that $\Aut(G)$ acts on $Z$, defining an action 
of $\Out(G)$ since  $\Int(G)$ acts trivially on $Z$. 
Hence in order to compute  $\Out(G)$ for any semisimple complex Lie group
$G$ out of the Table \ref{outer-groups}, we need to compute how 
$\Out(\tilde{G})$ acts on  $Z(\tilde{G})$. These are listed when $G$ is simple
 in Table  \ref{outer-groups} (see \cite{goto-kobayashi}).  If a semisimple 
group $G$ is isomorphic to $\tilde{G}/Z'$, where 
$Z'\subset Z(\tilde{G})$ is a subgroup, where 
$\tilde{G}$ is the universal covering of $G$, we have 
$\pi_1(G)=Z(\tilde{G})/Z'$. Let $f\in \Aut(\tilde{G})$. Then, clearly
$f$ descends to give an element in $\Aut(G)$ if and only if 
$f(Z')\subset Z'$. Similarly, $\Out(\tilde{G})$ acts on $Z(\tilde{G})$
and $a\in \Out(\tilde{G})$ descends to an element in $\Out(G)$ if 
$Z'$ is invariant under $a$.


\subsection{Involutions and real forms of a complex Lie algebra}
\label{realforms-algebra}

Let $\lieg$ be a complex Lie algebra and  $\lieg_\R$  its
underlying  real Lie algebra. A real subalgebra $\lieg_0$ of
$\lieg_\R$ is called a {\bf real form} of $\lieg$ if 
$\lieg_\R=\lieg_0 \oplus i\lieg_0$. In this case $\lieg_0\otimes\C$
may be  naturally identified with $\lieg$. Given a real form $\lieg_0$,
there is a bijection of 
$\lieg=\lieg_0\otimes \C$ onto itself defined by
$$
\sigma(x+iy)=x-iy,\;\;\; x,y\in \lieg_0,
$$
which is a {\bf conjugation or antiinvolution} of $\lieg$, i.e. an
antilinear homomorphism $\sigma$
such that $\sigma^2=\Id$. 
Conversely, any conjugation $\sigma:\lieg \lra \lieg$ defines the real
form 
$$
\lieg^\sigma:=\{z\in \lieg\;\;:\;\; \sigma(z)=z\}.
$$
 There is
thus a bijective correspondence between real forms and conjugations
of $\lieg$.

To  classify real forms of a complex Lie
algebra $\lieg$ up to
isomorphism, one first observes that if 
$\lieg_0$ and  $\lieg_0'$ are  two real forms of 
$\lieg$  and $f: \lieg_0 \to \lieg_0'$ is an isomorphism,
then $f$ extends uniquely to an automorphism $\alpha:=f^\C$ of $\lieg$.
Now, if  $\sigma$ and $\sigma'$ are  the
corresponding conjugations for $\lieg_0$ and $\lieg_0'$ respectively,
clearly the conjugations $\alpha\sigma_0$ and  $\sigma_0'\alpha$ 
coincide in $\lieg_0$. Therefore they coincide in $\lieg$ and
hence  $\sigma'=\alpha\sigma\alpha^{-1}$.  
Conversely, suppose that there exists $\alpha\in \Aut(\lieg)$ such that 
$\sigma'=\alpha\sigma\alpha^{-1}$.  It is immediate that 
$\alpha(\lieg^{\sigma})=\lieg^{\sigma'}$, i.e. $\alpha(\lieg_0)=\lieg_0'$.

In other words, if $\conj(\lieg)$ is the
set of conjugations of $\lieg$, 
the set of isomorphism classes of real forms of $\lieg$
is in bijection with

\begin{equation}\label{cartan-classes}
\conj(\lieg)/\sim_c,
\end{equation}
where the  equivalence
relation $\sim_c$ for $\sigma,\sigma'\in \conj(\lieg)$  is defined by 
$$
\sigma \sim_c \sigma' \;\;\mbox{if there is}\; \alpha \in
\Aut(\lieg)\;\; \mbox{such that}\;\;  \sigma'=\alpha\sigma \alpha^{-1}.
$$

In \cite{cartan} \'E. Cartan proved that for a semisimple complex Lie
algebra $\lieg$ in the statement above one can replace  conjugations 
(antiinvolutions) of $\lieg$ by $\C$-linear  {\bf involutions}, i.e. 
involutive  automorphisms  of $\lieg$. This is based on the existence of a 
compact  real form of $\lieg$.

Let $\lieg$ be a complex Lie algebra
and  $\tau$ a fixed conjugation.
We have a map
\begin{equation}\label{conj-inv}
   \begin{aligned}
\conj(\lieg) & \to \Aut(\lieg)\\ 
\sigma & \mapsto \theta:=\sigma\tau.
  \end{aligned}
\end{equation}

This correspondence depends on the choice of $\tau$.
Since $\sigma^2=\Id$, then  $\theta^{-1}=\tau\theta\tau$ and
$\theta^2=\Id$ if and only if $\sigma\tau=\tau\sigma$. Also if
$\theta$ corresponds to $\sigma$ and $\alpha\in \Aut(\lieg)$, then
the automorphism $\theta'$ corresponding to $\alpha\sigma\alpha^{-1}$
has the form $\theta'=\alpha\theta(\tau\alpha\tau)^{-1}$. 

We can rephrase these properties in terms of the Galois cohomology
of $\Z/2$ in the group  $\Aut(\lieg)$. To explain this, we recall first
the basic definition of {\bf non-abelian cohomology} 
(see \cite[Ch. III]{serre}). Let $\Gamma$ be a group and  $A$ 
another group acted on by  $\Gamma$, that is, 
every  $\gamma\in \Gamma$ defines an automorphism of $A$ that
we will denote also by $\gamma$, such that 
$\gamma(xy)=\gamma(x)\gamma(y)$, 
for every  $x,y\in A$.

We will define $H^0$ and $H^1$. We set  $H^0(\Gamma,A):=A^\Gamma$, the subgroup
of elements of $A$ fixed under $\Gamma$. 
To define $H^1$ we first define a $1$-{\bf cocycle} (or simply {\bf cocycle})
of $\Gamma$ in $A$ 
as a map $\gamma\mapsto a_\gamma$ of $\Gamma$ to $A$ such that
\begin{equation}\label{cocycle}
a_{\gamma\gamma'}=a_\gamma\gamma(a_{\gamma'})\;\; \mbox{for}\;\;
\gamma,\gamma'\in \Gamma.
\end{equation}
The set of cocycles is denoted by $Z^1(\Gamma,A)$.
Two cocycles $a,a'\in Z^1(\Gamma,A)$ are  said to be  {\bf cohomologous}
if there is $b\in A$ such that
\begin{equation}\label{cohomologous}
a'_\gamma=b^{-1}a_\gamma \gamma(b).
\end{equation}
 This is an equivalence relation in $Z^1(\Gamma,A)$ and the quotient is 
denoted by $H^1(\Gamma,A)$. This is the {\bf first cohomology set of
$\Gamma$ in $A$}. It has a distinguished element
(called the "neutral element") even though there is in general no composition: 
the class of the unit cocycle.
If $A$ is commutative $H^0(\Gamma,A)$ and  $H^1(\Gamma,A)$ are  the usual 
cohomology groups of dimensions $0$ and $1$.

We will be dealing with the special case in which 
$\Gamma=\{1,\gamma\}\cong \Z/2$, where  $\gamma$ is the non-trivial element in 
$\Gamma$.  In this situation a $1$-cocyle $a$ is basically given by an element
of $A$, say $s:=a_\gamma$, satisfying that 
$$
s\gamma(s)=1.
$$

Let $a'$ be another $1$-cocycle, and let $s':=a'_\gamma$. 
The cocycles $a$ and $a'$ are then cohomologous 
if there exists  $g\in A$ satisfying 
$$
s'=g^{-1}s\gamma(g).
$$

Now, the correspondence  (\ref{conj-inv})
establishes a bijection between $\conj(\lieg)/\sim_c$ and the 
cohomology set $H^1(\Z/2,\Aut(\lieg))$, where here $\Z/2$ is the Galois 
group of the field extension $\R\subset \C$ acting on the group $\Aut(\lieg)$
by the rule $\theta\mapsto \tau\theta\tau$.

If $\lieg$ is semisimple, a compact real form always exists
(see \cite[Ch. III]{helgason},  or \cite[Sec. 5.1]{onishchik-vinberg})
and we choose the corresponding conjugation $\tau$  as the fixed
conjugation in the correspondence given above. Cartan proved that 
given a conjugation $\sigma$ of $\lieg$ one can choose  a conjugation
$\sigma'=\alpha\sigma\alpha^{-1}$ where  $\alpha\in\Int(\lieg)$, 
such that $\sigma'\tau=\tau\sigma'$.
Hence $\theta=\sigma'\tau$ is an
involution (or equivalently, any Galois cohomology class contains an
invariant cocycle).

We  define the following  equivalence relations for  
$\theta,\theta'\in\Aut(\lieg)$:

$$
\theta \sim \theta'  \;\;\mbox{if there is}\;\; \alpha \in
\Int(\lieg)\;\; \mbox{such that}\;\;  \theta'=\alpha\theta
\alpha^{-1},
$$

$$
\theta \sim_c \theta'  \;\;\mbox{if there is}\; \; \alpha \in
\Aut(\lieg)\;\; \mbox{such that}\;\;  \theta'=\alpha\theta
\alpha^{-1}.
$$

In particular, these define equivalence relations in $\Aut_2(\lieg)$,  the set of involutions of $\lieg$, that
is, elements of order 2 in $\Aut (\lieg)$. Similarly, we define $\sim$ in $\conj(\lieg)$
(we have already defined $\sim_c$ above). 

We thus have the following.

\begin{proposition}
The map $\sigma\mapsto \theta:=\sigma\tau$ gives bijections

$$
\conj(\lieg)/\sim \longleftrightarrow \Aut_2(\lieg) /\sim,
$$
and
$$
\conj(\lieg)/\sim_c \longleftrightarrow \Aut_2(\lieg) /\sim_c.
$$
\end{proposition}

In particular,  the set of isomorphism classes of real
forms of $\lieg$ is in one to one correspondence with  
$\Aut_2(\lieg) /\sim_c$.

There is yet another equivalence relation for elements $\theta,\theta'\in
\Aut_2(\lieg)$: we say that $\theta$ and $\theta'$ are {\bf inner equivalent}
if their images in $\Out(\lieg)$ coincide, i.e. if there exists 
$\alpha\in \Int(\lieg)$ such that $\theta'=\alpha\theta$.

\subsection{Involutions and real forms of a complex Lie group}
\label{realforms-group}

Let $G$ be a complex Lie group,
and let $G_\R$ be the underlying
real Lie group. We will say that a  real Lie subgroup $G_0\subset G_\R$ 
is a {\bf real form} of $G$ if  $G_0=G^\sigma$, the fixed point set of a 
{\bf conjugation} (antiholomorphic involution)  $\sigma$ of $G$. 

Now, let $G$ be semisimple. 
A compact real form $U$ of $G$ always exists. 
From this we can define
a conjugation $\tau:G\to G$ such that $G^\tau=U$. This can be seen as 
follows: The conjugation defining  
the compact form of $\lieg$ can be lifted to a conjugation 
$\tilde \tau$  of the universal cover $\tilde G$. Let 
$\tilde U={\tilde G}^{\tilde\tau}$. Since  $Z(\tilde G)$ is finite, we have 
$Z({\tilde G})\subset \tilde U$, and hence $\tilde \tau$ acts
trivially on  $Z(\tilde G)$ descending to a conjugation $\tau$ of
$G$.

We can define for Lie groups the equivalence relations $\sim_c$ and $\sim$ 
as in the 
case of Lie algebras. 

From  Section \ref{realforms-algebra}, we have  
bijections

$$
\conj(\tilde G)/\sim \longleftrightarrow \Aut_2(\tilde G) /\sim,
$$ 
and
$$
\conj(\tilde G)/\sim_c \longleftrightarrow \Aut_2(\tilde G) /\sim_c,
$$ 
where  $\Aut_2(\tilde{G})$ is the set of elements of order $2$ 
in $\Aut(\tilde{G})$.
One checks that the action on $Z(\tilde G)$ send an element to itself or to the
inverse on both sides to conclude the following.


\begin{proposition}\label{conjugations-versus-involutions}
 
The map $\sigma\mapsto \theta:=\sigma\tau$ gives
bijections  
$$
\conj(G)/\sim \longleftrightarrow \Aut_2(G) /\sim,
$$ 
and
$$
\conj(G)/\sim_c \longleftrightarrow \Aut_2(G) /\sim_c,
$$ 
where  
$\Aut_2(G)$ is the set of elements of order $2$ in $\Aut(G)$.
\end{proposition}

The following is immediate.
\begin{proposition}\label{int-subgroups}
(1) Let $\sigma\sim\sigma'\in \conj(G)$, say  
$\sigma'=\Int(g)\sigma\Int(g)^{-1}$ for some $g\in G$, 
then $G^{\sigma'}=\Int(g)G^{\sigma}$, where $G^\sigma$ and 
$G^{\sigma'}$ are the subgroups of fixed points of $\sigma$
and $\sigma'$ respectively. 

(2) Similarly, let $\theta,\theta'\in \Aut_2(G)$ such
 $\theta'=\Int(g)\theta\Int(g)^{-1}$ for some $g\in G$, 
then $G^{\theta'}=\Int(g)G^{\theta}$, where $G^\theta$ and 
$G^{\theta'}$ are the subgroups of fixed points of $\theta$
and $\theta'$ respectively. 
\end{proposition}



If $G$ is semisimple then the set of  isomorphic  real forms of $G$ 
is in bijection with $\Aut_2(G)/\sim_c$. In particular,  the set
of equivalence classes of real forms that are inner is given by
$\Int_2(G)$, the set of  elements of order $2$ in $\Int(G)$, modulo
conjugation by elements in $\Aut(G)$.

As for Lie algebras, the correspondence given in 
Proposition \ref{conjugations-versus-involutions} has again a Galois
cohomology interpretation in the sense that $\conj(G)/\sim_c$ is in
bijection with the Galois cohomology set
$H^1(\Z/2,\Aut(G))$, where  $\Z/2$ is the Galois 
group of the field extension $\R\subset \C$ acting on the group $\Aut(G)$
by the rule $\theta\mapsto \tau\theta\tau$.

We have the following (see \cite[Sec. 3.4]{onishchik-vinberg}).

\begin{proposition}
If $G$ is a connected and simply connected Lie group and 
$s\in\Aut (\lieg) (=\Aut (G))$ is a semisimple automorphism,
then the group $G^s$ is connected.
\end{proposition}

In particular any real form of a connected and simply connected Lie
group $G$ is connected.


\begin{remark}\label{spin-real-forms}
Note that for a  complex Lie group $G$ there may exist 
a  real Lie subgroup $G_0\subset G_\R$ whose  corresponding 
Lie algebra $\lieg_0$ is a real form
of $\lieg$, without $G_0$ being a real form of $G$ 
according to our definition. Examples of this are:

(1) $G=\SL(2,\C)$, $G_0=\NSL(2,\R)$, the normalizer of $\SL(2,\R)$ in
$\SL(2,\C)$;  

(2) $G=\Spin(n,\C)$, $G_0=\Spin(p,q)$ (double cover of $\SO(p,q)$).


The group $G_0$ in (1) and (2) is a  lift to $G$ of a  real form in $\Ad(G)$.
\end{remark}


The set of isomorphism classes of real
forms of $G$ is in bijection with  
$\Aut_2(G) /\sim_c$. Let $[\theta]$ and $[\theta]_c$ 
denote, respectively, the images of $\theta$ 
under the natural maps $\Aut(G)\to \Aut(G) /\sim$ and
$\Aut(G)\to \Aut(G) /\sim_c$, respectively. 
Consider the natural projection $\pi:\Aut(G)\to \Out(G)$. 
Let $\theta,\theta'\in \Aut(G)$. Like in the case of Lie algebras, 
we say that $\theta$ and $\theta'$
are {\bf inner equivalent} if there is $\alpha\in \Int(G)$ such that
$\theta'=\alpha\theta$, or equivalently $\pi(\theta)=\pi(\theta')$. 

\begin{proposition}\label{cartan-versus-inner-groups}
Let $G$ be a semisimple complex Lie group.

(1)  The projection $\pi:\Aut(G) \to \Out(G)$ descends to define a  
surjective map 
$$
\Aut(G)/\sim \to \Out(G),
$$
in  particular, this defines a surjective map
$$
\cl: \Aut_2(G)/\sim \to \Out_2(G),
$$
where $\Out_2(G)$ is the set of elements of order $2$ in $\Out(G)$
(this is not satisfied  if we replace $\sim$ by $\sim_c$).

(2) The action of $\Aut(G)$ by inner automorphisms  on $\Aut(G)$ (respectively on 
$\Aut_2(G)$) induces an action of $\Out(G)$  on $\Aut(G) /\sim$  
(respectively $\Aut_2(G) /\sim$) and the quotient is given by 
$\Aut(G)/\sim_c$  (respectively $\Aut_2(G)/\sim_c$). 
The map $\cl$ is $\Out(G)$-equivariant. In particular the 
set $\Aut_2(G) /\sim$ is finite.



\end{proposition}

\begin{proof}

To prove (1), let $\theta,\theta'\in \Aut(G)$ such that
$[\theta]=[\theta']$. This means that $\theta'=\Int(g)\theta\Int(g^{-1})$
for some $g\in G$. One easily checks that
$\theta'=\Int(\tilde{g})\theta$, where $\tilde{g}=g\theta(g^{-1})$,
thus having $\pi(\theta)=\pi(\theta')$, i.e., $\theta$ and $\theta'$ are inner
equivalent. If $\theta$ and $\theta'$ are of order two, of course 
$\pi(\theta)$ and $\pi(\theta')$ are elements in $\Out_2(G)$.
It is worth pointing out that, since the extension
(\ref{outer-extension-group}) splits,
the elements in $\Out_2(G)$ can be lifted to $\Aut_2(G)$.

As for (2), let $\alpha,\theta\in \Aut(G)$. We define the action
$$
\alpha\cdot\theta=\alpha\theta\alpha^{-1}.
$$
Let $\alpha'=\Int(g)\alpha$. Then

\begin{align*}
\alpha'\cdot\theta & =  (\Int(g)\alpha)\cdot\theta \\ 
                   & =  (\Int(g)\alpha)\theta(\Int(g)\alpha)^{-1}\\
                   & =  \Int(g)\alpha\cdot\theta\Int(g)^{-1},
\end{align*}
and hence $[\alpha\cdot\theta]=[\alpha'\cdot\theta]$.
                      
Let $a\in\Out(G)$ and $\alpha\in \pi^{-1}(a)\subset  \Aut(G)$. 
We can thus define the action $a\cdot[\theta]:=[\alpha\cdot\theta]$. 
Of course 
$[\theta]_c=[\alpha\cdot\theta]_c$, proving the claim.
The statement when $\theta\in \Aut_2(G)$ is clear.

Now, Cartan in his classification of real forms for a simple group $G$
shows that $\Aut_2(G)/\sim_c$ is finite. This together with the
finiteness of $\Out_2(G)$ (see Table \ref{outer-groups}) implies that the set
$\Aut_2(G)/\sim$ is finite. 

\end{proof}

Consider the map $\cl: \Aut_2(G)/\sim \to \Out_2(G)$ 
defined in (1) of
Proposition \ref{cartan-versus-inner-groups}.
We will call the  image of $[\theta]\in \Aut_2(G)/\sim$  of this map  
the {\bf clique} of $[\theta]$.
Clearly $\cl^{-1}(1)=\Int_2(G)/\sim$, i.e. the $\Int(G)$-conjugacy classes
of inner involutions have trivial clique. 
Now, if we fix a conjugation
$\tau\in\conj(G)$ defining a compact real form of $G$, 
as mentioned in Proposition \ref{conjugations-versus-involutions},
in each class $[\theta]\in \Int_2(G)/\sim$, we can find a representative
$\theta=\Int(g)$ in $[\theta]$ for some $g\in G$ such that 
$\theta\tau=\tau\theta$ and hence 
$\sigma:=\theta\tau$ defines a conjugation inner equivalent to $\tau$,
in the sense that $\sigma=\Int(g)\tau$.  Real forms $\sigma$ of $G$  
inner equivalent to  $\tau$ are called {\bf real forms of Hodge type} 
and the corresponding real groups $G^\sigma$ are called 
{\bf groups of Hodge type}.

Combining the bijection 
$$
\conj(G)/\sim \longleftrightarrow \Aut_2(G) /\sim,
$$ 
given by Proposition \ref{conjugations-versus-involutions}, with the map
$\cl: \Aut_2(G)/\sim \to \Out_2(G)$, we obtain a map

\begin{equation}\label{clique-conjugations}
\widehat{\cl}: \conj(G)/\sim \to \Out_2(G).
\end{equation}

Of course $\widehat{\cl}^{-1}(1)$ consists of the equivalence classes of real
forms of Hodge type.

We will assume for the rest of this section 
$G$ to be  connected.  Since  $\Int(G)=\Int(\tilde{G})=\Int(\lieg)$, 
we have  in particular
$\Int_2(G)=\Int_2(\tilde{G})=\Int_2(\lieg)$, and hence every 
real form of Hodge type on $\tilde{G}$ descends to a real form of 
Hodge type on $G$. Of course, any real form on $\tilde{G}$ defines
a real form on $\Ad(\tilde{G})$ since the $Z(\tilde{G})$ is invariant
under any element in $\conj(\tilde{G})$. If $G=\tilde{G}/Z'$ for a subgroup 
$Z'\subset Z(\tilde{G})$, and $\sigma\in \conj(\tilde{G})$, 
the condition for $\sigma$ to define a conjugation on $G$, and
hence a real form on $G$,  is that $Z'$ be invariant under $\sigma$.

\begin{proposition}\label{cliques}
Let $\theta\in \Aut_2(G)$. Consider the set 
$$
S_\theta:=\{s\in G\;:\; s\theta(s)=z\in Z\}.
$$
Then

(1)  $Z$  acts on $S_\theta$ by multiplication. 

(2) $G$ acts on the right on $S_\theta$ by
$$
s\cdot g:=g^{-1}s \theta(g)\;\; g\in G, s\in S_\theta.
$$

(3) Let $\pi:\Aut_2(G)\to \Out_2(G)$ be the natural projection. 
Let $a\in \Out_2(G)$, and
let $\theta\in \pi^{-1}(a)$. Then
the map $\psi: S_\theta\to \pi^{-1}(a)$ defined by $s\mapsto \Int(s)\theta$
gives  a bijection 
$$
S_\theta/(Z\times G)\longleftrightarrow \cl^{-1}(a),
$$ 
where $\cl:\Aut_2(G)/\sim\to \Out_2(G)$ is the map induced by $\pi$.

(4) In particular, let $\theta=\Id\in \Aut_2(G)$. Then
$$
S_{\Id}=\{s\in G\;:\; s^2=z\in Z\},
$$
and the map $s\mapsto \Int(s)$ defines a bijection 
$$
S_{\Id}/(Z\times G)\longleftrightarrow \Int_2(G)/\sim.
$$ 
\end{proposition}

\begin{proof}
(1) Let $\theta\in \Aut_2(G)$ and let $s\in S_\theta$. Let $z\in Z$.
Since $\theta$ leaves $Z$ invariant, we have that 
$zs\theta(zs)=zs\theta(z)\theta(s)
=z\theta(z)s\theta(s)\in Z$ and hence $zs\in S_\theta$.

(2) Let $g\in G$ and $s\in S_\theta$. Let  
$s\cdot g:=g^{-1}s \theta(g)$.
We have 
$(s\cdot  g)\theta(s\cdot g)= g^{-1}s\theta(g)\theta(g^{-1})\theta(s)g=
g^{-1}s\theta(s)g=s\theta(s)$ and  hence 
$s\cdot g \in S_\theta$.

(3) One first checks that for  $g\in G$, the element $\Int(g)\theta$ is 
in $\Aut_2(G)$ if and only if $g\in S_\theta$. This is clear since, as one 
can easily compute, $(\Int(g)\theta)^2=1$ is equivalent to 
$g\theta(g)x\theta(g^{-1})g^{-1}=x$ for every $x\in G$, and hence
$g\theta(g)\in Z$. We thus have that $s\mapsto \Int(s)\theta$ 
defines a surjective map  $\psi: S_\theta\to \pi^{-1}(a)$.
Obviously, $\psi$ descendes to a map on $S_\theta/Z$ since
$\Int(s)=\Int(zs)$ for every $z\in Z$.

Let $s':=s\cdot g=g^{-1}s\theta(g)$, where, as above, $s\in S_\theta$ and
$g\in G$. Then $\psi(s')=\Int(s')\theta=\Int(g^{-1}s\theta(g))\theta=
\Int(g^{-1})\Int(s)\Int(\theta(g))\theta$.
But since $\Int(\theta(g))\theta=\theta\Int(g)$, we thus have
 $\psi(s')=\Int(g^{-1})\psi(s)\Int(g)$,
and hence  $\psi(s)\sim\psi(s')$.  

(4) follows from (3).

\end{proof}

Proposition \ref{cliques} gives indeed an interpretation in terms of
non-abelian cohomology (see Section \ref{realforms-algebra} for the 
basic definitions) of 
$\cl^{-1}(a)$, that is of the set of elements in $\Aut_2(G)/\sim$ with 
clique  $a$, in particular of the set $\Int_2(G)/\sim$, corresponding to
the trivial clique. We have the following.

\begin{proposition}\label{cohomology}
Let $H^1_\theta(\Z/2,\Ad(G))$ be the cohomology set defined by the action
of $\Z/2$ in $\Ad(G)$ given by  $\theta\in \Aut_2(G)$. Then, there is a 
bijection
$$
H^1_\theta(\Z/2,\Ad(G)) \longleftrightarrow S_\theta/(Z\times G).
$$ 
\end{proposition}
\begin{proof}
Consider the group $\Z/2$  generated by $\theta$. Consider now the action of $\Z/2$ on $\Ad(G)$ given by the
action of $\theta$ (we are denoting the automorphism of $\Ad(G)$ defined by
$\theta$ also by $\theta$). Let $Z^1_\theta(\Z/2,\Ad(G))$ be the set of 
cocycles of $\Z/2$ in $\Ad(G)$ given by this action. 
Let $s\in S_\theta$ and let $\tilde{s}$ the image of $s$ in 
$\Ad(G)$. The correspondence $s\mapsto \tilde{s}$ 
defines a bijection  $S_\theta/Z \to Z^1_\theta(\Z/2,\Ad(G))$, where here
we are identifying a cocycle $a$ with the corresponding element 
$a_\theta\in \Ad(G)$. The action of $G$ on $S_\theta/Z$ is via the action 
of $\Ad(G)$, in other 
words, $S_\theta/(Z\times G)= (S_\theta/(Z)/\Ad(G)$, where 
$\tilde{g}\in \Ad(G)$ acts on $\tilde{s} \in S_\theta/(Z$
 by the rule
$$
\tilde{s}\cdot \tilde{g}=\tilde{g}^{-1}\tilde{s}\theta(\tilde{g}).
$$
But, this is precisely the condition given in \ref{cohomologous} for 
the cocyles $a$ and $a'$ corresponding to $\tilde{s}$ and $\tilde{s'}$, 
respectively, to be cohomologous.
\end{proof}




In view of  Propositions \ref{cliques} and \ref{cohomology},
it is clear that if $\theta,\theta'\in \Aut_2(G)$ are such that
$\pi(\theta)=\pi(\theta')=a$ then  there is a bijection 
between  $H^1_\theta(\Z/2,\Ad(G))$ and
$H^1_{\theta'}(\Z/2,\Ad(G))$. In this sense this cohomology set 
could very well be denoted by $H^1_a(\Z/2,\Ad(G))$. With this notation,
we have shown the following.

\begin{proposition}\label{clique-cohomology}
Let $a\in \Out_2(G)$. There is a bijection
$$
\cl^{-1}(a) \longleftrightarrow H^1_a(\Z/2,\Ad(G)),
$$
and hence a bijection
$$
\Aut_2(G)/\sim \longleftrightarrow \bigcup_{a\in \Out_2(G)}
H^1_a(\Z/2,\Ad(G)).
$$ 
\end{proposition}




So, while the  cohomology set $H^1(\Z/2,\Aut(G))$ describes the set 
of conjugations of $G$, modulo the equivalence
given by $\sim_c$,   $H^1_\theta(\Z/2,\Ad(G))$ is in bijection with 
the set of equivalence classes of conjugations of $G$ given by the 
relation $\sim$ (note that $\Ad(G)\cong \Int(G))$. For similar 
discussion on cohomology and real forms look at the recent paper by 
J. Adams \cite{adams}.


\subsection{Normalizers, real forms and the isotropy representation}
\label{normalizers}

Let $G$ be a connected complex semisimple Lie group and, as above, denote the centre of $G$ by $Z$.
Let $\sigma\in \conj(G)$. We have  
considered the real form defined by $\sigma$, that is a real subgroup of $G$ defined by
$$
G^\sigma:=\{g\in G\;:\; \sigma(g)=g\}.
$$

We shall  consider now the  real subgroup of $G$ defined by  
$$
G_\sigma:=\{g\in G\;:\; \sigma(g)=c(g)g,\;\mbox{with}\; c(g)\in Z\}.
$$

\begin{proposition}\label{normalizer-sigma}
The group $G_\sigma$ is $N_G(G^\sigma)$, the normalizer of
$G^\sigma$ in $G$. 
\end{proposition}

\begin{proof}
It is immediate that $G_\sigma\subset N_G(G^\sigma)$. 

To show the converse, let $x\in N_G(G^\sigma)$. 
In order for $x$ to be in $G_\sigma$, 
we need to have that $x^{-1}\sigma(x)\in Z$. For this, it is enough to show 
that $x^{-1}\sigma(x)$ commutes with every element of $G^\sigma$, since the
centralizer of $x^{-1}\sigma(x)$  is a complex algebraic group, and if it
contains $G^\sigma$ it must be the whole $G$. Indeed, if $g\in G^\sigma$,
then $xgx^{-1}\in G^\sigma$ since $x\in N_G(G^\sigma)$, and hence  
$\sigma(xgx^{-1})=xgx^{-1}$. But $\sigma(xgx^{-1})=\sigma(x)g\sigma(x^{-1})$. 
These two equalities combined imply that $x^{-1}\sigma(x)$ commutes with $g$ as we wanted to
show. 
\end{proof}

The conjugation $\sigma$ leaves $Z$ invariant and hence descends to a 
conjugation of $\Ad(G)=G/Z$. The following is immediate.

\begin{proposition}\label{lift-g-sigma}
$\Ad(G)^\sigma =\Ad(G_\sigma):=G_\sigma/Z(G_\sigma)$,
where we denote also by $\sigma$ the conjugation defined on $\Ad(G)$. In other
words, $G_\sigma$ is a lift to $G$ of the real form in $\Ad(G)$ defined by 
$\sigma$.
\end{proposition}

Define $\Gamma_\sigma=G_\sigma/G^\sigma$.
We have the exact sequence
\begin{equation}\label{sigma-exact-sequence}
1 \lra G^\sigma\lra G_\sigma \lra \Gamma_\sigma \lra 1.
\end{equation}


Similarly, if $\theta\in \Aut_2(G)$ we can define the subgroups

$$
G^\theta:=\{g\in G\;:\; \theta(g)=g\},
$$

$$
G_\theta:=\{g\in G\;:\; \theta(g)=c(g)g,\;\mbox{with}\; c(g)\in Z\}.
$$

As in the case of conjugations, an element $\theta\in \Aut_2(G)$  leaves $Z$
invariant  and hence descends to an automorphism of order $2$ of $\Ad(G)$.
We have the following. 

\begin{proposition}\label{lift-g-theta}
$\Ad(G)^\theta =\Ad(G_\theta):=G_\theta/Z(G_\theta)$,
where we denote also by $\theta$ the involution defined on $\Ad(G)$. 
\end{proposition}

It is clear that $Z\subset G_\theta$. 
For any $g\in G$ we have $\theta(g)=\theta(g)g^{-1}g$. But if $g\in Z$, 
$\theta(g)$ is also 
in the centre and so is $\theta(g)g^{-1}$. Hence 
$g\in G_\theta$ with $c(g)=\theta(g)g^{-1}$.

\begin{remark}
We thus have that $G_{\theta}$ contains $G^{\theta}Z$, but may be larger.
Take $\theta$ corresponding to the conjugation defining $\SL(2,\R)$.
In this case the group $G_\theta$ contains the diagonal matrix with 
entries $(i,-i)$, and is in fact the normalizer of $\SO(2,\C)$ in 
$\SL(2,\C)$.
\end{remark}

\begin{remark}
Note that if $\theta$ is an inner involution 
i.e. $\theta=\Int(g)$ for
some $g\in G$ (real forms of Hodge type), then $Z\subset G^\theta$.
\end{remark}

\begin{remark} 
Even though $G$ is assumed to be connected, $G^\theta$ and
$G_\theta$ maybe, of course, non-connected, and hence the kernel
of the adjoint representations of $G^\theta$ and $G_\theta$ maybe
larger than their corresponding  centres, in fact the kernel
of the adjoint representation of a non-connected group is the
centralizer in the group  of the identity connected component. 
However, we will still denote $\Ad(G_\theta)=G_\theta/Z(G_\theta)$, 
and similarly for $G_\sigma$.
\end{remark}

Again $G_\theta$ normalizes $G^\theta$ and we have an exact sequence

\begin{equation}\label{g-theta}
1 \lra G^\theta\lra G_\theta \lra \Gamma_\theta \lra 1.
\end{equation}

\begin{proposition}\label{c-map}
(1) The map $c: G_\theta\lra Z$ defined by $c(g)$  appearing
in the definition of $G_\theta$, is a homomorphism.
We have  $\ker c= G^\theta$   and hence $c$ induces an injective homomorphism 
$\tilde{c}:\Gamma_\theta\lra Z$.

(2) The action of $\theta$ on $G$ restricts to an action  on $G_\theta$,
 and $c$ is $\theta$-equivariant with 
respect to this action and the natural  action of $\theta$ on $Z$.
Moreover this action descends to $\Gamma_\theta$ and hence  $\tilde{c}$
is $\theta$-equivariant.   

(3) $\theta(c(g))=c(g^{-1})$ and hence the image of $\tilde{c}$ in $Z$
is contained in 
$$
Z_a:=\{ z\in Z \; : \; a(z)=z^{-1}\}, 
$$
and contains
$$
\{ a(z)z^{-1} \; : \; z\in Z\}, 
$$
with $a=\pi(\theta)$, where $\pi$ is the projection $\Aut(G)\to\Out(G)$.

\end{proposition}

\begin{proof}
Let $g_1,g_2\in G_\theta$, we have 
$$
\theta(g_1g_2)=\theta(g_1)\theta(g_2)=c(g_1)g_1c(g_2)g_2=
c(g_1)c(g_2)g_1g_2,
$$
which implies that $c(g_1g_2)=c(g_1)(g_2)$, proving  
the first statement in (1).

Clearly, $\ker c= G^\theta$, and hence we complete the proof of (1).

To prove (2) we have that if $g\in G_\theta$ then $\theta(g)=c(g)g$ with
$c(g)\in Z$.  Now $\theta(c(g)g)=\theta(c(g))\theta(g)=\theta(c(g))c(g)g$. Hence,
since $\theta(c(g))\in Z$, we conclude that $c(g)g\in G_\theta$,  showing that
$\theta$ does indeed act on $G_\theta$. Our computation also shows that
$c(\theta(g))=\theta(c(g))$ and hence $c$ is $\theta$-equivariant.
The  action of $\theta$ on $G_\theta$  fixes  $G^\theta$ and
hence induces an action on $\Gamma_{\theta}$, giving that $\tilde{c}$ is
a $\theta$-equivariant injective homomorphism. 

(3) Now, $g=\theta^2(g)=\theta(c(g)g)=\theta(c(g))\theta(g)=\theta(c(g))c(g)g$,
from which we conclude that $\theta(c(g))=c(g)^{-1}$. 
Since the action of $\theta$ on $Z$ depends only on the clique
$a=\pi(\theta)$, the remaining  now follows from (2). 
\end{proof}

\begin{proposition}
Let $\theta,\theta'\in \Aut_2(G)$ be such that $\theta\sim\theta'$
with $\theta'=\Int(g)\theta\Int(g)^{-1}$ for $g\in G$. 

(1) The map $x \mapsto \Int(g)x$ defines an isomorphism 
$f_g:G_\theta \to G_{\theta'}$, which induces an isomorphism
$\tilde{f_g}:\Gamma_\theta \to \Gamma_{\theta'}$.

(2) Let  $c:G_\theta \to Z $ and $c':G_{\theta'} \to Z$ be the 
homomorphisms corresponding to $\theta$ and $\theta'$, and  
$\tilde{c}:\Gamma_\theta \to Z $ and 
$\tilde{c'}:\Gamma_{\theta'} \to Z$ be the  induced homomorphisms as 
defined in (1) of Proposition \ref{c-map}.
Then
$c=c'f_g$ and  $\tilde{c}=\tilde{c'}\tilde{f_g}$.

\end{proposition}
\begin{proof}

(1) The first statement is immediate. The second follows from this 
and (2) in Proposition \ref{int-subgroups}.

To prove (2), let $x\in G_{\theta}$, and let $x'=f_g(x)=\Int(g)(x)$.
Then $c'(x')x'=\theta'(x')=\Int(g)\theta\Int(g)^{-1}(x')=
\Int(g)\theta(x)=\Int(g)(c(x)x)=c(x)\Int(g)(x)=c(\Int(g)^{-1}(x'))x'$.
The result follows.
 \end{proof}

Let $\sigma$ be a conjugation of $G$ and choose a compact conjugation $\tau$ 
of $G$ such that $\tau\sigma=\sigma\tau=:\theta$. The group $U=G^\tau$
is a maximal compact subgroup of $G$.
Since $\sigma$ and $\tau$ commute, $\sigma$ acts on $U$ and  
$U^\sigma=U\cap G^\sigma$ is a maximal compact subgroup of $G^\sigma$.
We can also consider   
$$
U_\sigma:=\{u\in U\;:\; \sigma(u)=c(u)u,\;\mbox{with}\; c(u)\in
Z(U)=Z\}.
$$
We have that $U_\sigma=U\cap G_\sigma$ is a maximal compact subgroup of
$G_\sigma$.
We thus have the  exact sequence

\begin{equation}\label{}
1 \lra U^\sigma\lra U_\sigma \lra \Gamma_\sigma \lra 1.
\end{equation}

By complexifying this, comparing with  (\ref{g-theta}),
and the fact that $\Gamma_\theta$ is a finite group (this is clear
if $G$ is semisimple since 
$Z$ is finite). We have proved the following.

\begin{proposition}\label{gamma=gamma}
Let $\sigma$ be a conjugation of $G$ and $\tau$ a  compact conjugation
of $G$ commuting with $\sigma$. Let $\theta=\sigma\tau$. Then
$$
\Gamma_\theta=\Gamma_\sigma.
$$ 
\end{proposition}

\begin{proposition}\label{normalizer-u}
The group $U_\sigma$ is $N_U(U^\sigma)$, the  normalizer of 
$U^\sigma$ in $U$.
\end{proposition}
\begin{proof}
The proof follows from Proposition \ref{normalizer-sigma} and the fact
that $\sigma$ commutes with $\tau$, and $U^\sigma\subset G^\sigma$ and 
$U_\sigma\subset G_\sigma$ are maximal compact subgroups. 
\end{proof}

\begin{proposition}\label{theta-normalizer}
The group $G_\theta$ is $N_G(G^\theta)$, the  normalizer of 
$G^\theta$ in $G$.
\end{proposition}

\begin{proof}
Follows from Proposition \ref{normalizer-u} and the fact that
$G$, $G^\theta$ and $G_\theta$ are the complexifications $U$, $U^\sigma$ and 
$U_\sigma$,  respectively.




\end{proof}

\begin{remark}
If $\theta,\theta'\in \Aut_2(G)$ are in the same clique, i.e.
$\pi(\theta)=\pi(\theta')$, but $[\theta]\neq [\theta']$, then
$\Gamma_\theta$ and $\Gamma_{\theta'}$ need not be isomorphic. An example
is provided by $G=\SL(2n,\C)$, $\theta$ corresponding to $\SU(n,n)$ and
$\theta'$ corresponding to $\SU(p,q)$ with $p\neq q$. In this situation,
the normalizer of $\SU(n,n)$ in $\SL(n,\C)$ has two connected componentes,
while the normalizer of $\SU(p,q)$ coincides with $\SU(p,q)$.
\end{remark}

An element  $\theta\in \Aut_2(G)$ defines a Cartan decomposition of 
$\lieg$ in $(\pm 1)$-eigenspaces: 
\begin{equation}\label{cartan-decomposition}
\lieg=\lieg^+ \oplus \lieg^-,
\end{equation}
satisfying $[\lieg^+,\lieg^+]\subset\lieg^+$, 
$[\lieg^-,\lieg^-]\subset\lieg^+$, and $[\lieg^+,\lieg^-]\subset\lieg^-$.  
Clearly $\lieg^+$ is the Lie algebra of $G^\theta$. We have the following.

\begin{proposition}\label{representations}
(1) The restriction of the  adjoint representation of $G$ to $G^\theta$ 
defines representations 
$$
\iota^\pm:  G^\theta\to \GL(\lieg^\pm).
$$

(2) The restriction of the  adjoint representation of $G$ to $G_\theta$ 
defines representations 

$$
\iota_\pm: G_\theta\to \GL(\lieg^\pm).
$$
\end{proposition} 

\begin{proof}
(1) is well known: The restriction of the adjoint representation of $G$ 
 to $G^\theta$ gives indeed the adjoint representation of 
$G^\theta$ in $\lieg^+$ and the {\bf isotropy} representation in $\lieg^-$
(see e.g. \cite{helgason}).

(2) is a consequence of (1) applied to $\Ad(G)$,  together with 
Proposition \ref{lift-g-theta}, and the facts that the Lie algebras and Cartan
decompositions  for $G$ and  $\Ad(G)$ under $\theta$ coincide, and the adjoint
representation
of $G$ factors through the adjoint representation of $\Ad(G)$.

\end{proof}

\subsection{Finite order automorphisms of $G$}
\label{finite-order}

Let $G$ be a connected complex semisimple Lie group. We show now how  many 
results of Section \ref{realforms-group} generalise  to 
automorphisms of $G$ of arbitrary finite order. 

Let $\Aut_n(G)$ and $\Out_n(G)$ be the set of elements of order $n$
in $\Aut(G)$ and $\Out(G)$, respectively. A straightforward 
generalisation of Proposition \ref{cliques}  is the following 
(we leave the  proof to the reader).

\begin{proposition}\label{cliques-n}

Let $\theta\in \Aut_n(G)$.  Consider the set
$$
S^n_\theta:=\{s\in G\;:\; s\theta(s)\cdots \theta^{n-1}(s)=z\in Z\}.
$$
Then

(1)  $Z$  acts on $S^n_\theta$ by multiplication. 

(2) $G$ acts on the right on  $S^n_\theta$ by
$$
s\cdot g:=g^{-1}s \theta(g)=\;\; g\in G, s\in S^n_\theta.
$$

(3) Let $\pi:\Aut_n(G)\to \Out_n(G)$ be the natural projection. 
Let $a\in \Out_n(G)$, and
let $\theta\in \pi^{-1}(a)$. Then
the map $\psi: S^n_\theta\to \pi^{-1}(a)$ defined by $s\mapsto \Int(s)\theta$
gives  a bijection 
$$
S^n_\theta/(Z\times G)\longleftrightarrow \cl_n^{-1}(a),
$$ 
where $\cl_n:\Aut_n(G)/\sim\to \Out_n(G)$ is the map induced by $\pi$.

(4) In particular, let $\theta=\Id\in \Aut_n(G)$. Then
$$
S^n:=S^n_{\Id}=\{s\in G\;:\; s^n=z\in Z\},
$$
and the map $s\mapsto \Int(s)$ defines a bijection 
$$
S^n/(Z\times G)\longleftrightarrow \Int_n(G)/\sim.
$$ 

\end{proposition}

We will refer to an element of $\Out_n(G)$ as an {\bf $n$-clique}.
We also have an interpretation of $\cl_n^{-1}(a)$ in terms of non abelian
cohomology given by the following generalisation of Proposition
\ref{cohomology}.

\begin{proposition}\label{cohomology-n}
Let $\theta\in \Aut_n(G)$. Consider the natural action of the 
group generated by $\theta$, which is isomorphic to $\Z/n$,
on $\Ad(G)$. 
Let $H^1_\theta(\Z/n,\Ad(G))$ be the first cohomology set defined by this
action. Let  $S^n_\theta$  be as defined in Proposition 
\ref{cliques-n} with the action of $Z\times G$ on it  defined in that 
proposition.  
Then, there is a 
bijection
$$
H^1_\theta(\Z/n,\Ad(G)) \longleftrightarrow S^n_\theta/(Z\times G).
$$ 
\end{proposition}
\begin{proof}
The proof is indeed very similar to that of Proposition \ref{cohomology-n}.
Let $Z^1_\theta(\Z/n,\Ad(G))$ be the set of 
cocycles of $\Z/n$ in $\Ad(G)$ with the given action. 
Let $s\in S^n_\theta$ and let $\tilde{s}$ the image of $s$ in 
$\Ad(G)$. The correspondence $s\mapsto \tilde{s}$ 
defines a bijection  $S^n_\theta/Z \to Z^1_\theta(\Z/n,\Ad(G))$, where here
we are identifying a cocycle $a$ with the corresponding element 
$a_\theta\in \Ad(G)$, since the elements $a_{\theta^i}$ for 
$1\leq i\leq n$  are determined by  the recursive formula
$$
a_{\theta^i}=a_{\theta}\theta(a_\theta^{i-1})
$$
given by the cocycle condition (\ref{cocycle}).
Clearly if $a_\theta:=\tilde{s}$, then 
$$
1=a_{\theta^n}=\tilde{s}\theta(\tilde{s})\cdots \theta^{n-1}(\tilde{s}),
$$
proving that $s\in S^n_\theta$.
According to (\ref{cohomologous}) two cocycles are cohomologous
if there is $\tilde{g}\in \Ad(G)$ such that 
$a'_\theta=\tilde{g}^{-1}a_\theta\theta(\tilde{g})$
which coincides with the action of $\Ad(G)$ on $S^n_\theta/Z$. 
\end{proof}

Similarly to Proposition \ref{clique-cohomology}, we have the following.

\begin{proposition}\label{n-clique-cohomology}
Let $a\in \Out_n(G)$. There is a bijection
$$
\cl_n^{-1}(a) \longleftrightarrow H^1_a(\Z/n,\Ad(G)),
$$
and hence a bijection
$$
\Aut_n(G)/\sim \longleftrightarrow \bigcup_{a\in \Out_n(G)}
H^1_a(\Z/n,\Ad(G)).
$$ 
\end{proposition}

One way of understanding the `twisting' defined by $\theta\in \Aut_n(G)$ 
is provided by the following.

\begin{proposition}\label{twisted-group}
Let $\theta\in \Aut_n(G)$ and  $\widehat{G}:= G \rtimes \Z/n$ be the 
semidirect product defined by the natural action of 
$\Z/n=\langle \theta \rangle$ on $G$.
We have the following:

(1) $s\in S^n_\theta$ if and only if $(s,\theta)^n\in Z$ (where $Z$ here
is identified with $Z\times \{1\}\subset \widehat{G}$).

(2) Let   
$\widehat{S}^n=\{\hat{s} \in \widehat{G}\;:\; \hat{s}^n\in Z\}$.
Then  $S^n_\theta/G= \widehat{S}^n/\widehat{G}$, where 
$\widehat{G}$ acts on $\widehat{S}^n$ by conjugation.

\end{proposition}

\begin{proof}

Let $\ast$ denote the group operation in $\widehat{G}$. Then
if $s,s'\in G$, we have 
$(s,\theta^i)\ast (s',\theta^j)=(s\theta^i(s'),\theta^{i+j})$ and hence, 
by induction,
$$
(s,\theta)^n=(s\theta(s)\cdots \theta^{n-1}(s), 1).
$$
We thus conclude that $s\in S^n_\theta$ if and only if $(s,\theta)^n\in Z$,
and (1) follows.

To prove (2) we embed $G$ as the subset $G'\subset \widehat{G}$ by
$s\mapsto (\theta(s),\theta)$. Clearly if $s\in S_\theta^n$, 
$\theta(s)\in S_\theta^n$ and, from (1), this is equivalent to
$(\theta(s),\theta)^n\in Z$. Moreover, the conjugacy action of
$\widehat{G}$ leaves $G'$ invariant and when restricted to the subgroup 
$G$ gives the action by $g\in G$ as $s\mapsto g^{-1}\theta(s)g$.

\end{proof}

As for involutions, for $\theta\in \Aut_n(G)$,  we can define the subgroup

$$
G_\theta:=\{g\in G\;:\; \theta(g)=c(g)g,\;\mbox{with}\; c(g)\in Z\}.
$$
Again, like in the $n=2$ case,  $G_\theta$ normalizes $G^\theta$ and 
we have an exact sequence

\begin{equation}\label{g-theta-n}
1 \lra G^\theta\lra G_\theta \lra \Gamma_\theta \lra 1.
\end{equation}

Propositions \ref{lift-g-theta}
and \ref{c-map} generalise immediately to the following.

\begin{proposition}\label{c-map-n}

(1) $\Ad(G)^\theta =\Ad(G_\theta):=G_\theta/Z(G_\theta)$,
where we denote also by $\theta$ the automorphism defined on $\Ad(G)$. 

(2)  The map $c: G_\theta\lra Z$ defined by $c(g)$ for $c(g)$ appearing
in the definition of $G_\theta$ is a homomorphism,
with $\ker c= G^\theta$,   and hence inducing an injective homomorphism 
$\tilde{c}:\Gamma_\theta\lra Z$.

(3) The action of $\theta$ on $G$ restricts to an action  on $G_\theta$,
 and $c$ is $\theta$-equivariant with 
respect to this action and the natural  action of $\theta$ on $Z$.
Moreover this action descends to $\Gamma_\theta$ and hence  $\tilde{c}$
is $\theta$-equivariant.   

(4) For every $g\in G_\theta$, $c(g)$ satisfies the equation
$$
c(g)\theta(c(g))\cdots \theta^{n-1}(c(g))=1
$$
and hence the image of $\tilde{c}$ in $Z$
is contained in 
$$
Z_a:=\{ z\in Z \; : \; za(z)\cdots a^{n-1}(z)=1\}, 
$$
where $a=\pi(\theta)$.
\end{proposition}


We can decompose $\lieg$ as
\begin{equation}\label{lie-decomposition}
\lieg=\bigoplus_{k=0}^{n-1}\lieg^k,
\end{equation}
where $\lieg^k$ is the eigenspace
of the  automorphism of $\lieg$ defined by $\theta$
 with eigenvalue $\zeta_k:=\exp(2\pi i\frac{ k}{n})$.

Clearly $\lieg^0$ is
the Lie algebra of $G^\theta:=\{g\in G\;:\; \theta(g)=g\}$. We also have 
$[\lieg^l,\lieg^k]\subset \lieg^{l+k}$, in particular $[\lieg^0,\lieg^k]
 \subset \lieg^k$, which lift to  actions of $G^\theta$ and $G_\theta$
on $\lieg^k$. More precisely:

\begin{proposition}\label{representations-n}
(1) The restriction of the  adjoint representation of $G$ to $G^\theta$ 
defines representations 
$$
\iota^k:  G^\theta\to \GL(\lieg^k).
$$

(2) The restriction of the  adjoint representation of $G$ to $G_\theta$ 
defines representations 

$$
\iota_k: G_\theta\to \GL(\lieg^k).
$$
\end{proposition} 

\begin{remark}\label{double-notation}
If $\theta\in\Aut_2(G)$, from (\ref{cartan-decomposition}) and
Proposition \ref{representations}, we have $\lieg^0=\lieg^+$,
$\lieg^1=\lieg^-$,$\iota^0=\iota^+$, $\iota^1=\iota^-$, $\iota_0=\iota_+$
and $\iota_1=\iota_-$.
\end{remark}

\subsection{Groups of Hermitian type}
\label{hermitian-groups}

In this section $G$ is a connected complex simple Lie group. The semisimple
case follows from this. We use the notation of Sections
\ref{realforms-group} and \ref{normalizers}.
Let $\theta\in \Aut_2(G)$, and  $\tau$ be a compact conjugation of
$G$ commuting with $\theta$, defining the conjugation of $G$ given by
$\sigma=\tau\theta$. Let $G^\theta$, $G_\theta$, $G^\sigma$, $G_\sigma$,
$U$, $U^\sigma$, and $U_\sigma$  be as in Section  \ref{normalizers}.

We consider in this section a particular class  of conjugations  of $G$ 
(and hence involutions) which  define what are called groups of
{\bf Hermitian type}. These are conjugations $\sigma$  for which the symmetric
space $M_\sigma:=G^\sigma/U^\sigma$ is of Hermitian type,
in fact a K\"ahler manifold \cite{kobayashi-nomizu,helgason}.
These are distinguished by the fact that $\liez(\lieu^\sigma)$,
the centre of $\lieu^\sigma$, is isomorphic to $\C$. A base element 
in $\liez(\lieu^\sigma)$ defines via de adjoint 
representation a complex structure in $\lieg^-$, the tangent space  
of $M_\sigma$ at the point corresponding to the coset $U^\sigma$.
This applies also to $G_\sigma$. In fact the symmetric space defined
by $G_\sigma$, given by $G_\sigma/U_\sigma$ coincides with $M_\sigma$.
It is well-known that groups of Hermitian type are of Hodge type, as
defined in Section \ref{realforms-group} (see \cite{helgason}).

The conjugations of $G$  that are not of Hermitian type have the property
that $\liez(\lieu^\sigma)=0$ (see \cite{helgason}), and hence the 
groups $U^\sigma$, $G^\theta$, $U_\sigma$, and $G_\theta$ are semisimple. 
This difference between Hermitian and non-Hermitian groups  will have 
important consequences in the theory of $G$-Higgs bundles studied below.

\section{Finite order automorphisms of principal bundles} \label{automorphisms}

In this section $G$ is a connected complex semisimple Lie group and $E$ is a
holomorphic principal $G$-bundle over a compact Riemann surface $X$.


The centre $Z$ of $G$ acts on $E$ by $\xi \to \xi z$, for every $\xi \in 
E$ and $z\in Z$. This gives an inclusion $Z \subset \Aut(E)$ as a 
subgroup. Moreover, $Z$ is in the centre of $\Aut(E)$, for if $z\in Z$ 
and $A\in \Aut(E)$ then $Az(\xi ) =  A(\xi z) = A(\xi )z = z(A(\xi))$. 

In this section we study how an element $A\in \Aut(E)$ such that
$A^n=z\in Z$ for some $n$ gives rise to a reduction of structure group
of the bundle $E$. Such an $A$ defines an automorphism  of finite order 
of the $\Ad(G)$-bundle associated to $E$.
We will generalise this to `twisted' automorphisms in
a sense that we will explain.

\subsection{Finite order automorphisms and reductions of structure group}



Let $E$ be a principal $G$-bundle and $A$ an automorphism. Then there is a natural 
morphism $f_A:E \to G$ given by $A(\xi) = \xi  f_A(\xi )$, for all 
$\xi \in E$. 

\begin{lemma} \label{equivariance}
The map $f_A$ is equivariant for the right action of $G$ on
  $E$ and the (right) adjoint action of $G$ on itself namely $x\to g^{-1}x g$
  for $x,g\in G$.
\end{lemma}

\begin{proof} 
In fact, if $g \in G, \xi \in E$, we have 
$A(\xi g) = (\xi g)f_A(\xi g)$ on the one hand, and  
$A(\xi ) g =  \xi f_A(\xi ) g$ on the other, leading to the equality 
$gf_A(\xi g) = f_A(\xi )g$. 
\end{proof}

\begin{lemma} \label{composition}
If $A_1, A_2 \in \Aut(E)$, then $f_{A_1A_2} = f_{A_1}f_{A_2}$.
\end{lemma}

\begin{proof}
In fact, $\xi f_{A_1A_2}(\xi) = (A_1A_2)(\xi) = A_1(A_2\xi)=
A_1(\xi f_{A_2}(\xi )) = A_1(\xi )f_{A_2}(\xi) = 
\xi f_{A_1} (\xi )f_{A_2}(\xi )$, proving our assertion. 
\end{proof}

\begin{proposition} \label{orbit}
Let $A\in \Aut(E)$ be such that $A^n=z\in Z$. Then:

(1) $f_A$ maps $E$ onto a single orbit $S(E)$ of the set
$$
S^n:=\{s\in G\;:\; s^n=z\in Z\}
$$
under the right  action of  $G$ by inner automorphisms.

(2) Every element $s\in S(E)$ defines a reduction of structure group
of $E$ to $Z_G(s)$. 

(3) The automorphism $A$ leaves the $Z_G(s)$-bundle invariant and coincides
there with the central element $s$.

\end{proposition}

\begin{proof}

By Lemma \ref{composition}, 
$f_{A ^n} = (f_A)^n$
is  the constant map  $\xi \to z$ for an
element  $z\in Z$. Hence $f_A$ 
maps $E$ into $S^n$. The morphism $f_A$ defines a morphism  
$\tilde{f_A}: E \to G\sslash G$, where $G\sslash G$ is the 
GIT quotient for the (right) action of $G$ on itself 
given by $x\mapsto g^{-1}x g$ for $x,g\in G$.
Since $\tilde{f_A}$ is constant on the fibres (Lemma \ref{equivariance}),
it descends to a morphism $X=E/G\to G\sslash G$, which must be 
constant since $X$ is projective and  $G\sslash G$ is affine.
But  $f_A(\xi )\in S^n$  and hence is  a semisimple element of $G$ for 
all $\xi \in E$ and hence  a stable point  for the GIT quotient $G\sslash G$. 
Therefore $f_A(\xi)$ for every $\xi\in E$ lie on a single orbit for the given 
action  of $G$  on $S^n$ as  claimed.




To prove (2), let $s\in S(E)$, and  $S:= f_A ^{-1}(s)$. 
Clearly $\xi$ and $\xi '$ on a
 fibre of $E\to X$ belong to $S$ if and only if 
$\xi ' = \xi g$ for some $g\in G$, and  $f_A(\xi ) = s $ and 
$s=f_A(\xi ') = f_A(\xi  g) = g^{-1} f_A(\xi ) g = g^{-1}sg$.  
This implies that $g\in Z_G(s)$. Thus
 the  set $S$ is acted upon transitively on fibres by $Z_G(s)$ giving an
 $Z_G(s)$-bundle to which $E$ is reduced.

Assertion (3) is obvious.

\end{proof}

\subsection{Twisted automorphisms of principal bundles}
\label{twisted-automorphisms}
Let $E$ be a $G$-bundle over $X$.  Let $\theta\in \Aut(G)$. We define 
the set of 
$\theta$-{\bf twisted automorphisms} of $E$ by
$$
\Aut_\theta(E):=\{A: E\to E\;\; \mbox {bijective} \;\;:\;\; A(\xi
g)=A(\xi)\theta(g)\;\;\mbox{for}\;\; \xi\in E, g\in G \}.
$$

Twisted automorphisms are related to ordinary isomorphisms between $E$ and
a related bundle. More precisely, the following is immediate.

\begin{proposition}\label{twisted-bundle}
Let $E$ be a $G$-bundle over $X$.  Let $\theta\in \Aut(G)$. Then:

(1) The $G$-bundle $\theta(E)$ is isomorphic to the $G$-bundle whose total 
space is $E$ with a $G$-action defined by 
$$
e\cdot g:=e\theta(g)\;\;\;\mbox{for}\;\;\; e\in E, g\in G.
$$

(2) Under the isomorphism given in (1) an isomorphism $E\to \theta(E)$ can be
identified with a $\theta$-twisted automorphism of $E$.

\end{proposition}

Let $A\in \Aut_\theta(E)$. We define the function $f_A:E\to G$
by the formula
\begin{equation}\label{function-A}
A(\xi)=\xi f_A(\xi)\;\;\mbox{for every} \;\; \xi\in E.
\end{equation}

\begin{lemma}\label{theta-equivariance}
Let $A\in\Aut_\theta(E)$ and let $f_A: E\to G$ the function given
by   (\ref{function-A}). The function $f_A$ is $G$-equivariant for the 
right action of $G$ on $E$ and the right action of $G$ on itself given
by
$$
x\mapsto g^{-1}x\theta(g)\;\; \mbox{for}\;\; x,g \in G. 
$$
\end{lemma}

\begin{proof}
Let $\xi\in E$ and $g\in G$. On the one hand we have 
$A(\xi g)=(\xi g)f_A(\xi g)$. On the other hand 
$A(\xi g)=A(\xi)\theta(g)=\xi f_A(\xi)\theta(g)=\xi g g^{-1}f_A(\xi)\theta(g)$, 
thus concluding that $f_A(\xi g)= g^{-1} f_A(\xi)\theta(g)$. 
\end{proof}

\begin{lemma}\label{A-multiplication}
Let $\theta_1,\theta_2\in \Aut(G)$, and $A_1\in \Aut_{\theta_1}(E)$
and $A_2\in \Aut_{\theta_2}(E)$. Then

(1) $A_1A_2\in \Aut_{\theta_1\theta_2}(E)$.

(2) $f_{A_1A_2}= f_{A_1}\cdot\theta_1(f_{A_2})$, 
where this means that
$f_{A_1A_2}(\xi)= f_{A_1}(\xi)\theta_1((f_{A_2})(\xi))$, for every
$\xi\in E$.

\end{lemma}

\begin{proof}
Let $\xi\in E$ an $g\in G$.

\begin{align*}
A_1(A_2(\xi g))& = A_1(A_2(\xi)\theta_2(g))\\ 
                   & = A_1(A_2(\xi))\theta_1(\theta_2(g)) \\
                   & = (A_1 A_2)(\xi)(\theta_1\theta_2)(g), 
\end{align*}
proving (1). The proof of (2) is given by the following computation:

\begin{align*}
\xi f_{A_1A_2}(\xi)& = (A_1A_2)(\xi)\\ 
                   & = A_1(A_2(\xi)) \\
                   & = A_1(\xi f_{A_2}(\xi))\\
                   & = A_1(\xi)\theta_1(f_{A_2}(\xi))\\
                   & = \xi f_{A_1}(\xi)\theta_1(f_{A_2}(\xi)). 
\end{align*}

\end{proof}

From Lemma \ref{A-multiplication} we conclude the following.

\begin{proposition}
Let $\theta\in \Aut_n(G)$ and let $E$ be a $G$-bundle. The set
$$
\widehat{\Aut(E)}:=\bigcup_{i=0}^{n-1} \Aut_{\theta^i}(E)
$$
is a (possibly disconnected) group fitting in  an  exact sequence

$$
1 \to \Aut(E)\to \widehat{\Aut(E)}\to \Z/n.
$$

\end{proposition}

One can easily prove the following.

\begin{proposition}
Let $E$ be a $G$-bundle and  $\widehat{E}$ be the $\widehat{G}$-bundle 
associated to $E$ by extension of structure group,
where $\widehat{G}$ is the group defined in Proposition \ref{twisted-group}.
Then
$$
\widehat{\Aut(E)} \cong \Aut(\widehat{E}). 
$$ 
\end{proposition}

Here is the twisted version of Proposition \ref{orbit}.

\begin{proposition}\label{twisted-orbit}
Let $E$ be a $G$-bundle over a compact Riemann surface $X$. 
Let $\theta\in \Aut_n(G)$ and  $A\in \Aut_\theta(E)$ such that
$A^n=z\in Z\subset \Aut(E)$. Then: 

(1) The function $f_A$ given in  (\ref{function-A}) maps
$E$ onto a single orbit $S(E)$ of the set 
$S^n_\theta:=\{s\in G\;:\; s\theta(s)\cdots \theta^{n-1}(s)=z\in Z\}$
 defined 
in Proposition  \ref{cliques-n} under the right action of $G$ defined 
there, namely $s\cdot g= g^{-1}s\theta(g)$.

(2) Every element $s\in S(E)$ defines a reduction of structure group of
$E$ to $G^{\theta'}$, where $\theta'=\Int(s)\theta$, and $G^{\theta'}$ 
is the subgroup of $G$ of fixed points under $\theta'$.
\end{proposition}

\begin{proof}
We adapt the proof of Proposition \ref{orbit} to the twisted situation.
By Lemma \ref{A-multiplication}, 
$f_{A^n}=f_A\cdot \theta(f_A)\cdots \theta^{n-1}(f_A)=f_z$ 
where  $f_z: E\to G$ is the constant map  given by $f_z(\xi)=z$ for an
element  $z\in Z$ and every $\xi\in E$. By setting
$s:=f_A(\xi)$, this  implies that $s\in S^n_\theta$.

The morphism $f_A$ defines a morphism  
$\tilde{f_A}: E \to G\sslash_\theta G$, where $G\sslash_\theta G$ is the 
GIT quotient for the action of $G$ on itself 
given by the $\theta$-twisted action 
 $x\mapsto g^{-1}x \theta(g)$ for $x,g\in G$.
Since $\tilde{f_A}$ is constant on the fibres (Lemma \ref{theta-equivariance}),
it descends to a morphism $X=E/G\to G\sslash_\theta G$, which must be 
constant since $X$ is projective and  $G\sslash_\theta G$ is affine.

Now, we claim that every element $s\in S^n_\theta$ is stable for the 
$\theta$-twisted action. This follows from Proposition \ref{twisted-group},
since the element $s\in S^n_\theta$ is in correspondence with an element of
order $n$ in the twisted group $\widehat{G}$ defined in  Proposition
\ref{twisted-group} and is hence a semisimple element of $\widehat{G}$.
We thus deduce that  $f_A(\xi )\in S^n_\theta$  is a stable point 
for the GIT quotient $G\sslash_\theta G$, and hence  $f_A(\xi)$ 
for every $\xi$ lie on a single orbit for the $\theta$-twisted 
action  of $G$  on $S^n_\theta$ as we claim in (1).

The proof of (2) is a straightforward generalisation of that of (2) in
Proposition \ref{orbit}. 

\end{proof}

One may consider a different kind of automorphisms of a $G$-bundle that 
involves a central subgoup $\Gamma $ (in particular $\Gamma=Z$). Given $\theta \in \Aut(G)$, we 
define a {\bf $(\theta , \Gamma)$-twisted automorphism} to be a 
fibre-preserving morphism $A:E\to E$ satisfying $A(\xi g) = A(\xi )z \theta(g)$
for all $\xi \in E$ and $g\in G$ with $z\in \Gamma$ (depending on $A$, 
$\theta$, $\xi$ and $g$. We will denote the set of 
$(\theta , \Gamma)$-twisted automorphisms of $E$ by $\Aut^\Gamma_\theta(E)$.

Let then $E$ (resp. $\alpha $) be a principal $G$-bundle (resp. 
$\Gamma$-bundle). The fibre product $E\times _X\alpha $ is then in a 
natural way a principal ($G\times \Gamma $)-bundle. Since $\Gamma $ is 
abelian, we may, and will often, write the action of $\Gamma $ on $\alpha$ 
on the left.

Using the homomorphism $m:G\times \Gamma \to G$ given by multiplication, 
we get by extension of structure group, a principal $G$-bundle which we 
denote by $E\otimes \alpha $. This is clearly a quotient of $E\times 
_X\alpha $ by the action of $\Gamma $ given by $z(\xi, a) = (\xi z^{-1}, 
za)$. We will denote the image of $(\xi , a)$ in $E\otimes \alpha $ by 
$\xi\otimes a$. Notice that for any $z\in \Gamma $ we have $\xi.z\otimes a 
= \xi \otimes za$ for all $a\in \alpha$. 

The proof of the following proposition is similar to that of Proposition 
\ref{twisted-bundle}.

\begin{proposition}\label{alpha-theta-twisted-automorphisms}
Let $E$ and $\alpha $ be principal bundles with 
structure groups $G$ and a central subgroup $\Gamma$ of $G$ respectively.
Let $\theta\in \Aut(G)$.  
An isomorphism $E \to \theta(E)\otimes \alpha$ can be identified with 
a $(\theta,\Gamma)$-twisted automorphism of $E$.
\end{proposition}





Let $A\in \Aut^\Gamma_\theta(E)$. As in the previous cases, 
we define the function $f_A:E\to G$
by the formula
\begin{equation}\label{function-A-Z}
A(\xi)=\xi f_A(\xi)\;\;\mbox{for every} \;\; \xi\in E.
\end{equation}

As in Lemmas  \ref{theta-equivariance} and \ref{A-multiplication}, 
we can easily prove the following.

\begin{lemma}\label{theta-Z-equivariance}
Let $A\in\Aut_\theta^\Gamma(E)$ and let $f_A: E\to G$ the function given
by   (\ref{function-A-Z}). Then
$$
f_A(\xi g)=zg^{-1}f_A(\xi)\theta(g) \;\;\mbox{for}\;\; \xi\in E, g\in G,
$$
where $z\in \Gamma$ depends on $A$, $\theta$, $\xi$, and $g$. 
\end{lemma}

\begin{lemma}\label{A-Z-multiplication}
Let $\theta_1,\theta_2\in \Aut(G)$, and $A_1\in \Aut^\Gamma_{\theta_1}(E)$
and $A_2\in \Aut^\Gamma_{\theta_2}(E)$. Then

(1) $A_1A_2\in \Aut^\Gamma_{\theta_1\theta_2}(E)$.

(2) $f_{A_1A_2}(\xi)= zf_{A_1}(\xi)\theta_1(f_{A_2}(\xi))$, for 
$\xi\in E$, where $z\in \Gamma$ depends on $\xi$, $A_1$, $A_2$ and $\theta_1$.
\end{lemma}

\begin{proposition} \label{Z-twisted-orbit}
Let $E$ be a $G$-bundle over a compact Riemann surface $X$. 
Let $\theta\in \Aut_n(G)$ and  $A\in \Aut^Z_\theta(E)$  such 
that $A^n=f$, for a function $f:E\to Z$. Then:

(1) $f_A$ defined by (\ref{function-A-Z}) maps $E$ onto a single orbit $S(E)$ 
of the set $S_\theta^n$  under the action of  $Z\times G$  defined in 
Proposition \ref{cliques-n}.

(2) Every element $s\in S(E)$ defines a reduction of structure group
of $E$ to $G_{\theta'}$, where  $\theta'=\Int(s)\theta$ and $G_{\theta'}$ is
 defined as in  Section \ref{finite-order}. 
\end{proposition}
\begin{proof}
We follow closely the proofs of Propositions \ref{orbit} and 
\ref{twisted-orbit}.
By Lemma \ref{A-Z-multiplication}, 
$$
f(\xi)=f_{A^n}(\xi)= zf_A(\xi)\theta(f_A(\xi))\cdots\theta^{n-1}(f_A(\xi)),
$$
where $z$ depends on $\xi$. 
Setting $s:=f_A(\xi)$,  we conclude
that $s\in S_\theta^n$ since $f(\xi)\in Z$.

The morphism $f_A$ defines now a morphism  
$\tilde{f_A}: E \to G\sslash_\theta (Z\times G)$, where 
$G\sslash_\theta (Z\times G)$ is the 
GIT quotient for the action of $Z\times G$ on $G$
given by 
 $x\mapsto z g^{-1}x \theta(g)$ for $x,g\in G$ and $z\in Z$.
Since $\tilde{f_A}$ is constant on the fibres 
(Lemma \ref{theta-Z-equivariance}),
it descends to a morphism $X=E/G\to G\sslash_\theta (Z\times G)$, which must be 
constant since $X$ is projective and  $G\sslash_\theta (Z\times G)$ is affine.
The rest of the argument to prove (1) is like in the proof 
of Proposition \ref{twisted-orbit}.

Again the  proof of (2) is a straightforward generalisation of that of (2) in
Proposition \ref{orbit}.  Setting  $S:= f_A ^{-1}(s)$ for $s\in S(E)$, we see 
that $\xi$ and $\xi '$ on a
 fibre of $E\to X$ belong to $S$ if and only if 
$\xi ' = \xi g$ for some $g\in G$, and  $f_A(\xi ) = s $ and $f_A(\xi' ) = s
 $.  But,
from Lemma \ref{theta-Z-equivariance},
$s=f_A(\xi ') = f_A(\xi  g) = zg^{-1} f_A(\xi ) \theta(g) $,  
which implies that $g\in G_{\theta'}$, with $\theta'=\Int(s)\theta$.
Thus the  set $S$ is acted upon transitively on fibres by $G_{\theta'}$ 
giving a $G_{\theta'}$-bundle to which $E$ is reduced.

\end{proof}
\section{$G$-Higgs bundles}
\label{higgs-bundles}

\subsection{Moduli space of $G$-Higgs bundles}\label{g-higgs}

Let $G$ be a complex semisimple Lie group (not necesarily connected)
 with Lie algebra $\lieg$. 
Let $X$ be a smooth projective curve over $\C$, equivalently a 
compact Riemann surface.
A {\bf $G$-Higgs bundle} over $X$ is a pair $(E,\varphi)$ where
$E$ is  a principal $G$-bundle $E$ over $X$ and
$\varphi$ is a  section of $E(\lieg)\otimes K$, where $E(\lieg)$ is
the bundle associated to $E$ via the adjoint representation of $G$, 
and $K$ is the canonical bundle on $X$.

Two $G$-Higgs bundles $(E,\varphi)$ and $(F,\psi)$ are isomorphic if there is
an isomorphism $f:E\to F$  such that  the induced isomorphism
$\Ad(f)\otimes \Id_K: E(\lieg)\otimes K \to F(\lieg)\otimes K$
sends $\varphi$ to $\psi$.

In order to consider moduli spaces of $G$-Higgs bundles we need the
corresponding notions of (semi,poly)stability.
We briefly recall the main definitions. Our approach follows  
\cite{garcia-prada-gothen-mundet}, where all these general 
notions are studied in detail.


Let $\lieu$ be the Lie algebra of a maximal compact subgroup $U$  of $G$. 
Given $s\in i\lieu$,
\begin{equation}\label{eq:parabolicPs}
P_s=\{g\in G\st e^{ts}ge^{-ts}\;\; \text{is bounded  as}\;\; t\to\infty\},
\end{equation}
is a parabolic subgroup of $G$, whose corresponding parabolic subalgebra 
of $\lieg$ is 
$$
\liep_s=\{v\in\lieg\st\Ad(e^{ts})(v)\;\; \text{is bounded  as}\;\; t\to\infty\}.
$$
If, moreover, we define 
\begin{equation}\label{eq:LeviLs}
L_s=\{g\in G\st\lim_{t\to\infty} e^{ts}ge^{-ts}=g\}
\end{equation}
then $L_s\subset P_s$ is a Levi subgroup of $P_s$, and 
$$\liel_s=\{v\in\lieg\st\lim_{t\to\infty}\Ad(e^{ts})(v)=0\}$$ is the 
corresponding Levi subalgebra of $\liep_s$.

When $G$ is connected, every parabolic subgroup $P$ is of the form 
\eqref{eq:parabolicPs} for some $s\in i\lieu$; the same holds for the 
Levi subgroups. For $G$ non-connected that may not be the case 
(cf. \cite[Remark 5.3]{martin:2003}). However, in order to define 
semistability, the parabolic subgroups which need to be considered 
are precisely the ones of the form \eqref{eq:parabolicPs}. 
Hence, for simplicity, and when no explicit mention to $s\in i\lieu$ 
is needed, we refer to these as the parabolic subgroups of $G$, 
keeping in mind that we mean the groups defined by \eqref{eq:parabolicPs}. 
We will do the same for the Levi subgroups, referring to \eqref{eq:LeviLs}.

Let $P$ be a parabolic subgroup of $G$. A character of the Lie algebra
$\liep$ of $P$ is a complex linear map $\liep\to\C$ which factors through
$\liep/[\liep,\liep]$. Let $\liel\subset\liep$ be the corresponding Levi 
subalgebra and let $\liez_\liel$ be the centre of $\liel$. 
Then, one has that $(\liep/[\liep,\liep])^*\cong\liez_\liel^*$, 
so the characters of $\liep$ are indeed classified by elements of 
$\liez_\liel^*$.
Using the Killing form a character $\chi \in\liez_\liel^*$ of $\liep$ is
uniquely determined 
by an element $s_{\chi}\in\liez_\liel$. Indeed, it can be shown that 
$\liez_\liel\subset i\lieu$, so that $s_{\chi}\in i\lieu$.
Now, the character $\chi$ of $\liep$ is said to be {\bf antidominant} 
if $\liep\subset\liep_{s_{\chi}}$ and  {\bf strictly antidominant} 
if $\liep=\liep_{s_{\chi}}$. 
Given a character $\chi:P\to\C^*$ of $P$, denote by $\chi_*$ the 
corresponding character of $\liep$. We say that $\chi$ is 
{\bf (strictly) antidominant} if $\chi_*$ is.
An element $s\in i\lieu$ defines clearly a character $\chi_s$ of $\liep_s$ 
since $\la s,[\liep_s,\liep_s]\ra=0$ (where $\la\cdot,\cdot\ra$ is 
the Killing form)  which is of course strictly antidominant.


Let $\chi:P\to\C^*$ be an antidominant character of $P$ and 
Let $\sigma\in \Gamma(X,E/P)$ be a reduction of the structure group of 
$E$ to $P$. Denote by $E_{\sigma}\subset E$ the corresponding holomorphic $P$-bundle. 
The {\bf degree} of $E$ with respect to $\sigma$ and  $\chi$, denoted by 
$\deg(E)(\sigma,\chi)$, 
is the degree of the line bundle obtained by extending the structure group 
of $E_\sigma$ through $\chi$. In other words,
\begin{equation}\label{def:degree}
\deg(E)(\sigma,\chi)=\deg(E_\sigma\times_\chi\C^*).
\end{equation}

Given  $s\in i\lieu$ and a reduction of structure group of $E$ to $P_s$
we can also define $\deg(E)(\sigma,s)$, even though $\chi_s$ may not lift
to a character of $P_s$. One way to do this is to use Chern--Weil theory.
This definition is  more natural when considering gauge-theoretic 
equations as we do below. For this, define $U_s=U\cap L_s$ and
$\lieu_s=\lieu\cap\liel_s$. 
Then $U_s$ is a maximal compact subgroup of $L_s$, so the inclusion
$U_s\subset L_s$ is a homotopy equivalence. Since the inclusion
$L_s\subset P_s$ is also a homotopy equivalence, given a reduction
$\sigma$ of the structure group of $E$ to $P_s$ one can
further restrict the structure group of $E$ to $U_s$ in a unique
way up to homotopy. Denote by $E'_{\sigma}$ the resulting $U_s$
principal bundle.
Consider now a connection $A$
on $E'_{\sigma}$ and let
$F_A\in\Omega^2(X,E'_{\sigma}(\lieu_s)$ be  its
curvature. Then $\chi_s(F_A)$ is a $2$-form on $X$ with
values in $i\R$, and 
\begin{equation}\label{degree-chern-weil}
\deg(E)(\sigma,s):=\frac{i}{2\pi}\int_X \chi_s(F_A).
\end{equation}
This coincides with $\deg(E)(\sigma,\widetilde{\chi_s})$ as defined in
(\ref{def:degree}) when $\chi_s$ can be lifted to a character 
$\widetilde{\chi_s}$ of $P_s$.


\begin{definition}\label{def:semipoly}
A $G$-Higgs bundle $(E,\varphi)$ over $X$ is:
 
{\bf semistable} if
$\deg(E)(\sigma,\chi)\geq 0$, 
for any parabolic subgroup $P$ of $G$, any non-trivial antidominant 
character $\chi$ of $P$ and any reduction of structure group $\sigma$ of 
$E$ to $P$ such that $\varphi\in H^0(X,E_\sigma(\liep)\otimes K)$.

{\bf stable} if 
$\deg(E)(\sigma,\chi)> 0$, for any non-trivial parabolic subgroup $P$ of 
$G$, any 
non-trivial antidominant  character $\chi$ of $P$ and any reduction of 
structure group $\sigma$ of $E$ 
to $P$ such that $\varphi\in H^0(X,E_\sigma(\liep)\otimes K)$.

 {\bf polystable} if it is semistable and if 
$\deg(E)(\sigma,\chi)=0$, for some parabolic subgroup 
$P\subset G$, some non-trivial strictly antidominant character $\chi$ of $P$ 
and some reduction of structure group $\sigma$ of $E$ to $P$
such that $\varphi\in H^0(X,E_\sigma(\liep)\otimes K)$, then there is a 
further holomorphic reduction of structure group $\sigma_L$ of $E_\sigma$ 
to the Levi subgroup $L$ of $P$ such that 
$\varphi\in H^0(X,E_{\sigma_L}(\liel)\otimes K)$.

\end{definition}

\begin{remark}\label{rm:(semi)stability of bundles}

(1) A $G$-Higgs bundle with $\varphi=0$ is a holomorphic principal $G$-bundle 
and a (semi)stability condition for these objects over compact Riemann
surfaces  was established by Ramanathan in \cite{ramanathan:1975}. 
One has a direct generalisation of Ramanathan's condition to 
the $G$-Higgs bundle 
case for $G$ complex (see for example 
\cite{biswas-gomez,garcia-prada-gothen-mundet}). 

(2) The notion of $G$-Higgs bundle given above makes sense also when $G$ is more
generally a complex reductive Lie group (in fact any complex Lie group). In
the general reductive case, however,  the stability 
criteria has to be modified  by replacing $\deg(E)(\sigma,\chi)$ in
Definition \ref{def:semipoly} 
with  $\deg(E)(\sigma,\chi)-\la\alpha,s_{\chi_*}\ra$, where  
$\la\cdot,\cdot\ra$ is
an invariant $\C$-bilinear pairing on $\lieg$ extending the Killing 
form on the semisimple part, and $\alpha$ is an element in $\in i\liez_\lieu$ 
determined by the topology of the $G$-bundle $E$.
There is no discrepancy with \cite{ramanathan:1975,biswas-gomez} where
the reductive case is treated and there is no parameter $\alpha$. The reason
is that these authors  consider characters 
which are trivial on the centre of $G$, and hence  the corresponding 
characters on the Lie 
algebra are orthogonal to $\alpha$ with respect to the pairing 
$\la\cdot,\cdot\ra$.
\end{remark}

Let  $\cM(G)$  be {\bf moduli space of semistable $G$-Higgs bundles}.
As usual, the moduli space $\cM(G)$ can also be viewed as parametrizing
isomorphism classes of polystable $G$-Higgs bundles. The space
$\cM(G)$ has the structure of a quasi-projective variety, as one can
see  from the Schmitt's general Geometric Invariant Theory construction 
(cf. \cite{schmitt:2008}). For related  constructions
see \cite{nitsure,simpson:1994,simpson:1995}. 
If we fix  the topological class $c$  of $E$ we can consider  
$\cM_c(G)\subset \cM(G)$, the  moduli space of semistable $G$-Higgs bundles
with fixed topological class $c$. If $G$ is connected the topological 
class is given by an element of $c\in \pi_1(G)$.
In this situation it is well-known (\cite{li,donagi-pantev}) 
that $\cM_c(G)$ is non-empty 
and connected.  A Morse-theoretic  proof of this fact has been given recently 
in \cite{garcia-prada-oliveira}, where the connectedness 
and non-emptyness  of  $\cM_c(G)$ 
is also proved when $G$ is a non-connected complex reductive Lie group.


For the rest of the section we will assume $G$ to be connected.

Let  $(E, \varphi )$ be a Higgs bundle. Since the induced action of $Z$ on 
$E(\lieg)$ is trivial, they preserve $\varphi $ and we see  that $Z$
is a central subgroup of $\Aut(E, \varphi)$, and that the quotient is finite if 
$(E,\varphi)$ is stable. 

\begin{definition}\label{simple-higgs}
A  $G$-Higgs bundle $(E, \varphi )$ 
is said to be {\bf simple} if  $\Aut(E, \varphi )$ is $Z$. 
\end{definition}

\begin{remark}
When $G = \GL(n,\C)$ or $SL(n,\C)$, every stable bundle is simple. This is not 
true in general. In the case $G = \SO(n,\C)$, the direct sum of two 
non-isomophic orthogonal bundles is still stable but is in general 
not simple.  The phenomenon of stable bundles not being simple has 
to do with the 
coefficients of the highest root of a simple Lie algebra being not $1$. This
phenomenon is related to the existence of elements of finite order in the
group $G$ whose centralizers are semisimple (not just reductive). One can 
show that $\GL(n,\C)$ or $SL(n,\C)$ are the only cases for which this does
not happen.
\end{remark}

\begin{remark}\label{non-connected-simplicity}
If $G$ is non-connected, the notion of simplicity of a $G$-Higgs bundle 
$(E,\varphi)$ depends on the topological type of $E$, more concretely,
on the monodromy class $w:\pi_1(X)\to \pi_0(G)$, where $\pi_0(G)$ is the 
group of connected components of $G$. Such a Higgs bundle includes  in
its group of automorphisms the
invariant part of $Z_G(G_0)$ under the action of the image of $w$, where
$G_0$ is the identity connected component, and $Z_G(G_0)$ is
the centralizer of $G_0$ in $G$. This is larger in general than $Z$
(see \cite{garcia-prada-gothen-mundet2}).
\end{remark}

We have the following (see \cite{garcia-prada-gothen-mundet}).

\begin{proposition}\label{smoothness}
A $G$-Higgs bundle which is stable and simple defines a smooth point in 
$\cM(G)$. 
\end{proposition}


\subsection{$G$-Higgs bundles and Hitchin equations}
\label{section-hitchin-equations}

As above, let $G$ be a connected  complex semisimple Lie  group. Let 
$U\subset G$ be a maximal compact subgroup defined by a conjugation
$\tau$ of $G$.  Let $(E,\varphi)$ be a
$G$-Higgs bundle over a compact Riemann surface $X$. By a slight abuse
of notation, we shall denote the $C^\infty$-objects underlying $E$ and
$\varphi$ by the same symbols. In particular, the Higgs field can be
viewed as a $(1,0)$-form $\varphi \in
\Omega^{1,0}(E(\lieg))$ with values in $E(\lieg)$. 
Given a $C^\infty$ reduction of
structure group $h$ of the principal $G$-bundle $E$ to $U$, 
we can define
\begin{equation}\label{tau-h}
\tau_h\,\colon\,
\Omega^{1,0}(E(\lieg)) \,\longrightarrow\, \Omega^{0,1}(E(\lieg))
\end{equation}
the isomorphism induced by the
compact conjugation of $\lieg$ defined by $\tau$,
 combined with the complex
conjugation on complex $1$-forms.

we denote 
by $F_h$ the  curvature of the unique connection compatible with $h$ and 
the holomorphic structure on $E$ (see \cite{atiyah}). 
One has the following.

\begin{theorem} \label{higgs-hk}
Let $(E,\varphi)$ be a $G$-Higgs bundle. There is a reduction $h$ of 
structure group of $E$ from $G$ to $U$ that satisfies the Hitchin equation
  $$
 F_h -[\varphi,\tau_h(\varphi)]= 0
  $$
  if and only if $(E,\varphi)$ is polystable.
\end{theorem}

Theorem \ref{higgs-hk} was proved by
Hitchin \cite{hitchin1987}  for $G\,=\,\SL(2,\C)$, and by Simpson in
\cite{simpson,simpson:1992} for the general case (see also 
\cite{garcia-prada-gothen-mundet}).


\begin{remark}\label{alpha=0}
Theorem \ref{higgs-hk} can be extended to the case in which $G$ is a complex 
reductive Lie group, as well as non-connected. In this situation the Hitchin 
equation is replaced with  
$F_h -[\varphi,\tau_h(\varphi)]= -i\alpha\omega$, where  $\alpha\in
i\liez_\lieu$ is determined by the topology of $E$, via Chern--Weil theory
from the equation and  $\omega$ is  a volume form of $X$. 

\end{remark}

{}From the point of view of moduli spaces it is convenient
to fix  a $C^\infty$ principal   $U$-bundle
$\bE_U$  and study  the moduli space of solutions to 
\textbf{Hitchin's equations}
for a pair $(A,\varphi)$ consisting of  an $U$-connection $A$ and
a section $\varphi\,\in\, \Omega^{1,0}(X,\bE_U(\lieg))$:
\begin{equation}\label{hitchin}
\begin{array}{l}
F_A -[\varphi,\tau_h(\varphi)]= 0\\
\dbar_A\varphi = 0.
\end{array}
\end{equation}
Here $d_A$ is the covariant derivative associated to $A$, and
$\dbar_A$ is the $(0,1)$ part of $d_A$. The $(0,1)$ part of $d_A$ defines 
a holomorphic
structure on $\bE_U$. The gauge group $\UUU$  of $\bE_U$, i.e. the group of 
automorphisms of $\bE_U$  acts on the
space of solutions and the moduli space of solutions is
$$
\Mg_c(G):= \{ (A,\varphi)\;\;\mbox{satisfying}\;\;
(\ref{hitchin})\}/\UUU,
$$
where $c$ is the topological type of the bundle $\bE_U$. 
The irreducible solutions define smooth points in $\Mg_c(G)$.
Now, if $\cM_c(G)$ is the moduli space of $G$-Higgs bundles 
of topological type $c$, i.e. the $G$-Higgs bundles whose underlying
 $C^\infty$ bundle is the $G$-bundle obtained from $\bE_U$ by extension
of the structure group to $G$, 
Theorem \ref{higgs-hk} can be reformulated as follows.

\begin{theorem} \label{hk}
There is a homeomorphism
$$
\cM_c(G)\,\cong\, \Mg_c(G).
$$

Moreover, the irreducible solutions in $\Mg_c(G)$
correspond with the stable and simple Higgs bundles in $\cM_c(G)$.
\end{theorem}

To explain this correspondence we interpret the moduli
space of $G$-Higgs bundles in terms of pairs $(\dbar_E, \varphi)$ consisting
of a $\dbar$-operator (holomorphic structure) $\dbar_E$
on the $C^\infty$ principal $G$-bundle $\bE_{G}$ obtained from
$\bE_U$ by the extension of structure group $U\,\hookrightarrow\, G$, and
$\varphi\in \Omega^{1,0}(X,\bE_G(\lieg))$ satisfying $\dbar_E\varphi\,=\,0$.
Such pairs are in one-to-one correspondence with  $G$-Higgs bundles $(E,\varphi)$,
where $E$ is the holomorphic $G$-bundle defined by the operator
$\dbar_E$ on $\bE_G$. The equation $\dbar_E\varphi\,=\,0$ is equivalent
to the condition that $\varphi\in H^0(X,E(\lieg)\otimes K)$. 
The moduli space of polystable $G$-Higgs bundles of topological type 
given by $\bE_U$ can now be identified with the orbit space
$$
\{ (\dbar_E,\varphi)\;\;:\;\; \dbar_E\varphi=0\;\;\mbox{which are polystable}\}/\GGG,
$$
where $\GGG$ is the gauge group of $\bE_G$, which is in fact
the complexification of $\UUU$.
Since  there is a one-to-one correspondence between
$U$-connections on $\bE_U$ and $\dbar$-operators on $\bE_{G}$,
the correspondence given in Theorem \ref{hk} can be reformulated
by saying that in the $\GGG$-orbit of a polystable $G$-Higgs
bundle $(\dbar_{E_0},\varphi_0)$ we can find another Higgs bundle
$(\dbar_E,\varphi)$
whose corresponding pair $(d_A,\varphi)$ satisfies the Hitchin equation
$F_A -[\varphi,\tau_h(\varphi)]\,=\, 0$ with this pair
$(d_A,\varphi)$ being unique up to $U$-gauge transformations.


Following \cite{hitchin1987}, one can see that
the smooth locus of the moduli space $\cM(G)$ 
has a hyperk\"ahler structure given by exhibiting the moduli
space of solution to Hitchin equations, and hence $\cM(G)$,
as a hyperk\"ahler quotient. So the smooth locus of $\cM(G)$  
is equipped with a 
Riemannian metric $g$ and complex structures $J_i$, $i=1,2,3$ 
satisfying the quaternion relations $J_i^2=-I$, 
$J_3=J_1J_2\,=\, -J_2J_1$, $J_2\,=\, -J_1J_3\,=\, J_3J_1$ and $J_1\,=\, J_2J_3
\,=\, -J_3J_2$ such that if we define $\omega_i(\cdot, \cdot)\,=\,g(J_i\cdot, \cdot)$, then
$(g,J_i,\omega_i)$ is a K\"ahler structure on $\cM(G)$.
Let $\Omega_i$ denote the holomorphic symplectic structure on $\cM(G)$ with respect
to the complex structure $J_i$. That is, $\Omega_1\,=\, \omega_2+\sqrt{-1}\omega_3$,
$\Omega_2\,=\, \omega_3+\sqrt{-1}\omega_1$ and $\Omega_3\,=\, \omega_1+\sqrt{-1}\omega_2$.

\subsection{Automorphisms of $\cM(G)$}\label{higgs-automorphisms}

Let $G$ be a complex semisimple Lie group and let $\Aut(G)$, $\Int(G)$ and
$\Out(G)$ be as defined in Section \ref{automorphisms-g}.

Let  $\theta\in \Aut(G)$, and let $(E,\varphi)$ be a $G$-Higgs bundle
over $X$. We define the $G$-Higgs bundle $(\theta(E),\theta(\varphi))$
by taking:
$$
\theta(E):=E\times_\theta G,
$$ 
and using  that $\theta(E)(\lieg)\cong E(\lieg)$
and the differential $d\theta\in \Aut(\lieg)$, associating
to the  section $\varphi$ of $E(\lieg)\otimes K$, 
a section of $\theta(E)(\lieg)\otimes K$, that we will
denote by $\theta(\varphi)$.
It is clear that if $\theta\in \Int(G)$, the Higgs bundle 
$(\theta(E),\theta(\varphi))$ is isomorphic to $(E,\varphi)$.
Hence  the group $\Out(G)$ acts on the set of isomorphisim classes
of $G$-Higgs bundles. It is easy to show that stability, semistability and
polystability are preserved by the action of $\Aut(G)$ and hence $\Out(G)$
acts on the moduli space of $G$-Higgs bundles $\cM(G)$.

We also have an  action of  $\mathbb{C}^*$ on the set of $G$-Higgs bundles
defined as follows. Let  $(E,\varphi)$ be a $G$-Higgs bundle and
$\lambda\in\mathbb{C}^*$, we define
$$
\lambda(E,\varphi):=(E,\lambda\varphi).
$$
It is immediate that stability, semistability and polystability are 
preserved by this action and this defines an action of  $\C^\ast$ on $\cM(G)$. 

Since the centre  $Z$ of $G$ is abelian, 
$H^1(X,Z)$  is a group and it is identified with
the set of isomorphism classes of principal
$Z$-bundles over $X$. 
For $\alpha,\beta\in H^1(X,Z)$ we consider first the $(Z\times Z)$-bundle
given by the fibre product
 $\alpha \times_X\beta$ over $X$. Now, via the group operation 
$Z\times Z\overset{m}{\rightarrow} Z$  we can associate to it the $Z$-bundle 
$$
\alpha\otimes \beta:=(\alpha \times_X\beta)\times_m Z,
$$
defining the group structure of  $H^1(X,Z)$.

Let $E$ be a principal $G$-bundle and let  $\alpha\in H^1(X,Z)$ be a principal
$Z$-bundle.  The fibre product $E\times_X\alpha$ has the structure of a 
principal $(G\times Z$)-bundle. 
Combining  this with the action of $Z$ on  $G$ given  by
multiplication $G\times Z \overset{m}{\rightarrow} G$,
by extension of structure group we  associate to
$E$ and $\alpha$ the principal $G$-bundle
$E\otimes \alpha:=(E\times_X\alpha)\times_m G$.

Since $G$ is semisimple $Z$ is finite, and the topological type of $E$ and
$E\otimes \alpha$ is the same.

It is clear that $E(\lieg)=(E\otimes\alpha)(\lieg)$ and hence we can associate
to a $G$-Higgs bundle   $(E,\varphi)$  and  
$\alpha\in H^1(X,Z)$ the $G$-Higgs bundle defined by
$$
\alpha\cdot(E,\varphi):=(E\otimes \alpha,\varphi).
$$

Again it is immediate to show  that this defines an action of
$H^1(X,Z)$ on the moduli space of $G$-Higgs bundles,
$\mathcal{M}(G)$. 

Every automorphism of $G$ restricts to an automorphism of $Z$. 
Inner automorphisms of $G$ induce the identity automorphism on $Z$, 
defining an action of $\Out(G)$ on $Z$, $\Out(G)\times Z\rightarrow Z$,
$(\sigma,\lambda)\mapsto\sigma(\lambda)$.  This
induces an action of $\Out(G)$ on $H^1(X,Z)$. We consider the
semidirect product $H^1(X,Z)\rtimes\Out(G)$ 
defined by

$$
(\beta,b)\cdot(\alpha,a)=(\beta\otimes
b(\alpha),ba).
$$

We thus notice that the group $H^1(X,Z)\rtimes\Out(G)$  acts on 
$\mathcal{M}(G)$ in the following way: if $(E,\varphi)$ is
a polystable $G$-Higgs bundle and $(\alpha,a)\in
H^1(X,Z)\rtimes\Out(G)$, then
$$
(\alpha,a)\cdot(E,\varphi)=(a(E)\otimes \alpha,a(\varphi)).
$$

We also have the semidirect product
$$
H^1(X,Z)\rtimes \Aut(X)
$$
defined by
$$
(\alpha_2,f_2)(\alpha_1,f_1)=(\alpha_2\otimes f_2^*\alpha_1,f_2f_1).
$$
One can show that we have an action of
$H^1(X,Z)\rtimes \Aut(X)$ on $\mathcal{M}(G)$  given by
$$
(\alpha,f)(E,\varphi)=(f^*E\otimes \alpha,f^*\varphi),
$$
where $(E,\varphi)\in\mathcal{M}(G)$ and $(\alpha,f)\in H^1(X,Z)\rtimes
\Aut(X)$.

Combining the preceding actions 
we obtain an action of the
group
$$
H^1(X,Z)\rtimes(\Out(G)\times \Aut(X))\times \C^*
$$
on $\mathcal{M}(G)$.


\begin{remark}
The above discussion can be extended to the case where $G$ is reductive. We
just need to replace  $H^1(X,Z)$ with $H^1(X,\underline{Z})$, where
$\underline{Z}$ be the sheaf of local $Z$-functions on $X$. Of course 
now the topological type of $E$ and
$E\otimes \alpha$ is the same only if $\alpha\in H^1(X,\underline{Z})$ 
has trivial topological class.
\end{remark}

\section{$G$-Higgs bundles and automorphisms of $G$}
\label{invo-higgs-bundles} 
In this section we will assume that $G$ is a connected complex semisimple 
Lie group. To study the  finite order automorphisms of $\cM(G)$  considered in 
the following sections, we introduce now  a class of Higgs bundles 
defined by an element $\theta\in \Aut_n(G)$. These involve
the subgroups  $G^\theta$ and $G_\theta$    
of $G$ defined in Sections \ref{normalizers} and \ref{finite-order}.

\subsection{Higgs bundles defined by automorphisms of $G$}\label{g-theta-higgs}

Let $\theta\in \Aut(G)$ of order $n$. 
Let $0\leq k\leq n-1$ and $\zeta_k:=\exp(2\pi i\frac{ k}{n})$. A  
{\bf $(G^\theta,\zeta_k)$-Higgs bundle} over $X$
is a pair $(E,\varphi)$ where
$E$ is  a principal $G^\theta$-bundle over $X$ and
$\varphi$ is a  section of $E(\lieg^k)\otimes K$, where $E(\lieg^k)$ 
is the bundle associated to $E$ via the representation 
$\iota^k:G^\theta \to \GL(\lieg^k)$ defined in Proposition
\ref{representations-n}, and $K$ is the canonical bundle on $X$.
If $\theta\in \Aut_2(G)$, we will also use the term 
{\bf $(G^\theta,\pm)$-Higgs bundle} over $X$
for a pair $(E,\varphi)$ where
$E$ is  a principal $G^\theta$-bundle over $X$ and
$\varphi$ is a  section of $E(\lieg^\pm)\otimes K$, where $E(\lieg^\pm)$ 
is the bundle associated to $E$ via the representation 
$\iota^\pm:G^\theta \to \GL(\lieg^\pm)$ defined in Proposition
\ref{representations} (see Remark \ref{double-notation}).

It is clear that a $(G^\theta,\zeta_0)$-Higgs bundle, and hence a  
 $(G^\theta,+)$-Higgs bundle when $\theta$ is of order 2, is simply a 
$G^\theta$-Higgs bundle as defined in Section \ref{g-higgs}. 
The notions of (semi)stability 
and polystability  for these objects are hence 
given in Definition \ref{def:semipoly} (see (2) in Remark 
\ref{rm:(semi)stability of bundles} for the case in
which $G^\theta$ is reductive). The cases of $(G^\theta,\zeta_k)$-Higgs 
bundles for $k>0$, in particular $(G^\theta,-)$-Higgs bundle
when $\theta$ is of order 2, requires  a new definition. In what follows 
we will give a  definition that accommodates all the cases simultaneously.

Let $\tau$ be a fixed compact conjugation of $G$ defining a maximal compact
subgroup $U:=G^\tau$, such that $\tau\theta=\theta\tau$. Then $\tau$ defines
also a conjugation on $G^\theta$, giving  the compact real form 
$U':=U\cap G^\theta$. Let $\lieu'$ be the Lie algebra of $U'$. Clearly
$\lieu'^\C=\lieg^\theta$. Given $s\in i\lieu'$, as explained in Section 
\ref{g-higgs},

\begin{equation}
P_s=\{g\in G^\theta\st e^{ts}ge^{-ts}\;\; \text{is bounded  as}\;\; t\to\infty\},
\end{equation}
is a parabolic subgroup of $G^\theta$.

For each summand in the decomposition (\ref{lie-decomposition}) we define

\begin{equation}\label{g_s}
\lieg^k_s=\{v\in\lieg^k\st\Ad(e^{ts})(v)\;\; \text{is bounded  as}\;\; t\to\infty\}.
\end{equation}

\begin{equation}\label{g_s^0}
\lieg^k_{s,0}=\{v\in\lieg^k \st\lim_{t\to\infty}\Ad(e^{ts})(v)=0\}.
\end{equation}

Clearly $\lieg^0_s=\liep_s$, the Lie algebra of $P_s$,
and $\lieg^0_{s,0}$ is the  Levi part $\liel_s$  of $\liep_s$.

As in Section \ref{g-higgs},  the only  parabolic subgroups of $G^\theta$ that 
we will consider are subgroups of form $P_s$ for some $s\in i \lieu'$.

We consider the non-degenerate 
invariant $\C$-bilinear pairing $\la\cdot,\cdot\ra$ 
on $\lieg^\theta$ induced by the Killing form on $\lieg$.


Let $\liez(\lieu')$  be the centre 
of $\lieu'$. In general  the natural stability criteria depend on a 
parameter $\alpha\in i \liez(\lieu')$  
(see \cite{garcia-prada-gothen-mundet} for details). 

We define the subalgebra $\lieu'_k$ as follows. 
Consider the decomposition 
$ \lieu' = \liez(\lieu') + [\lieu', \lieu'] $, 
and the representation 
$d\iota^k= \ad:\lieu'\to \End(\lieg^k)$. Let $
\liez'_k=\ker(d\iota^k|_{\liez(\lieu')})$  and 
take $\liez''_k$ such that $\liez(\lieu')=\liez'_k+\liez''_k$. Define 
the subalgebra 
$\lieu'_k := \liez''_k + [\lieu', \lieu']$. The subindex 
$k$ denotes that 
we have taken away the part of the centre $\liez(\lieu')$ acting 
trivially via  the isotropy representation $d\iota^k$.
Note that since $d\iota^0$ is the adjoint representation of 
$\lieu'$, $\liez'=\liez(\lieu')$ and 
$\lieu'_0 = [\lieu', \lieu']$. 

We have now all the ingredients to define stability.

\begin{definition}\label{def:g+-semipoly}
Let $\alpha\in i\liez(\lieu')$. A $(G^\theta,\zeta_k)$-Higgs bundle 
$(E,\varphi)$ over $X$ is:
 
{\bf $\alpha$-semistable} if for any $s\in i\lieu'$ and any reduction 
of structure group $\sigma$ of 
$E$ to $P_s$ such that $\varphi\in H^0(X,E_\sigma(\lieg^k_s)\otimes K)$
we have that $\deg(E)(\sigma,s)-\la\alpha,s\ra\geq 0$.

{\bf$\alpha$-stable} 
if for any $s\in i\lieu'_k$ and any reduction 
of structure group $\sigma$ of 
$E$ to $P_s$ such that $\varphi\in H^0(X,E_\sigma(\lieg^k_s)\otimes K)$
we have that $\deg(E)(\sigma,s)-\la\alpha,s\ra > 0$. 

 {\bf $\alpha$-polystable} if it is $\alpha$-semistable and 
$s\in i\lieu'_k$ and any reduction
of structure group $\sigma$ of 
$E$ to $P_s$ such that $\varphi\in H^0(X,E_\sigma(\lieg^k_s)\otimes K)$
and $\deg(E)(\sigma,s)-\la\alpha,s\ra = 0$ 
there is a further reduction of structure group $\sigma_L$ 
of $E_\sigma$  to the Levi subgroup $L_s$ of $P_s$ such that 
$\varphi\in H^0(X,E_{\sigma_L}(\lieg^k_{s,0})\otimes K)$.
\end{definition}

Let  $\cM^\alpha(G^\theta,\zeta_k)$  be the {\bf moduli space of polystable 
$(G^\theta,\zeta_k)$-Higgs bundles}. The construction of the moduli space 
of pairs given by Schmitt using Geometric Invariant Theory  
(cf. \cite{schmitt:2008}) 
applies also to these moduli spaces.

\begin{remark}
When $\alpha=0$ we will simply say stability to refer to $0$-stability and
will drope the index $\alpha$ denoting  the moduli space by  
$\cM(G^\theta,\pm)$.  
\end{remark}

\begin{remark}
As mentioned above, a $(G^\theta,\zeta_0)$-Higgs bundle $(E,\varphi)$ 
is simply a 
$G^\theta$-Higgs bundle, and, as pointed out in Remark 
\ref{rm:(semi)stability of bundles}, the parameter 
$\alpha$ is determined  by the topology of $E$. This is definitely not the 
case for $(G^\theta,\zeta_k)$-Higgs bundles for $k>0$, where $\alpha$ can 
take different values. When $\theta$ is of order $2$, this happens precisely 
when the real form defined by $\sigma:=\theta\tau$ has a factor of 
Hermitian type. In this case, the possible values of the parameter $\alpha$ are
governed by  a Milnor--Wood type  inequality  
(see \cite{biquard-garcia-prada-rubio}).
 \end{remark}

As for $G$-Higgs bundles, there are relevant Hitchin equations
linked to a $(G^\theta,\zeta_k)$-Higgs bundle.
With the same notation as in Theorem \ref{higgs-hk}, One has the 
following.

\begin{theorem} \label{higgs-g+-hk-connected}
Let $(E,\varphi)$ be a $(G^\theta,\zeta_k)$-Higgs bundle. 
There is a reduction $h$ of  structure group of $E$ from $G^\theta$ to 
$U'$ that satisfies the Hitchin equation
  $$
  F_h -[\varphi,\tau_h(\varphi)]= -i\alpha\omega 
  $$
  if and only if $(E,\varphi)$ is $\alpha$-polystable.
\end{theorem}

The proof of Theorem \ref{higgs-g+-hk-connected} is given in 
\cite{garcia-prada-gothen-mundet} (the $(G^\theta,\zeta_0)$-case 
is essentially given by Theorem \ref{higgs-hk}). 

Recall from Section \ref{finite-order}
the short exact sequence (\ref{g-theta-n}):

\begin{equation} \label{exact-sequence}
1 \lra G^\theta\lra G_\theta \lra \Gamma_\theta \lra 1.
\end{equation}

Similarly to $(G^\theta,\zeta_k)$-Higgs bundles we can consider
$(G_\theta,\zeta_k)$-Higgs bundles.
A  {\bf $(G_\theta,\zeta_k)$-Higgs bundle} over $X$
is a pair $(E,\varphi)$ where
$E$ is  a principal $G_\theta$-bundle over $X$ and
$\varphi$ is a  section of $E(\lieg^k)\otimes K$, where $E(\lieg^k)$ 
is the bundle associated to $E$ via the representation 
$\iota_k:G_\theta \to \GL(\lieg^k)$ defined in Proposition
\ref{representations-n}. Stability can be defined as in 
Definition \ref{def:g+-semipoly}, where now the parabolic subgroups defined 
by elements  $s\in i\lieu'$ are parabolic subgroups of $G_\theta$. Note
that the maximal compact subgroup of $G_\theta$ is the group 
$U'':=G_\theta\cap U$, and its Lie algebra is also $\lieu'$.
We denote the moduli space of $\alpha$-polystable 
such objects as $\cM^\alpha(G_\theta,\zeta_k)$ --- or simply
$\cM(G_\theta,\zeta_k)$ if $\alpha=0$.

A correspondence theorem analogous to Theorem \ref{higgs-g+-hk-connected} also 
exists for these objects (\cite{garcia-prada-gothen-mundet}):

\begin{theorem} \label{higgs-g+-hk}  
Let $(E,\varphi)$ be a $(G_\theta,\zeta_k)$-Higgs bundle. 
There is a reduction $h$ of  structure group of $E$ from $G_\theta$ to 
$U''$ that satisfies the Hitchin equation
  $$
  F_h -[\varphi,\tau_h(\varphi)]= -i\alpha\omega 
  $$
  if and only if $(E,\varphi)$ is $\alpha$-polystable.
\end{theorem}

If $\theta\in \Aut_2(G)$, we will also use the term 
{\bf $(G_\theta,\pm)$-Higgs bundle} over $X$
for a pair $(E,\varphi)$ where
$E$ is  a principal $G_\theta$-bundle over $X$ and
$\varphi$ is a  section of $E(\lieg^\pm)\otimes K$, where $E(\lieg^\pm)$ 
is the bundle associated to $E$ via the representation 
$\iota^\pm:G_\theta \to \GL(\lieg^\pm)$ defined in Proposition
\ref{representations} (see Remark \ref{double-notation}).

In the case of order $2$ we will also use the notations

$$
\begin{array}{ccc}
\cM^\alpha(G^\theta,+)& = & \cM^\alpha(G^\theta,\zeta_0),\\ 
\cM^\alpha(G^\theta,-) & = &\cM^\alpha(G^\theta,\zeta_1), \\
\cM^\alpha(G_\theta,+) & = & \cM^\alpha(G_\theta,\zeta_0),\\ 
\cM^\alpha(G_\theta,-) &  = &\cM^\alpha(G_\theta,\zeta_1).
\end{array}
$$

And similarly when $\alpha=0$. 
For reasons that will be apparent in Section \ref{higgs-reps}, 
$(G^\theta,-)$-Higgs bundles and $(G_\theta,-)$-Higgs bundles  
are referred in the literature as $G^\sigma$-Higgs bundles
and $G_\sigma$-Higgs bundles, where $\sigma=\theta\tau$ and 
$G^\sigma$ and $G_\sigma$ are defined as in Section \ref{normalizers}.

Consider the map
\begin{equation}\label{gamma-map}
\gamma_\theta: H^1(X,\underline{G_\theta}) \lra H^1(X,\Gamma_\theta)
\end{equation}
induced by the homomorphism   $G_\theta \lra \Gamma_\theta$.
This map associates to every $G_\theta$-bundle an invariant in 
$H^1(X,\Gamma_\theta)$.
If we fix $\gamma \in H^1(X,\Gamma_\theta)$ we can consider the moduli
space $\cM^\alpha_\gamma(G_\theta,\zeta_k)$ of $\alpha$-polystable elements  
$(E,\varphi)$ 
as above such that $\gamma_\theta(E)=\gamma$ (or simply
$\cM_\gamma(G_\theta,\zeta_k)$ for the moduli space $0$-polystable elements). 
Higgs bundles  $(E,\varphi)$ such that
$\gamma(E)=e \in  H^1(X,\Gamma_\theta)$ --- the identity element --- have the 
property that the structure  group of $E$  reduces to $G^\theta$.

\begin{proposition}\label{quotient-Gamma}
Let $\theta\in \Aut_n(G)$ and let $\cM_e(G_\theta,\zeta_k)$ be the moduli space 
polystable $(G_\theta,\zeta_k)$-Higgs bundles  $(E,\varphi)$ such that
$\gamma(E)=e \in  H^1(X,\Gamma_\theta)$.
The group $\Gamma_\theta$ acts on $\cM^\alpha(G^\theta,\zeta_k)$, and 
$$
\cM^\alpha_e(G_\theta,\zeta_k)= \cM^\alpha(G^\theta,\zeta_k)/\Gamma_\theta.
$$
\end{proposition}

\begin{proof}

Choose a section $s:\Gamma_\theta \to G_\theta$ of the extension
(\ref{exact-sequence}), i.e. $s$ is a map such that $p\circ s=\Id_{\Gamma_\theta}$,
where $p:G_\theta \to \Gamma_\theta$ is the natural projection in 
(\ref{exact-sequence}).
This section defines a map $\Gamma_\theta \to \Aut(G)$ 
given by $\gamma\mapsto s(\gamma)gs(\gamma)^{-1}$, for $\gamma\in \Gamma_\theta$
and $g\in G$. This is not in general
a homomorphism, but descends to a homomorphism
\begin{equation}\label{Gamma-action} 
\Gamma_\theta \to \Aut(G)/\Int(G^\theta),
\end{equation}
where $\Int(G^\theta)$ here  is the subgroup of $\Aut(G)$ consisting 
of conjugations in $G$  by elements of $G^\theta$.  In particular
we obtain the characteristic homomorphism 
$\Gamma_\theta\to \Out(G^\theta)=\Aut(G^\theta)/\Int(G^\theta)$, which then
defines an action on $H^1(X,\underline{G_\theta})$. So if $E$ is
$G^\theta$-bundle and $\gamma\in \Gamma_\theta$, we have another
$G^\theta$-bundle $\gamma(E)$. 
To obtain a section of 
$\gamma(E)(\lieg^k)\otimes K$ out of the Higgs field
 $\varphi\in E(\lieg^k)\otimes K$ we observe that $\Gamma_\theta$ acts
 also on $\lieg^k$ via the homomorphism (\ref{Gamma-action}) 
in such  a way that
$\iota^k: G^\theta\to \GL(\lieg^k)$ is $\Gamma_\theta$-equivariant. 
The preservation of semistability, stability
and polystability under this action is clear.

Let now $(E,\varphi)\in \cM^\alpha(G^\theta,\zeta_k)$. We can define a 
$(G_\theta,\pm)$-Higgs bundle by extending the structure group of 
$E$ to $G_\theta$ obtaining a $G_\theta$-bundle $E_{G_\theta}$. The extended
Higgs field $\varphi_{G_\theta}$ is in fact $\varphi$ itself since
$E(\lieg^k)=E_{G_\theta}(\lieg^k)$.  This map preserves polystability
as can be seen easily from Theorem \ref{higgs-g+-hk-connected}. 
Notice that
a reduction of the structure group of a $G^\theta$-bundle $E$ to the maximal
compact subgroup $U'$ is a smooth section of the bundle 
$E(G^\theta/U')$ and since these symmetric spaces $G^\theta/U'$
and 
$G_\theta/U''$ coincide $E(G^\theta/U')=
E_{G_\theta}(G_\theta/U'')$ --- recall that  
$U''$ is the maximal compact subgroup of $G_\theta$.
Since the different reductions of a  $G_\theta$-bundle whose with 
trivial class in $H^1(X,\Gamma_\theta)$ are given by  
$H^0(X,\Gamma_\theta)\cong \Gamma_\theta$ the map defined above 
descends to give the desired  map
 $\cM^\alpha_e(G_\theta,\zeta_k)= \cM^\alpha(G^\theta,\zeta_k)/\Gamma_\theta$.
\end{proof}

There are appropriate notions of simplicity for $(G^\theta,\zeta_k)$- and 
$(G_\theta,\zeta_k)$-Higgs bundles (see Remark \ref{non-connected-simplicity} and \cite{garcia-prada-gothen-mundet2}).


From Proposition \ref{smoothness} a stable 
and simple  $(G^\theta,\zeta_0)$-Higgs bundle 
(resp. $(G_\theta,\zeta_0)$-Higgs bundle) 
defines a smooth point in $\cM(G^\theta,\zeta_0)$ 
(resp.  $\cM(G_\theta,\zeta_0)$). In contrast, for  a 
$(G^\theta,\zeta_k)$-Higgs bundle and a $(G_\theta,\zeta_k)$-Higgs bundle
to define a smooth point,  
in addition to stability  and simplicity, the vanishing of a certain
obstruction living in the second degree hypercomology group of 
the deformation complex  defined by the Higgs bundle (see
\cite{garcia-prada-gothen-mundet}).

\subsection{$(G^\theta,\zeta_k)$- and $(G_\theta,\zeta_k)$-Higgs bundles versus 
$G$-Higgs bundles}

Let $(E,\varphi)$ be a $(G^\theta,\zeta_k)$-Higgs bundle. By extending the 
structure  group to $G$ we obtain a $G$-bundle $E_G$ and since 
$E_G(\lieg)=\oplus_k E(\lieg^k)$, we can associate to  the Higgs field 
$\varphi$  a section $\varphi_G$ of $E_G(\lieg)\otimes K$ by taking  $0$ in 
$E(\lieg^j)$ component for $j\neq k$. The $G$-Higgs bundle $(E_G,\varphi_G)$ will be
referred as the {\bf extension} of $(E,\varphi)$.
Also, if $(E,\varphi)$ is a $G$-Higgs bundle we say 
that it reduces to a $(G^\theta,\zeta_k)$-Higgs bundle if $E$ reduces to 
a $G^\theta$-bundle  $E_{G^\theta}$ and $\varphi$ takes values in 
$E_{G^\theta}(\lieg^k)\otimes K$, in which case we rename it
$\varphi_k$. We say that $(E_{G^\theta},\varphi_k)$ is a {\bf reduction} 
of $(E,\varphi)$. We consider the similar construction  
also for $(G_\theta,\zeta_k)$-Higgs bundles.
Recall that we  say that a $(G^\theta,\zeta_k)$-Higgs bundle 
($(G_\theta,\zeta_k)$-Higgs bundle) is polystable when it is $0$-polystable.

\begin{proposition} \label{extension-reduction-2}
Let  $\theta\in \Aut_n(G)$. 

\begin{enumerate}

\item
Let $(E,\varphi)$ be a  polystable  $(G^\theta,\zeta_k)$-Higgs bundle. Then the
corresponding $G$-Higgs bundles $(E_G,\varphi_G)$ is also polystable. 
We thus have a map $\cM(G^\theta,\zeta_k)\lra \cM(G)$.

\item
 Let $(E,\varphi)$ be a   $G$-Higgs bundle which reduces to a 
$(G^\theta,\zeta_k)$-Higgs bundle 
$(E_{G^\theta},\varphi_k)$. 
Then if  $(E,\varphi)$ is (semi,poly)stable, $(E_{G^\theta},\varphi_k)$
is also (semi,poly)polystable.

\item Let  $\theta,\theta'\in \Aut_n(G)$ such that 
$\theta'=\Int(g)\theta\Int(g^{-1})$, with $g\in G$.
Then $\Int(g)$ gives rise to a canonical isomorphism of 
$\cM(G^\theta,\zeta_k)$ with   $\cM(G^{\theta'},\zeta_k)$.
Since the action of $\Int(g)$ in $\cM(G)$ is 
trivial we have a commutative diagram

$$
\begin{array}{ccc}
\cM(G^\theta,\zeta_k)& \to & \cM(G)\\
\downarrow & \nearrow &        \\
\cM(G^{\theta'},\zeta_k).         &          &  
\end{array}
$$
\end{enumerate}

\end{proposition}

\begin{proof}
That the polystability of a $(G^\theta,\zeta_k)$-Higgs bundle $(E,\varphi)$
implies that of  $(E_G,\varphi_G)$ follows from 
Theorems \ref{higgs-hk}  and \ref{higgs-g+-hk}, together with the 
observation that $E(G^\theta/U')\subset E_G(G/U)$ and hence
a reduction of structure group of $E$ to $U'$ defines a reduction of
structure group of $E_G$ to $U$. The fact that the former satisfies the Hitchin
equation for $\alpha=0$ implies that the latter satisfies the Hitchin 
equation in Theorem  \ref{higgs-hk}  implying the polystability of
$(E_G,\varphi_G)$.

To prove (2) suppose that  $(E_{G_\theta},\varphi_k)$ is not semistable.
Following Definition \ref{def:g+-semipoly}, there is an 
$s\in \lieu'$ defining a parabolic subgroup $P_s\in G^\theta$, 
and a reduction $\sigma$ of
$E_{G^\theta}$ to a $P_s$-bundle such that $\deg(E_{G^\theta})(s,\sigma)< 0$.
But $s$ defines also a parabolic subgroup $\widetilde{P}_s$ of $G$, and 
the reduction $\sigma$ defines a reduction  $\widetilde{\sigma}$ of $E$ to
$\widetilde{P}_s$ such that  
$\deg(E)(s,\widetilde{\sigma})=\deg(E_{G^\theta})(s,\sigma)$. The result
follows now. The same argument applies to stability and polystability.

(3) follows from the fact that $\Int(G)$ acts trivially on $\cM(G)$ as
explained in Section \ref{automorphisms}.
\end{proof}

Similarly to Proposition \ref{extension-reduction-2}, we have the following.

\begin{proposition} \label{extension-reduction}
Let  $\theta\in \Aut_n(G)$. 

\begin{enumerate}

\item
Let $(E,\varphi)$ be a  polystable  $(G_\theta,\zeta_k)$-Higgs bundle. Then the
corresponding $G$-Higgs bundles $(E_G,\varphi_G)$ is also polystable. 
We thus have a map $\cM(G_\theta,\pm)\lra \cM(G)$.

\item
 Let $(E,\varphi)$ be a   $G$-Higgs bundle which reduces to a 
$(G_\theta,\zeta_k)$-Higgs bundle 
$(E_{G_\theta},\varphi_k)$. 
Then if  $(E,\varphi)$ is (semi,poly)stable, $(E_{G_\theta},\varphi_k)$
is also (semi,poly)polystable.

\item Let  $\theta,\theta'\in \Aut_n(G)$ such that 
$\theta'=\Int(g)\theta\Int(g^{-1})$, with $g\in G$.
Then $\Int(g)$ gives rise to a canonical isomorphism of 
$\cM(G_\theta,\zeta_k)$ with   $\cM(G_{\theta'},\zeta_k)$.
Since the action of $\Int(g)$ in $\cM(G)$ is 
trivial we have a commutative diagram

$$
\begin{array}{ccc}
\cM(G_\theta,\zeta_k)& \to & \cM(G)\\
\downarrow & \nearrow &        \\
\cM(G_{\theta'},\zeta_k).         &          &  
\end{array}
$$
\end{enumerate}

\end{proposition}

More can be proved if $\theta\in \Aut_2(G)$:

\begin{proposition}\label{stronger-order-2}
Let $\theta\in \Aut_2(G)$. Then the map $\cM(G^\theta,\pm)\lra \cM(G)$
is a $|\Gamma_\theta|:1$-map, where  $|\Gamma_\theta|$ is the order of $\Gamma_\theta$.
Moreover, the map 
$\cM(G_\theta,\pm)\lra \cM(G)$ is  injective.
\end{proposition}

\begin{proof}
The injectivity of the map $\cM(G_\theta,\pm)\lra \cM(G)$
follows from the fact that, as shown in Proposition 
\ref{theta-normalizer},  $G_\theta$ is the normalizer of $G^\theta$ in 
$G$. The first statement follows now from  Proposition 
\ref{quotient-Gamma}.
\end{proof}

\begin{remark}
(1) Note  that the relation in Propositions  \ref{extension-reduction-2} and
 \ref{extension-reduction} 
between  $(G^\theta,\zeta_k)$-Higgs bundles (resp. $(G_\theta,\zeta_k)$-Higgs bundles)
and $G$-Higgs bundles 
applies when the stability parameter $\alpha$ for the 
$(G^\theta,\zeta_k)$-Higgs bundles (resp. $(G_\theta,\zeta_k)$-Higgs bundles)
is $0$. Although the other values of $\alpha$ do not relate in general 
to polystable  $G$-Higgs bundles, they turn out to play an important 
role in the study of the topology of the moduli space for $\alpha=0$.

(2) In the more general case in which $G$ is reductive, there is an analogous 
results to Propositions \ref{extension-reduction-2} and 
\ref{extension-reduction}, 
for $\alpha$-polystable
objects, where the element $\alpha$  in the centre of $\lieg$  is determined 
by the topological class of the $G$-bundle.

\end{remark}


\section{Finite order automorphisms of $\cM(G)$}\label{higher-order-auto}

In this section $G$ is a connected complex semisimple Lie group, 
$X$ is a compact Riemann surface, $\cM(G)$ is  the moduli space
of $G$-Higgs bundles over $X$. 
Our goal is to study fixed 
points of automorphisms of $\cM(G)$
defined by finite order elements in $H^1(X,Z)\rtimes\Out(G)\times \C^*$. 
These  involve the  moduli spaces 
$\cM(G^\theta,\zeta_k)$ and  $\cM(G_\theta,\zeta_k)$ defined in Section 
\ref{g-theta-higgs}
where $\theta\in \Aut_n(G)$, $\zeta_k:=\exp(2\pi i\frac{k}{n})$, 
and $G^\theta$ and $G_\theta$ are the
subgroups of $G$ defined in Sections \ref{normalizers} and \ref{finite-order}.
We will denote by
$\cM(G)_{ss}$ the subvariety of stable and simple points of $\cM(G)$
and by $\widetilde{\cM}(G^\theta,\zeta_k)$ and  
$\widetilde{\cM}(G_\theta,\zeta_k)$ 
the images  of $\cM(G^\theta,\zeta_k)$ and  $\cM(G_\theta,\zeta_k)$ 
in $\cM(G)$ under the maps defined  in Propositions 
\ref{extension-reduction-2} and \ref{extension-reduction} respectively.

\begin{proposition}\label{ext-iso}
Let $\theta\in \Aut_n(G)$, and 
let $(E,\varphi)$ be a $(G^\theta,\zeta_k)$-Higgs bundle
and  $(E_G,\varphi_G)$ be the corresponding
extension to a $G$-Higgs bundle. Then
$(E_G,\varphi_G)$ is isomorphic to $(\theta(E_G),\zeta_k\theta(\varphi_G))$. 
\end{proposition}
\begin{proof}
The bundle $E_G$ is obtained from $E$ by the extension $G^{\theta} 
\subset G$ of structure group. The extension obtained by further 
composition with $\theta $ gives $\theta (E_G)$. Since $\theta $ is 
the identity on $G^{\theta }$ there is a canonical isomorphism of $E_G$  
with $\theta (E_G)$. If $\varphi $ takes values in the $\zeta_k$-eigenspace 
of $\ad(\theta )$, it gives rise to a Higgs field on $E_G$ on which 
$\ad(\theta )$ acts as $\zeta_k$.
 \end{proof}

\begin{proposition}\label{fixed-iota-theta}
Let $\theta\in \Aut_n(G)$, 
and let  $(E, \varphi )$ be  a simple $G$-Higgs 
bundle  isomorphic to  $(\theta(E), \zeta_k\theta(\varphi))$.
Then, except for $\theta\in \Int(G)$ and $k=0$, we have the following.
 
(1) The structure  group of $E$ can be reduced to $G^{\theta'}$ with 
$\theta'=\Int(s)\theta$ and $s\in S^n_\theta$, where $S^n_\theta$ as defined
in Proposition \ref{cliques-n}. Moreover $s$ is unique up to the action of
$G$ and $Z$ defined in Proposition \ref{cliques-n}.

(2) The Higgs field  $\varphi$  takes values in the
$\zeta_k$-eigenspace  of the automorphism of 
$\lieg$ defined by $\theta'$. In other words, $(E,\varphi)$ reduces to a
$(G^{\theta'},\zeta_k)$-Higgs bundle.
\end{proposition}

\begin{proof}
Suppose that $(E,\varphi)$ is isomorphic to 
$(\theta(E),\zeta_k\theta(\varphi))$. This means that there is an isomorphism 
of $A: E \to \theta(E)$ such that $\Ad(A)(\varphi)=\zeta_k \theta(\varphi)$.
From Proposition \ref{twisted-bundle}, 
the isomorphism $A$ can be identified with
a $\theta$-twisted automorphism of $E$,
as defined in Section \ref{twisted-automorphisms}. Since $\theta$ is of order 
$n$, $A^n$ defines  an  automorphism  of $(E,\varphi)$, and hence $A^n=z$ 
with $z\in Z$ since  $(E,\varphi)$ is simple.
From Proposition \ref{twisted-orbit}  
the  function $f_A$ given in  (\ref{function-A}) maps
$E$ onto a single orbit of the set $S^n_\theta$ defined 
in Proposition  \ref{cliques-n} under the action of $G$ defined 
there. 
If $A':E\to \theta (E)$ is another isomorphism such that 
$\Ad(A')(\varphi ) = \zeta_k \theta (\varphi )$ then $A^{-1}A' = z'$ for some 
$z'\in Z$ and the orbit defined by $A'$ is given by multiplication by 
$z'$. We thus obtain a unique orbit in $S^n_{\theta }$ under the action 
of $G\times Z$.

Now, by (2) in Proposition \ref{twisted-orbit},  every element $s$ 
in this  $Z\times G$-orbit defines a reduction  of structure group of
$E$ to $G^{\theta'}$, where $\theta'=\Int(s)\theta$, and $G^{\theta'}$ 
is the subgroup of $G$ of fixed points under $\theta'$.
This concludes the proof of (1).

To prove (2), let $\theta':=\Int(s)\theta$ for  $s\in S^n_\theta$ in the 
$Z\times G$-orbit defined in (1). The bundle $E$ reduces then to a
$G^{\theta'}$-bundle $E_{G^{\theta'}}$ and the adjoint
bundle decomposes as 
$$
E(\lieg)=\oplus_k E_{G^{\theta'}}(\lieg^k),
$$
where $\lieg=\oplus_k \lieg^k$ is the decomposition of $\lieg$ in
$(\zeta_k)$-eigenspaces of $\theta$ (see Proposition \ref{representations}). 
It is clear that $\Ad(A)(\varphi)=\zeta_k\theta(\varphi)$ is equivalent to
$\varphi\in H^0(X, E_{G^{\theta'}}(\lieg^k))$.
\end{proof}

Here is one of the main results of this paper.

\begin{theorem}\label{moduli-n-automorphism}
Let $a\in \Out_n(G)$ and $\zeta_k=\exp(2\pi i\frac{k}{n})$.
Consider the automorphism 

$$
\begin{aligned}
\iota(a,\zeta_k): \cM(G) & \to \cM(G) \\
(E,\varphi) & \mapsto (a(E),\zeta_k a(\varphi)).
  \end{aligned}
$$
Then
\begin{enumerate}

\item 
$$
\bigcup_{[\theta]\in \cl_n^{-1}(a)}
\widetilde{\cM}(G^\theta,\zeta_k)\subset \cM(G)^{\iota(a,\zeta_k)},
$$  

\item
$$
\cM(G)_{ss}^{\iota(a,\zeta_k)}
\subset 
\bigcup_{[\theta]\in \cl_n^{-1}(a)}
\widetilde{\cM}(G^\theta,\zeta_k)
$$
(except for $\iota(1,1)$),
\end{enumerate}
where $\cl_n:\Aut_n(G)/{\sim} \to \Out_n(G)$ is defined  in 
Proposition \ref{cliques-n}, and 
$\cl_n^{-1}(a)$ is, by 
Proposition \ref{n-clique-cohomology}, in bijection with
$H^1_a(\Z/n,\Ad(G))$.
\end{theorem}

\begin{proof}
Let $a\in \Out_n(G)$ and $(E,\varphi)$ be a $G$-Higgs
bundle. Recall from Section \ref{higgs-automorphisms}
that $(E,\varphi) \cong (a(E),\zeta_k a(\varphi))$ is equivalent to
$(E,\varphi) \cong (\theta(E),\zeta_k \theta(\varphi))$ for
any $\theta\in \Aut_n(G)$, such that $\pi(\theta)=a$, where 
$\pi:\Aut_n(G)\to \Out_n(G)$ is the natural projection. 

Let  $\theta\in \Aut_n(G)$, and $(E,\varphi)\in \cM(G^\theta,\zeta_k)$. 
From (1) in Proposition \ref{extension-reduction-2} 
the image of $(E,\varphi)$ in $\cM(G)$ is given
by the extended $G$-Higgs bundle $(E_G,\varphi_G)$. 
From Proposition \ref{ext-iso}, 
$(E_G,\varphi_G)$ is isomorphic to $(\theta(E_G),\zeta_k\theta(\varphi_G))$, 
showing that $(E_G,\varphi_G)\in 
\cM(G)^{\iota(a,\zeta_k)}$.
To complete the proof of (1), we apply (3) in Proposition 
\ref{extension-reduction-2}, which
says that if $\theta\sim\theta'$, $\cM(G^\theta,\zeta_k)$ and
$\cM(G^{\theta'},\zeta_k)$ are isomorphic and their images in 
$\cM(G)$ coincide.   

The proof of (2) follows from Proposition \ref{fixed-iota-theta} combined with
(2) and (3) of Proposition \ref{extension-reduction-2} and
(3) of Proposition \ref{cliques-n}.
\end{proof}

\begin{remark}
If $a=1$, the map $\iota(a,\zeta_0)$ is the identity map of $\cM(G)$ and 
(1) follows trivially from  Proposition \ref{extension-reduction-2}.
\end{remark}

\begin{remark}\label{empty-fixed-locus}
Of course the fixed point 
locus could be  entirely contained in the strictly polystable part, and 
hence $\cM(G)_{ss}^{\iota(a,\zeta_k)}$  be empty. 
\end{remark}

The case $a=1$ in Theorem \ref{moduli-n-automorphism} reduces to the study
of {\bf cyclic} Higgs bundles done by Simpson in \cite{simpson:2009} for 
$\SL(n,\C)$, and 
recently developed by Collier in general \cite{collier}. It should
be interesting to compare the different points of view.
In this situation we have the following.

\begin{theorem}\label{moduli-n-automorphism-cyclic}
Let $\zeta_1=\exp(\frac{2\pi i}{n})$.
Consider the automorphism 

$$
\begin{aligned}
\iota: \cM(G) & \to \cM(G) \\
(E,\varphi) & \mapsto E,\zeta_1 \varphi).
  \end{aligned}
$$
Then
\begin{enumerate}

\item 
$$
\bigcup_{[\theta]\in \Int_n(G)/\sim}
\widetilde{\cM}(G^\theta,\zeta_1)\subset \cM(G)^\iota,
$$  

\item
$$
\cM(G)_{ss}^\iota
\subset 
\bigcup_{[\theta]\in \Int_n(G)/\sim}
\widetilde{\cM}(G^\theta,\zeta_1)
$$
\end{enumerate}
\end{theorem}

Notice that in this case, by Proposition \ref{cliques-n}, 
$\cl_n^{-1}(1)=\Int_n(G)/\sim$.

\subsection{Automorphisms defined by elements in 
$H^1(X,Z)\rtimes\Out(G)$}

Let $(\alpha,a)\in (H^1(X,Z)\rtimes\Out(G))_n$ and 
$\zeta_k=\exp(2\pi i\frac{k}{n})$.
With such an  element  we can define
the automorphism
\begin{equation}\label{involution-a}
   \begin{aligned}
\iota(a,\alpha,\zeta_k): \cM(G) & \to \cM(G) \\
(E,\varphi) & \mapsto (a(E)\otimes\alpha,\zeta_k a(\varphi)).
  \end{aligned}
\end{equation}

To describe the fixed points of these automorphisms,
recall from Sections \ref{normalizers} and \ref{finite-order} 
that given  $\theta\in \Aut_n(G)$ 
we have subgroups  $G^\theta$ and $G_\theta$ of $G$
and the exact sequence (\ref{g-theta-n}).  We have the following.  

\begin{proposition}\label{ext-twisted-iso}
Let $\theta\in \Aut_n(G)$ and let
$$
c_\theta: H^1(X,\Gamma_\theta) \to H^1(X,Z)
$$
be the map induced by the injective
homomorphism  $\tilde{c}: \Gamma_\theta \to Z$ defined in   
Proposition  \ref{c-map-n}. Let  
$\gamma_\theta: H^1(X,\underline{G_\theta}) \lra H^1(X,\Gamma_\theta)$ be the
map defined in (\ref{gamma-map}).
We have the following:

(1) The map $c_\theta$ is injective.

(2) Let $(E,\varphi)$ be a $(G_\theta,\zeta_k)$-Higgs bundle with 
$\gamma_\theta(E)=\gamma$, and let
$\alpha:=c_\theta(\gamma)$. Let $(E_G,\varphi_G)$ be the 
extension of $(E,\varphi)$  to a $G$-Higgs bundle.  
Then $(E_G,\varphi_G)$ is isomorphic to 
$(\theta(E_G)\otimes\alpha,\pm \theta(\varphi_G))$.
\end{proposition}
\begin{proof}
Both $\Gamma _{\theta }$ and $Z$ are finite groups and hence the
first cohomology of $X$ with coefficient in these groups are given by
$\Hom(\pi_1(X) ,\Gamma _{\theta })$ and $\Hom(\pi_1(X) ,Z)$ respectively. 
Hence $c_{\theta }$ is injective.

(2) Let $(E, \varphi )$ be a $(G_{\theta }, \zeta_k )$-Higgs bundle. Under 
the natural homomorphisms $G_{\theta } \to \Gamma _{\theta } \to Z$ we 
get, by extension of structure group a $Z$-principal bundle $\alpha $. On 
the other hand, the inclusion of $G_{\theta }$ in $G$ gives rise to a 
$G$-Higgs bundle $(E_G, \varphi _G)$. Clearly, $\theta(E_G)$ is obtained 
from $E$ by extension of structure group $\theta\circ \iota $ where $\iota $ 
is the inclusion of $G_{\theta }$ in $G$. But $\theta \circ \iota = \iota \circ 
\theta|_{G_{\theta}}$. Now, if $x\in G_{\theta }$, we have $\theta (x) = 
xx^{-1}\theta (x) = x (\tilde{c}\circ \eta )(x)$ where $\eta $ is the 
natural surjection $G_{\theta } \to \Gamma _{\theta }$ and ${\tilde c}$ 
is the inclusion of $\Gamma _{\theta }$ in $Z$. Now $E_G \otimes \alpha$ 
is obtained by extension of structure group by the multiplication map 
$G \times Z \to G$. Hence $\theta (E_G)\otimes \alpha $ is got from $E$ by 
the extension of structure group by $g \to (\theta(g), \theta (g)^{-1}g) 
\to g$.  This proves our assertion.
\end{proof}

The following generalises Proposition \ref{fixed-iota-theta}.

\begin{proposition}\label{fixed-iota-theta-alpha}
Let $\theta\in \Aut_n(G)$ and $\alpha\in H^1(X,Z)$ such that
$\alpha\theta(\alpha)\cdots \theta^{n-1}(\alpha)=1$, and let  
$(E, \varphi )$ be  a simple $G$-Higgs 
bundle  isomorphic to  $(\theta(E)\otimes\alpha, \zeta_k\theta(\varphi))$.  
Then, except for the case $\theta\in \Int_n(G)$, $\alpha=1$ and
$k=0$, we have the following.
 
(1) The structure  group of $E$ can be reduced to $G_{\theta'}$ with 
$\theta'=\Int(s)\theta$ and $s\in S^n_\theta$, where $S^n_\theta$ as defined
in Proposition \ref{cliques-n}, with $s$  unique up to the action of
$Z\times G$ defined in Proposition \ref{cliques-n}. Moreover, if
$\gamma\in H^1(X,\Gamma_{\theta'})$ is the class of the reduced 
$G_{\theta'}$-bundle under the map defined in (\ref{gamma-map}), 
then $c_{\theta'}(\gamma)=\alpha$, with $c_{\theta'}$ is as defined in
Proposition \ref{ext-twisted-iso}.

(2) The Higgs field  $\varphi$  takes values in the
$\zeta_k$-eigenspace  of the automorphism of 
$\lieg$ defined by $\theta'$. In other words, $(E,\varphi)$ reduces to a
$(G_{\theta'},\zeta_k)$-Higgs bundle.
\end{proposition}

This proposition is a consequence of the following more general statement.

\begin{proposition}
Let $\theta \in \Aut_n(G)$ and $\Gamma $ a 
central subroup of $G$ invariant under $\theta $. Let $A$ be  
an isomorphism  of the Higgs bundle $(E, \varphi )$ with $(\theta (E) 
\otimes \alpha , \zeta_k\theta (\varphi ))$, where $\alpha $ 
is a $\Gamma$-bundle such that $\alpha\theta (\alpha)\cdots 
\theta^{n-1}(\alpha)=1$. 
Also assume that the composite of $\theta (A)$ with $A$ gives rise to an 
automorphism of $(E,\varphi )$ which is induced by an element of 
$\Gamma $. Then we have the following.

(1) If $t\in G$ such that $t\theta (t)$ is in the centre $Z$, then 
$\theta ' = \Int(t)\theta $ is also an involution of $G$. With our 
assumption above, $E$ can be reduced to a $G'$-bundle $E'$ where $G' = 
G_{\theta '} = \{ g\in G: g\theta '(g)^{-1} \in \Gamma \} $ for 
some $t$ satisfying $t\theta (t) \in \Gamma $. Let $\Gamma ' =\Gamma 
_{\theta '}$ be the image of $G'$ in $\Gamma $ given by the homomorphism 
$g\to g\theta'(g)^{-1}$. Then by extension of structure group by the 
homomorphism $G'\to \Gamma '\subset \Gamma $ the bundle $E'$ gives rise 
to a $\Gamma '$-bundle $\alpha ' $ and a $\Gamma $-bundle which is 
isomorphic to $\alpha $.

(2) The Higgs field $\varphi $ takes values in the $\zeta_k$-eigenspace of 
${\lieg}$ defined by $\theta '$. In other words, $(E, \varphi )$ 
reduces to a $(G', \zeta_k)$-bundle.

(3) The isomorphism $A$ is induced by the natural isomorphism of $E'$ 
with $\theta (E')\otimes \alpha '$ as $G'$-bundles using multiplication by
$t$ which gives an isomorphism of $\theta (E)$ and $\theta '(E)$. 

\end{proposition}

\begin{proof}
 Firstly, we recall that the bundle $E\otimes \alpha $ 
is simply the quotient of $E\times \alpha $ by the action $\gamma (\xi , 
a) = (\xi \gamma, \gamma ^{-1}a), \gamma \in \Gamma $. The (right) 
action of $G$ on $E\times \alpha $ given by $(\xi , a)s = (\xi s, a)$ 
goes down to an action on $E\otimes \alpha $ and makes it a principal 
$G$-bundle. The image of $(\xi , a)$ may as well be written as $\xi 
\otimes a$. The projection $E\times \alpha \to E$ goes down to a map 
$q:E\otimes \alpha \to E/\Gamma $. Note that $E/\Gamma $ is a principal 
$G/\Gamma $-bundle. We will denote the image of $\xi \in E$ in 
$E/\Gamma  $ by ${\overline \xi }$ and the image of $s\in G$ in $G/\Gamma $ by 
${\overline s}$. Then we have $q((\xi \otimes a)s) = q(\xi \otimes 
a)(\overline s)$.

Note that as a space over $X$, the bundle $\theta (E)$ is the same as 
$E$ except that the action of $G$ is now $\xi\cdot s = \xi \theta (s)$.  
Hence $F = \theta (E)\otimes \alpha $ is the quotient of $E\times \alpha
$ under the action of $\Gamma $ by $\gamma (\xi , a) = (\xi \theta 
\gamma , \gamma ^{-1}a)$. The action of $G$ on $\theta (E) \times 
\alpha $ given by $(\xi ,a)\cdot s = (\xi \theta (s), a)$ for $s\in G$ 
goes 
down to an action of $G$ on $F$ making it a principal $G$-bundle.

We are given an isomorphism $A: E\to \theta (E)\otimes \alpha $. Let $p$ 
be the composite $q\circ A: E \to \theta (E)/\Gamma = E/\Gamma $ and let 
$\xi $ be an element of $E$. If $A(\xi ) = \eta \otimes a$, we have 
$$A(\xi s) = (\eta \otimes a)\cdot s = (\eta \cdot s \otimes a) = (\eta \theta 
(s)\otimes a)$$ 
so that $p(\xi s) = p(\xi )\overline {\theta (s)}$ for 
all $s\in G$. We now have two (fibre-respecting) morphisms $E\to 
E/\Gamma $, namely, the natural quotient map $\pi $ and the morphism $p$. 
Hence there exists a morphism $f_A:E\to G/\Gamma $ such that $p(\xi ) = 
\pi (\xi )f_A(\xi )$. Also we have $\pi (\xi s) = \pi (\xi ){\overline 
s}$. Thus $p(\xi s) = p(\xi ) \overline {\theta (s)} = \pi (\xi )f_A(\xi 
)\overline {\theta (s)}$ on the one hand, and $p(\xi s) = \pi (\xi 
s)f_A(\xi s) = \pi (\xi ) \overline {s}f_A(\xi s)$ for all $s\in G$ on 
the other. Hence $f_A(\xi s) = \overline{ s}^{-1}f_A(\xi )\overline 
{\theta (s)}$.

We let $G$ act on $G/\Gamma $ by $s.\overline {g} = \overline {s} 
\overline {g} \overline {{\theta s}^{-1}} $.  The computation above shows 
that the composite $g_A$ of $f_A$ with the natural map of $G/\Gamma $ 
onto the quotient for the above action of $G$, is invariant under the 
action of $G$ on $E$. Hence it induces a morphism of $X$ into this 
quotient. Note now that $X$ is projective while this quotient is an 
affine variety. Hence the morphism $g_A$ is constant.

%

As in Proposition \ref{orbit}, we see that the image of $g_A$ consists of stable points 
for the action of $G/\Gamma $ on itself. Hence there is an element 
$\tau \in G/\Gamma $ such that for every $\xi \in E$, we have 
$f_A(\xi ) = \overline {g}\tau \overline {\theta 
(g)^{-1}}$ for some $g\in G$. Let $t\in G$ such that ${\overline t} 
=\tau $.

Now the isomorphism $A$ induces an isomorphism $\theta (E) \to E \otimes 
\theta (\alpha ) = E\otimes \alpha ^{-1}$ and hence an isomorphism 
$\theta (A):\theta (E) \otimes \alpha \to E$. Composing this with $A$ we 
get an automorphism of $(E, \varphi )$ which, by assumption, is given by 
an element $\gamma \in \Gamma $. Clearly this implies that $t\theta (t) 
\in \Gamma $, i.e. $t\in \Gamma _{\theta }$.

For every $\xi \in E$, there exists $a\in \alpha $ such that $A(\xi ) = 
\xi gt{\theta (g)}^{-1} \otimes a$ for some $g\in G$ and $a\in \alpha $. 
Consider now the subspace $E'$ of $E$ given by 
$$E' = \{\xi \in E: A(\xi ) = \xi t\otimes a{\rm~for~some~}a\in \alpha\}$$ 
Then for every $\xi \in E'$, we have $A(\xi s) = (\xi t\otimes a)\cdot s = \xi 
t\theta (s)\otimes a = \xi s.(s^{-1}t\theta (s))\otimes a$ for any $s\in 
G$. In particular, if $s$ satisfies $s^{-1}t\theta (s)t^{-1}\in \Gamma 
$, then (and only then) $\xi s \in E'$, that is to say, $\xi $ and $\xi 
s$ are both in $E'$ if and only if $s\in G' = G_{\Int(t)\theta }$. 
This shows that the structure group of $E$ can be reduced to $G'$ and 
that $E'$ provides such a reduction.

Moreover, if $\xi \in E'$ then there is a unique $a(\xi )$ such that 
$A(\xi ) = \xi t \otimes a(\xi )$. Let $s\mapsto s\theta '(s)^{-1}$ be 
the homomorphism $\rho :G' \to \Gamma '$. Then the above computation 
shows that $a(\xi s) = a(\xi )\rho (s)$ for all $\xi \in E'$ and $s\in 
\Gamma '$.  The image $a(E')$ is a $\Gamma '$ bundle which gives a 
reduction of the structure group of $\alpha $ to $\Gamma '$. 

This proves assertion (1).

Assertion (2) is obvious.

Assertion (3). Notice that $\alpha '$ is given by extension of structure 
group by the homomorphism $g \to g \theta '(g )^{-1}$ of $G' \to \Gamma 
'$. Hence $\theta '(E)\otimes \alpha '$ is given by the extension $g \to 
g \theta '(g)^{-1} \theta '(g)$ which is the identity! So there is a 
natural isomorphism of $E'$ with $\theta '(E')\otimes \alpha '$. From 
the definition it is clear that the extension of this natural isomorphism 
is $A$ as claimed. 
 \end{proof}

Our main result is the following.

\begin{theorem}\label{moduli-n-automorphism-alpha}
Let $(\alpha,a)\in (H^1(X,Z)\rtimes\Out(G))_n$ and
$\zeta_k=\exp(2\pi i\frac{k}{n})$. Consider the 
automorphism
$$
   \begin{aligned}
\iota(a,\alpha,\zeta_k): \cM(G) & \to \cM(G) \\
(E,\varphi) & \mapsto (a(E)\otimes\alpha,\zeta_k a(\varphi)).
  \end{aligned}
$$
Then
\begin{enumerate}

\item 

$$
\bigcup_{[\theta]\in \cl_n^{-1}(a), \; c_\theta(\gamma)=\alpha}
\widetilde{\cM_{\gamma}}(G_\theta,\zeta_k) \subset \cM(G)^{\iota(a,\alpha,\zeta_k)}.
$$  

\item

$$
\cM(G)_{ss}^{\iota(a,\alpha,\zeta_k)}
\subset 
\bigcup_{[\theta]\in \cl_n^{-1}(a), \; c_{\theta}(\gamma)=\alpha}
\widetilde{\cM_\gamma}(G_\theta,\zeta_k)
$$
(except for $\iota(1,1,1)$),
\end{enumerate}
where $\cl_n:\Aut_n(G)/{\sim} \to \Out_n(G)$ is defined  in 
Proposition \ref{cliques-n}, and 
$\cl_n^{-1}(a)$ is, by 
Proposition \ref{n-clique-cohomology}, in bijection with
$H^1_a(\Z/n,\Ad(G))$.
\end{theorem}

\begin{remark}
If $a=1$ and $\alpha=1$,  the map $\iota(a,\alpha,\zeta_0)$ is the identity map of $\cM(G)$ and (1) follows
trivially from  Proposition \ref{extension-reduction}.
\end{remark}

\begin{remark}
As pointed out in Remark \ref{empty-fixed-locus}, the fixed point 
locus could be entirely contained in the strictly polystable part and hence
$\cM(G)_{ss}^{\iota(a,\zeta_k,\alpha)}$ be empty.
\end{remark}

\begin{remark}\label{special-case}
It is clear that Theorem \ref{moduli-n-automorphism} is a special  
case of Theorem \ref{moduli-n-automorphism-alpha}, 
obtained  by taking $\alpha\in H^1(X,Z)$ to be the neutral element.
Note
that for the neutral element $e\in H^1(X,\Gamma_\theta)$, from 
Proposition \ref{quotient-Gamma}, we have
$\cM_e(G_\theta,\zeta_k)= \cM(G^\theta,\zeta_k)/\Gamma_\theta$, thus implying
from Propositions \ref{extension-reduction-2} and \ref{extension-reduction}, 
that
$\widetilde{\cM_e}(G_\theta,\zeta_k)= \widetilde{\cM}(G^\theta,\zeta_k)$.
\end{remark}

\begin{remark}\label{singular-locus}
It is important to point out that 
the second statement in Theorems \ref{moduli-n-automorphism} and 
\ref{moduli-n-automorphism-alpha} cannot be extended
in general to the whole moduli space $\cM(G)$. The fixed points of the singular 
locus in $\cM(G)$ may lead to extra components that are not in general
$(G_\theta,\zeta_k)$-Higgs bundles. Indeed, a strictly polystable $G$-Higgs bundle
can reduce its structure group to a reductive subgroup $H\subset G$ for which
the resulting $H$-Higgs bundle is both stable an simple (see
\cite{garcia-prada-oliveira}).  One has to analyse 
the precise relation between the cliques of $H$ and $G$ to give a 
description of the fixed points. This issue arises already for $G=\SL(n,\C)$
as we will see in Section \ref{sln}. 
\end{remark}

\begin{remark}\label{topological-invariants}
As mentioned in Section \ref{section-hitchin-equations}, fixing  the topological class $c$  
of $E$ we can consider  
$\cM_c(G)\subset \cM(G)$, the  moduli space of semistable $G$-Higgs bundles
with fixed topological class $c$. This is connected and non-empty
(see \cite{garcia-prada-oliveira}). If $G$ is conneced the topological
class is an element  $c\in \pi_1(G)$. Of course $\Out(G)$ acts on $\pi_1(G)$,
and we require for fixed points of the involutions studied above to exist
that $c$ be fixed under the element $a\in \Out_n(G)$. Recall  that,
since $G$ is semisimple, $Z$ is finite and hence the topological class of a 
$G$-bundle $E$ coincides with that of $E\otimes\alpha$ for any 
$\alpha\in H^1(X,Z)$.
\end{remark}

\section{Higgs bundles and representations of the fundamental group}
\label{higgs-reps}

\subsection{Representations of the fundamental group and harmonic reductions}

In this section we take  $G$ to be a reductive Lie group (real or complex).
By a {\bf representation} of $\pi_1(X)$ in
$G$ we understand a homomorphism $\rho\colon \pi_1(X) \to G$.
The set of all such homomorphisms,
$\Hom(\pi_1(X),G)$,  is an analytic  variety, which is algebraic
if $G$ is algebraic.
The group $G$ acts on $\Hom(\pi_1(X),G)$ by conjugation:
$$
(g \cdot \rho)(\gamma) = g \rho(\gamma) g^{-1}
$$
for $g \in G$, $\rho \in \Hom(\pi_1(X),G)$ and
$\gamma\in \pi_1(X)$. If we restrict the action to the subspace
$\Hom^+(\pi_1(X),g)$ consisting of reductive representations,
the orbit space is Hausdorff.  By a {\bf reductive representation} we mean
one that, composed with the adjoint representation in the Lie algebra
of $G$, decomposes as a sum of irreducible representations.
If $G$ is algebraic this is equivalent to the Zariski closure of the
image of $\pi_1(X)$ in $G$ being a reductive group.
(When $G$ is compact every representation is reductive.)  The
{\bf moduli space of reductive representations} of $\pi_1(X)$ in $G$
is defined to be the orbit space
$$
\mathcal{R}(G) = \Hom^{+}(\pi_1(X),G) / G. 
$$
If $G$ is complex $\calR(G)$ coincides with the GIT quotient
$$
\mathcal{R}(G) = \Hom(\pi_1(X),G) \sslash G. 
$$

It has the structure of an  analytic variety (see e.g.\cite{goldman}) 
which is algebraic if $G$ is algebraic and is real if $G$ is real or  
complex  if $G$ is complex.

Suppose that $G$ is connected. Let  $\rho:\pi_1(X)\to G$ be a representation of $\pi_1(X)$ in
$G$. Let $Z_G(\rho)$ be the centralizer in $G$ of
$\rho(\pi_1(X))$. We say that  $\rho$ is {\bf  irreducible} if and
only if it is reductive and $Z_G(\rho)=Z$, where $Z$ is the
centre  of $G$. If $G$ is not connected $Z$ must be replaced
by the invariant subgroup of $Z_G(G_0)$ under the action of the image of 
the map $\Hom(\pi_1(X),G)\to \Hom(\pi_1(X), \pi_0(G))$ associated 
to $\rho$, where
$G_0$ is the identity connected component, $Z_G(G_0)$ is
the centralizer of $G_0$ in $G$, and $\pi_o(G)$ is the group of
connected components of $G$  (see Remark \ref{non-connected-simplicity}).

Given a representation $\rho\colon\pi_{1}(X) \,\longrightarrow\,
G$, there is an associated flat principal $G$-bundle on
$X$, defined as
$$
  E_{\rho} \,=\, \widetilde{X}\times_{\rho}G\, ,
$$
where $\widetilde{X} \,\longrightarrow\, X$ is the universal cover
associated to $X$ and $\pi_{1}(X)$ acts
on $G$ via $\rho$.
This gives in fact an identification between the set of equivalence classes
of representations $\Hom(\pi_1(X),G) / G$ and the set of equivalence classes
of flat principal $G$-bundles, which in turn is parametrized by
the (nonabelian) cohomology set $H^1(X,\, G)$. We have the following.

An important result is the following  theorem of Corlette \cite{corlette},
also proved by Donaldson \cite{donaldson} when $G=\SL(2,\C)$ (see also 
\cite{labourie}).

\begin{theorem}\label{corlette}
Let $G$ be a reductive Lie group. Let $\rho$ be a representation of 
$\pi_1(X)$ in $G$ with corresponding
flat $G$-bundle $E_\rho$. Let $H\subset G$ be a maximal compact subgroup,
and let $E_\rho(G/H)$ be the associated $G/H$-bundle.
Then the existence of a harmonic section of $E_\rho(G/H)$ is equivalent to
the reductivity of $\rho$. 
\end{theorem}


\subsection{$G$-Higgs bundles and representations}\label{higgs-and-reps}

In this section  $G$ will be a complex  semisimple  Lie group.
We explain now the important relation between  $G$-Higgs bundles 
and representations of the 
fundamental group in $G$.  Let $(E,\varphi)$ be a 
$G$-Higgs bundle, and  let $U\subset G$ be a maximal compact subgroup of $G$.
Let $h$ be a reduction of structure group of $E$ from
$G$ to $U$. Let $d_h$ be the Chern connection --- the unique connection on $E$
compatible with $h$ and the holomorphic structure of $E$ ---
and let $F_h$ be its curvature. 
Theorem \ref{higgs-hk} states that the polystability of $(E,\varphi)$ is
equivalent to the existence of a reduction   $h$ satisfying 
the Hitchin equation. A computation shows that if that is the case
$$
D=d_h + \varphi - \tau_h(\varphi)
$$
defines a flat connection on the principal $G$-bundle $E$,
and the reduction $h$ is harmonic. Here $\tau_h$ is as defined in Section
\ref{section-hitchin-equations}. Hence 
the holonomy of this connection defines a representation of 
$\pi_1(X)$ in $G$, which, by Theorem \ref{corlette}, is reductive. In fact 
all  reductive representations  of $\pi_1(X)$ in $G$ arise in this way. 
More precisely, we have the following.

\begin{theorem}\label{naht}
Let $G$ be complex semisimple Lie group. The moduli space $\cM(G)$ of 
polystable  
$G$-Higgs bundles and the moduli space $\calR(G)$  reductive representations 
of $\pi_1(X)$ in $G$ are homeomorphic. Under this homeomorphism the 
irreducible representations of
$\pi_1(X)$ in $G$  are in correspondence with the stable and simple 
$G$-Higgs bundles.  
\end{theorem}

\begin{remark}\label{naht-reductive}
If  $G$ is reductive (but not semisimple) the same result is true if 
we require that the topological class of the $G$ bundle $E$  
(given by an element in $\pi_1(G)$ if $G$ is connected)  be trivial.
If we do not make any restriction on the topological class of $E$ 
there is a similar correspondence involving representations of the 
universal central extension of the fundamental group.
\end{remark}

\subsection{$(G^\theta,\zeta_k)$- and $(G_\theta,\zeta_k)$-Higgs bundles and 
representations}

Theorem \ref{naht} and Remark \ref{naht-reductive} imply the following.

\begin{theorem}
Let $G$ be a complex semisimple Lie group and let $\theta\in \Aut_n(G)$.
Then we have the following:

(1) $\cM(G^\theta,\zeta_0)$ is homeomorphic to $\calR(G^\theta)$,

(2) $\cM(G_\theta,\zeta_0)$ is homeomorphic to $\calR(G_\theta)$.

Under these homeomorphisms the irreducible representations  are in
correspondence with the stable and simple objects.
\end{theorem}

More can be said if $\theta\in \Aut_2(G)$. Recall 
from Section \ref{realforms-group} that in this case we can choose a compact 
conjugation   $\tau$  of $G$,  defining a compact real form $U$  of $G$  
such that $\tau\theta=\theta\tau$. We then 
consider the complex conjugation $\sigma$  of $G$ defined by 
$\sigma=\theta\tau$. This defines a real form $G^\sigma$ of $G$. 
Consider the notations used in Sections \ref{normalizers} and 
\ref{g-theta-higgs}.
Applying similar arguments to the ones used to prove 
Theorem \ref{naht}, we can  combine Theorem \ref{corlette} with 
Theorems \ref{higgs-g+-hk-connected}  and \ref{higgs-g+-hk} 
to prove the following (see \cite{garcia-prada-gothen-mundet} for details).

\begin{theorem}\label{naht-real}
Let $G$ be a complex semisimple Lie group and let $\theta\in \Aut_2(G)$. Let
$\tau$ be a compact conjugation of $G$ commuting with $\theta$ and 
$\sigma$ be the conjugation of $G$ defined by $\sigma:=\theta\tau$. We have the following:

(1) $\cM(G^\theta,-)$ is homeomorphic to $\calR(G^\sigma)$,

(2)  $\cM(G_\theta,-)$ is homeomorphic to $\calR(G_\sigma)$.

Under these homeomorphisms the irreducible representations  are in
correspondence with the stable and simple objects.
\end{theorem}

\begin{remark}
In contrast with  the order 2 case, when $\theta\in \Aut_n(G)$, the
$(G^\theta,\zeta_k)$- and  $(G_\theta,\zeta_k)$-Higgs
bundles for $k>0$ do not have in general  an interpretation in terms of 
representations of the fundamental group of $X$, unless $n$ is even and $k=n/2$.
They are related,  however, to the so-called Hodge
bundles and variations of Hodge structure.
\end{remark}

\section{Involutions of $\cM(G)$ and $\calR(G)$}
\label{section-involutions-higgs-reps}
As in Section \ref{higher-order-auto}, in this section $G$ is a connected complex semisimple 
Lie group,  $X$ is a compact Riemann surface, and $\cM(G)$ is the moduli space
of $G$-Higgs bundles over $X$. 
We specialise the results of Section \ref{higher-order-auto} to the case
of involutions of $\cM(G)$  and study the  resulting involutions
on $\calR(G)$, the moduli space of representations of $\pi_1(X)$ in $G$,
under the non-abelian Hodge theory correspondence (Theorem \ref{naht}). 

\subsection{Involutions of $\cM(G)$}
\label{section-involutions}
Our description involves now the  moduli spaces 
$\cM(G^\theta,\pm)$ and  $\cM(G_\theta,\pm)$ defined in Section 
\ref{invo-higgs-bundles}, 
where $\theta\in \Aut_2(G)$, and $G^\theta$ and $G_\theta$ are the
subgroups of $G$ defined in Section \ref{normalizers}.
As in Section \ref{higher-order-auto}, we will denote by
$\cM(G)_{ss}$ the subvariety of stable and simple points of $\cM(G)$
and by $\widetilde{\cM}(G^\theta,\pm)$ and  $\widetilde{\cM}(G_\theta,\pm)$ 
the images  of $\cM(G^\theta,\pm)$ and  $\cM(G_\theta,\pm)$, respectively  
in $\cM(G)$ under the maps defined  in Propositions 
\ref{extension-reduction-2} and \ref{extension-reduction} respectively.

Theorem \ref{moduli-n-automorphism} specialises in the case of involutions 
to the following.

\begin{theorem}\label{fixed-connected-a}
Let $a\in \Out_2(G)$. Consider the  involutions 

$$
   \begin{aligned}
\iota(a,\pm): \cM(G) & \to \cM(G) \\
(E,\varphi) & \mapsto (a(E),\pm a(\varphi)).
  \end{aligned}
$$
Then
\begin{enumerate}

\item 
$$
\bigcup_{[\theta]\in \cl^{-1}(a)}
\widetilde{\cM}(G^\theta,\pm)\subset \cM(G)^{\iota(a,\pm)},
$$  

\item
$$
\cM(G)_{ss}^{\iota(a,\pm)}
\subset 
\bigcup_{[\theta]\in \cl^{-1}(a)}
\widetilde{\cM}(G^\theta,\pm)
$$
(except for $\iota(1,+)$),
\end{enumerate}
where $\cl:\Aut_2(G)/{\sim} \to \Out_2(G)$ is defined  in 
Proposition \ref{cartan-versus-inner-groups},
and $\cl^{-1}(a)$ is, by Proposition \ref{clique-cohomology},
in bijection  with $H^1_a(\Z/2,\Ad(G))$.
\end{theorem}

An important particular case of Theorem \ref{fixed-connected-a} is the 
case when $a$ is the trivial element.

\begin{theorem}\label{fixed-connected}
Consider the involution
$$
   \begin{aligned}
\iota: \cM(G) & \to \cM(G) \\
(E,\varphi) & \mapsto (E,- \varphi).
  \end{aligned}
$$

Then
 
\begin{enumerate}

\item 
$$
\bigcup_{[\theta] \in \Int_2(G)/\sim}
\widetilde{\cM}(G^\theta,-) \subset \cM(G)^\iota,
$$  

\item
$$
\cM(G)_{ss}^\iota
\subset 
\bigcup_{[\theta] \in \Int_2(G)/{\sim}}\widetilde{\cM}(G^\theta,-).
$$
\end{enumerate}

\end{theorem}

Theorem \ref{moduli-n-automorphism-alpha} in the case of order 2 gives the 
following.

\begin{theorem}\label{fixed-connected-a-alpha}
Let $a\in \Out_2(G)$ and $\alpha\in H^1(X,Z)$ such that
$a(\alpha)=\alpha^{-1}$. Consider the 
involutions
$$
   \begin{aligned}
\iota(a,\alpha,\pm): \cM(G) & \to \cM(G) \\
(E,\varphi) & \mapsto (a(E)\otimes\alpha,\pm a(\varphi)).
  \end{aligned}
$$
Then
\begin{enumerate}

\item 

$$
\bigcup_{[\theta]\in \cl^{-1}(a), \; c_\theta(\gamma)=\alpha}
\widetilde{\cM_{\gamma}}(G_\theta,\pm) \subset \cM(G)^{\iota(a,\alpha,\pm)}.
$$  

\item
$$
\cM(G)_{ss}^{\iota(a,\alpha,\pm)}
\subset 
\bigcup_{[\theta]\in \cl^{-1}(a), \; c_{\theta}(\gamma)=\alpha}
\widetilde{\cM_\gamma}(G_\theta,\pm)
$$
(except for $\iota(1,1,+)$),
\end{enumerate}
where $\cl:\Aut_2(G)/{\sim} \to \Out_2(G)$ is defined  in 
Proposition \ref{cartan-versus-inner-groups},
and $\cl^{-1}(a)$ is, by Proposition \ref{clique-cohomology},
in bijection with $H^1_a(\Z/2,\Ad(G))$.
\end{theorem}

\subsection{Involutions of $\calR(G)$}
\label{involutions-reps}

The involutions of $\cM(G)$ studied in Section 
\ref{section-involutions}
induce naturally involutions on the moduli space of representations
$\calR(G)$, under the homeomorphism between $\cM(G)$ and $\calR(G)$
 given by Theorem \ref{naht}. We analyse now these involutions.


Let $(E,\varphi)$ be a polystable $G$-Higgs bundle and let $h$ be 
a solution to the Hitchin equation given by Theorem 
\ref{higgs-hk}. Recall from  Section \ref{higgs-and-reps} that
$$
D=d_h + \varphi - \tau_h(\varphi)
$$
defines a flat connection on the principal $G$-bundle $E$, where 
$d_h$ is the unique connection on $E$
compatible with $h$ and the holomorphic structure of $E$,
and $\tau_h$ is defined, by $h$,  a 
conjugation $\tau$  of $G$ defining a compact real
form, and the natural conjugation 
on $(1,0)$-forms on $X$ (see  (\ref{tau-h}) for the precise definition).

Let $\theta\in \Aut_2(G)$. From Proposition 
\ref{conjugations-versus-involutions}, we can choose $\theta$
in the class in $\Aut_2(G)/\sim$ such that $\theta\tau=\tau\theta$. Let
$\sigma$ be the conjugation of $G$ defined by $\sigma:=\theta\tau$,
and $\sigma_h:=\theta\tau_h$.
We have the following

\begin{proposition} 
(1) The flat $G$-connection corresponding to $(\theta(E),\theta(\varphi))$ is
given by $\theta(D)$.

(2) The flat $G$-connection corresponding to $(\theta(E),-\theta(\varphi))$ is
given by $\sigma_h(D)$.
\end{proposition}

\begin{proof}
Let us represent the holomorphic structure of $E$  by a Dolbeault operator 
$\dbar_E$. We then have that $d_h=\dbar_E+\tau_h(\dbar_E)$. From this
we have
$$
\theta(D)=\theta(\dbar_E)+\theta\tau_h(\dbar_E)+\theta(\varphi)-
\theta\tau_h(\varphi).
$$
But $\theta\tau_h=\tau_h\theta$ and $\theta(\dbar_E)=\dbar_{\theta(E)}$,
proving (1).

The proof of (2) follows from the following  computation:
\begin{align*}
\sigma_h(D) & =\sigma_h(\dbar_E)+\sigma_h\tau_h(\dbar_E)+\sigma_h(\varphi)-
\sigma_h\tau_h(\varphi)\\
            & =\tau_h\theta(\dbar_E)+
\theta(\dbar_E)+\tau_h\theta(\varphi)-\theta(\varphi)\\
             & =
\dbar_{\theta(E)}+ \tau_h(\dbar_{\theta(E)}) -\theta(\varphi)- 
\tau_h(- \theta(\varphi)).
\end{align*}
\end{proof}

From this we immediately have the following.

\begin{proposition}\label{correspondence-involutions}
Let $\theta\in \Aut_2(G)$, and $\sigma=\theta\tau$.
Let $\alpha\in H^1(X,Z)$ and 
$\lambda: \pi_1(X)\to Z$ its corresponding representation such that
$\theta(\alpha)=\alpha^{-1}$, which is equivalent to 
$\theta(\lambda)=\sigma(\lambda)=\lambda^{-1}$.
Let $(E,\varphi)$ be a polystable $G$-Higgs
bundle and  $\rho$ be the 
corresponding element in $\calR(G)$. Then:

(1)  The involution of $\cM(G)$ defined by  $(E,\varphi)\mapsto 
(\theta(E)\otimes \alpha,\theta(\varphi))$
corresponds with the involution of $\calR(G)$ given by 
$\rho\mapsto \lambda\theta(\rho)$. 

(2) The involution of $\cM(G)$ defined by $(E,\varphi)\mapsto 
(\theta(E)\otimes \alpha,-\theta(\varphi))$
corresponds with the involution of $\calR(G)$ given by 
$\rho\mapsto \lambda\sigma(\rho)$.
\end{proposition}

The involutions 
$\rho\mapsto \lambda\theta(\rho)$ and $\rho\mapsto \lambda\sigma(\rho)$
depend naturally only on the cliques $\cl([\theta])$ and 
$\widehat{\cl}([\sigma])$, respectively, where $\cl$ and $\widehat{\cl}$
are defined in Section \ref{realforms-group}. 
Let $a\in \Out_2(G)$ and $\theta\in \Aut_2(G)$ such that 
$\cl([\theta])=a$. Let $\sigma=\theta\tau$ be the corresponding
conjugation. Of course $\widehat{\cl}([\sigma])=a$.
Let $\rho\in\calR(G)$.  We define
$$
a^+(\rho)=\theta(\rho)\;\;\;\mbox{and}\;\;\; a^-(\rho)=\sigma(\rho).
$$

Recall from Theorem \ref{naht} that the smooth locus $\calR(G)_i\subset\calR(G)$
consisting of irreducible representations is homeomorphic to $\cM(G)_{ss}$,
the smooth locus of $\cM(G)$ consisting of stable and simple objects.

Of course we have statements corresponding to Propositions
\ref{extension-reduction-2}  and \ref{extension-reduction}
for  the moduli
spaces of representations $\calR(G^\theta)$,   $\calR(G^\sigma)$,
$\calR(G_\theta)$,   and $\calR(G_\sigma)$ (these can be proved directly or
invoking  Theorem \ref{naht-real}). We denote their images in
$\calR(G)$, respectively, by   
$\widetilde{\calR}(G^\theta)$,   $\widetilde{\calR}(G^\sigma)$,
$\widetilde{\calR}(G_\theta)$,   and $\widetilde{\calR}(G_\sigma)$. 

From Theorems \ref{naht-real} and \ref{fixed-connected-a} we have 
the following.

\begin{theorem}\label{fixed-a-reps}
Let $a\in \Out_2(G)$. Consider the  involutions 

$$
   \begin{aligned}
\widehat{\iota}(a,\pm): \calR(G) & \to \calR(G) \\
\rho  & \mapsto a^\pm(\rho).
  \end{aligned}
$$

Then

\begin{enumerate}

\item 
$$
\bigcup_{[\theta]\in \cl^{-1}(a)}
\widetilde{\calR}(G^\theta)\subset \calR(G)^{\widehat{\iota}(a,+)},
$$  

\item
$$
\calR(G)_{i}^{\widehat{\iota}(a,+)}
\subset 
\bigcup_{[\theta]\in \cl^{-1}(a)}
\widetilde{\calR}(G^\theta),
$$
except for $a=1$.
\item 
$$
\bigcup_{[\sigma]\in {\widehat{\cl}}^{-1}(a)}
\widetilde{\calR}(G^\sigma)\subset \calR(G)^{\widehat{\iota}(a,-)},
$$  

\item
$$
\calR(G)_{i}^{\widehat{\iota}(a,-)}
\subset 
\bigcup_{[\sigma]\in {\widehat{\cl}}^{-1}(a)}
\widetilde{\calR}(G^\sigma).
$$
\end{enumerate}
\end{theorem}

\begin{remark}
The  particular case of Theorem \ref{fixed-a-reps} for 
$\hat{\iota}(a,-)$ 
gives the representation statement corresponding to Theorem 
\ref{fixed-connected}. This case
involves only the equivalence classes of real forms that are inner equivalent
to the compact conjugation, that is the real forms of Hodge type.
\end{remark}

To describe the fixed points of the involutions of $\calR(G)$ involving also
an element $\lambda\in \calR(Z)=\Hom(\pi_1(X),Z)$, recall the extensions
(\ref{sigma-exact-sequence}) and (\ref{g-theta}).
These define  maps
$$
\hat{\gamma}_\theta:\calR(G_\theta)\to \calR(\Gamma_\theta),
\;\;\;\;\mbox{and}\;\;\;\; 
\hat{\gamma}_\sigma: \calR(G_\sigma)\to \calR(\Gamma_\sigma),
$$
which assign to every $\rho\in \calR(G_\theta)$
(resp. $\rho\in \calR(G_\sigma)$)
an invariant 
$\hat{\gamma}_\theta(\rho)\in \calR(\Gamma_\theta)=\Hom(\pi_1(X), \Gamma_\theta)$
(resp. $\hat{\gamma}_\sigma(\rho)\in \calR(\Gamma_\sigma)=\Hom(\pi_1(X), \Gamma_\sigma)$).
We will denote by $\calR_{\gamma}(G_\theta)$ (resp. $\calR_{\gamma}(G_\sigma)$) 
the subvariety of $\calR(G_\theta)$ (resp. $\calR(G_\sigma)$) 
with fixed invariant $\gamma$, and by
$\widetilde{\calR_{\gamma}}(G_\theta)$ 
(resp. $\widetilde{\calR_{\gamma}}(G_\sigma)$), 
its corresponding image in 
$\calR(G)$.  

From Propositions \ref{c-map}, \ref{gamma=gamma} and \ref{ext-twisted-iso},
we also have injective homomorphisms
$$
{\hat c}_\theta:\calR(\Gamma_\theta)\to \calR(Z), \;\;\;\;\mbox{and}\;\;\;\;
{\hat c}_\sigma:\calR(\Gamma_\sigma)\to \calR(Z).
$$

\begin{theorem}\label{fixed-a-lambda-reps}
Let $a\in \Out_2(G)$ and $\lambda\in\calR(Z)$ such that
$a(\lambda)=\lambda^{-1}$.

Consider the  involutions 

$$
   \begin{aligned}
\widehat{\iota}(a,\lambda,\pm): \calR(G) & \to \calR(G) \\
\rho  & \mapsto \lambda a^\pm(\rho).
  \end{aligned}
$$

Then

\begin{enumerate}

\item 
$$
\bigcup_{[\theta]\in \cl^{-1}(a),\; \hat{c}_\theta(\gamma)=\lambda}
\widetilde{\calR_\gamma}(G_\theta)\subset \calR(G)^{\widehat{\iota}(a,\lambda,+)},
$$  

\item
$$
\calR(G)_{i}^{\widehat{\iota}(a,\lambda,+)}
\subset 
\bigcup_{[\theta]\in \cl^{-1}(a),\; \hat{c}_\theta(\gamma)=\lambda}
\widetilde{\calR_\gamma}(G_\theta),
$$
except for $a=1$ and $\lambda=1$,

\item 
$$
\bigcup_{[\sigma]\in \widehat{\cl}^{-1}(a), \; \hat{c}_\sigma(\gamma)=\lambda}
\widetilde{\calR_\gamma}(G_\sigma)\subset \calR(G)^{\widehat{\iota}(a,\lambda,-)},
$$  

\item
$$
\calR(G)_{i}^{\widehat{\iota}(a,\lambda,-)}
\subset 
\bigcup_{[\sigma]\in \widehat{\cl}^{-1}(a), \; \hat{c}_\sigma(\gamma)=\lambda}
\widetilde{\calR_\gamma}(G_\sigma).
$$

\end{enumerate}
\end{theorem}

\begin{remark}
One can make corresponding considerations to those
in Remarks \ref{special-case}, \ref{singular-locus} and 
\ref{topological-invariants}
for the involutions on $\calR(G)$ studied in this section.
\end{remark}

\subsection{Hyperk\"akler and Lagrangian subvarieties of $\cM(G)$}

The  fixed points on the smooth locus of $\cM(G)$ of the involutions 
studied in Section \ref{section-involutions} provide examples of 
hyperk\"ahler and Lagrangian subvarieties of $\cM(G)$. Recall from 
Section \ref{section-hitchin-equations} that
the smooth locus $\cM(G)_{ss}\subset\cM(G)$ has a hyperk\"ahler 
structure, obtained as a hyperk\"ahler quotient
by solving Hitchin equations. In particular  $\cM(G)_{ss}$ has  
complex structures $J_i$, $i=1,2,3$ 
satisfying the quaternion relations $J_i^2=-I$, 
and real symplectic structures $\omega_i$, $i=1,2,3$.
The Lagrangian condition we are referring above is with respect 
to the $J_1$-holomorphic symplectic form
$\Omega_1\,=\, \omega_2+\sqrt{-1}\omega_3$.
More precisely we have the following.

\begin{theorem}\label{hyperkahler-lagrangian}
Let $a\in \Out_2(G)$ and $\alpha\in H^1(X,Z)$, such that
$a(\alpha)=\alpha^{-1}$. Then, for every 
$\theta\in\Aut_2(G)$ and $\gamma \in H^1(X,\Gamma_\theta)$
such that 
$[\theta]\in \cl^{-1}(a)$, and  $c_{\theta}(\gamma)=\alpha$,
we have 

(1) 
$ \cM(G)_{ss}\cap \widetilde{\cM_\gamma}(G_\theta,+)$
is a hyperk\"ahler submanifold of $\cM(G)_{ss}$.
In particular $\cM(G)_{ss}\cap \widetilde{\cM}(G^\theta,+)$
is a hyperk\"ahler submanifold of $\cM(G)_{ss}$.

(2) $\cM(G)_{ss}\cap \widetilde{\cM_\gamma}(G_\theta,-)$
is a $(J_1,\Omega_1)$-complex Lagrangian  submanifold of $\cM(G)_{ss}$.
In particular $\cM(G)_{ss}\cap \widetilde{\cM}(G^\theta,-)$
is a $(J_1,\Omega_1)$-complex Lagrangian submanifold  of $\cM(G)_{ss}$.

\end{theorem}

\begin{proof}
From Proposition \ref{correspondence-involutions},  
the involution $\iota(a,\alpha,+)$ on $\cM(G)$ is holomorphic with 
respect to complex structure $J_2$ (the natural complex structure on 
$\calR(G)$). Since it is
$J_1$-holomorphic (recall that $J_1$ is the natural complex structure on
$\cM(G)$) is also $J_3$-holomorphic, and hence (1) follows.

The proof of (2) follows from the fact that the involution $\iota(a,\alpha,-)$ 
on $\cM(G)$ is $J_1$-holomorphic and
$J_2$-antiholomorphic, 
 by Proposition \ref{correspondence-involutions},    and hence
 $J_3$-antiholomorphic. Since it is an isometry this implies that $\omega_2$,
and $\omega_3$, and hence $\Omega_1$ vanish on the fixed point locus,
proving the assertion. This argument is given by Hitchin in 
\cite{hitchin1987} to give 
this result when $G=\SL(2,\C)$, and  more generally in 
\cite{garcia-prada-2007,garcia-prada}.  
\end{proof}

Part (1) of Theorem \ref{hyperkahler-lagrangian} generalises in 
a straightforward manner to
higher order automorphisms, namely we have the following.

\begin{theorem}\label{hyperkahler-n}
Let $a\in \Out_n(G)$ and $\alpha\in H^1(X,Z)$, such that
$\alpha a(\alpha)\cdots a^{n-1}(\alpha)=1$. Then, for every 
$\theta\in\Aut_n(G)$ and $\gamma \in H^1(X,\Gamma_\theta)$
such that 
$[\theta]\in \cl_n^{-1}(a)$, and  $c_{\theta}(\gamma)=\alpha$,
we have that
$ \cM(G)_{ss}\cap \widetilde{\cM_\gamma}(G_\theta,\zeta_0)$
is a hyperk\"ahler submanifold of $\cM(G)_{ss}$.
In particular $\cM(G)_{ss}\cap \widetilde{\cM}(G^\theta,\zeta_0)$
is a hyperk\"ahler submanifold of $\cM(G)_{ss}$.
\end{theorem}

In the context of mirror symmetry for Calabi--Yau  manifolds
\cite{kapustin-witten}, one refers to a
complex submanifold equipped with a holomorphic bundle (or more generally a 
sheaf) as a $B$-brane, and to a Lagrangian submanifold equipped with a flat
bundle as an $A$-brane. In this hyperk\"ahler situation we thus have branes
of various types according to complex structures $J_i$, $i=1,2,3$, and 
K\"ahler forms $\omega_i$, $i=1,2,3$. Namely one can have branes of types
$(B,B,B)$, $(B,A,A)$, $(A,B,A)$ and $(A,A,B)$. In this language, it is
clear that the submanifolds
$\cM(G)_{ss}\cap \widetilde{\cM_\gamma}(G_\theta,+)$ in Theorem 
\ref{hyperkahler-lagrangian} are the support of  $(B,B,B)$-branes,
while the submanifolds $\cM(G)_{ss}\cap \widetilde{\cM_\gamma}(G_\theta,-)$ 
are the support of $(B,A,A)$-branes.

\begin{remark}
Support for branes of types  $(A,B,A)$ and $(A,A,B)$, that is,  complex 
 Lagrangian with 
respect to $(J_2,\Omega_2)$ and $(J_3,\Omega_3)$, respectively can also be
obtained from involutions on $\cM(G)$, involving now also  conjugations 
on $X$ 
(see \cite{biswas-garcia-prada,baraglia-schaposnik}).
\end{remark}

\section{Involutions of $\cM(\SL(n,\C))$}\label{sln}

To ilustrate our main results, we will consider the case $G=\SL(n,\C)$.
For this group, like for all classical groups, it is convenient to consider 
$G$-Higgs bundles in terms of vector bundles. From this point of view, a 
Higgs bundle over $X$ is a pair $(V,\varphi)$, where $V$ is a holomorphic 
vector bundle over $X$ and $\varphi\in H^0(X,\End(V)\otimes K)$, that is
a homomorphism $\varphi:V\to V\otimes K$. In this case stability is defined
in terms of slopes. Recall that the {\bf slope} of a bundle is defined as 
$\mu(V)=\deg V/\rank V$.  We say that $(V,\varphi)$ is stable if
for every proper subbundle $V'\subset V$ such that $\varphi(V')\subset V'$
we have $\mu(V')<\mu(V)$. The Higgs bundle is said to be polystable if
$(V,\varphi)=\oplus_i (V_i,\varphi_i)$, with $(V_i,\varphi_i)$ stable and 
$\mu(V_i)=\mu(V)$ for every $i$. This is the notion introduced in the 
original paper by Hitchin \cite{hitchin1987}.

In order for a Higgs bundle $(V,\varphi)$ to correspond to a $\SL(n,\C)$-Higgs
bundle we must require that $\det V$ be trivial and $\Tr(\varphi)=0$. The 
stability conditions for the principal and vector bundle points of view 
coincide naturally.   

We will start with the simplest  possible situation: $G=\SL(2,\C)$.

\subsection{$G=\SL(2,\C)$}

In this case $\Out(G)=\{1\}$, and the compact real form $\SU(2)$ and 
split real form $\SL(2,\R)\cong \SU(1,1)$ of $\SL(2,\C)$,
corresponding to conjugations $\tau(A)={(\overline{A}^t)}^{-1}$ and
$\sigma_s(A)=\overline{A}$, respectively, are indeed inner equivalent.
We can see this explictly at the level of Lie algebras for example, since
the conjugation
with respect to  the real form $\liesu(2)$,   
$\tau(A)= -\overline{A}^t$, and the conjugation  with respect to  the real 
form $\liesl(2,\R)$, $\sigma_s(A)=\overline{A}$, are related by
$$
\sigma_s(A)=J\tau(A)J^{-1}
$$
for $J\in \liesl(2,\R)$ given by
$$
J =
\begin{pmatrix}
  0 & 1 \\
  -1  & 0
\end{pmatrix},
$$
simply because for every $A\in\liesl(2,\R)$, one has 
$ JA=-A^tJ$.

We consider now the involutions $\iota(a,\pm)$  defined in Section 
\ref{section-involutions}.
Since $\iota(1,+)$ is the identity map, the only non-trivial involution of this
type is $\iota(1,-)$ --- the special case studied in Section 
\ref{section-involutions}. This case is studied by Hithin in \cite{hitchin1987}
(see also \cite{garcia-prada-ramanan-rank2}).

The elements in $\Aut_2(G)$ corresponding to the conjugations $\tau$ and 
$\sigma_c$ are given, respectivey by $\theta_c=\tau^2=\Id_G$ and
$\theta_s=\sigma_c\tau$. We thus have that $\theta_s(A)={(A^t)}^{-1}$, and hence
$G^{\theta_c}=G=\SL(2,\C)$ and $G^{\theta_s}=\SO(2,\C)\cong \C^\ast$.

The moduli space $\cM(\SL(2,\C),-)$ is then  isomorphic to the moduli space of
polystable $\SL(2,\C)$-bundles, since $\lieg^-=0$ in this case, and hence the
Higgs fields must  vanish identically. By the theorem of Narasimhan and 
Seshadri \cite{narasimhan-seshadri} this is homeomorphic to $\calR(\SU(2))$.

On the other hand the moduli space $\cM(\SO(2,\C),-)$ is described by the 
isomorphism clasess of Higgs bundles of the form

\begin{equation}\label{higgs-bundle}
V = L\oplus L^{-1} \;\;\;\; \;\; \mbox{and}\;\;\;\;\;\; 
\varphi =
\begin{pmatrix}
  0 & \beta\\
  \gamma & 0
\end{pmatrix},
\end{equation}
where $L$ is a line bundle, $\beta\in H^0(X,L^2\otimes K)$ and 
$\gamma\in H^0(X,L^{-2}\otimes K)$. From stability one deduces 
(see \cite{hitchin1987}) that if 
$d=\deg L$, then $|d|\leq g-1$, where $g$ is the genus of $X$. The subspace
of elements of $\cM(\SO(2,\C),-)$ with fixed 
$d$ satisfying this inequality defines a connected component, as shown 
in \cite{hitchin1987}, where an explicit description of this subspaces is given.

The moduli space $\cM(\SO(2,\C),-)$ is homeomorphic to 
$\calR(\SL(2,\R))$. From the point of view of representations the inequality 
$|d|\leq g-1$ is proved by Milnor \cite{milnor} and the connectedness fixing
the degree is proved by Goldman \cite{goldman}, who also shows that the
components with maximal degree $d$ is identified with Teichm\"uller space.

In the image of $\cM(\SO(2,\C),-)$ in $\cM(\SL(2,\C)$, denoted
by $\widetilde{\cM}(\SO(2,\C),-)$, the degree $d$ component with $d\neq 0$ is 
identified with the degree $-d$ component, while in the degree $0$ component
there are some identifications.

Note that, although $\iota(1,+)$ is the identity, the moduli space
$\cM(\SO(2,\C))$, consisting  of 
Higgs bundles of the form

\begin{equation}\label{higgs-bundle}
V = L\oplus L^{-1} \;\;\;\; \;\; \mbox{and}\;\;\;\;\;\; 
\varphi =
\begin{pmatrix}
  \psi & 0\\
  0 & \psi
\end{pmatrix},
\end{equation}
where $L$ is a line bundle and $\psi\in H^0(X,K)$, maps to $\cM(\SL(2,\C))$, going to the 
strictly polystable locus.

We consider now the involutions appearing in Theorem 
\ref{fixed-connected-a-alpha}.
In our situation  $Z=\{\pm I\}\cong \Z/2$, and hence $H^1(X,Z)=J_2$, the 
$2$-torsion elements in the Jacobian $J$ of $X$. A direct approach to this 
is given in \cite{garcia-prada-ramanan-rank2}.

Consider the normalizer $N\SO(2,\C)$ of $\SO(2,\C)$ in $\SL(2,\C)$. This is
generated by $\SO(2,\C)$ and 
$J=\begin{pmatrix}
  0 & i \\
  i  & 0
\end{pmatrix}$.
The group generated by $J$ is isomorphic to $\Z/4$
and fits in the exact sequence 
\begin{equation}\label{z4}
0\lra \Z/2\lra \Z/4\lra  \Z/2\lra 1,
\end{equation}
where the subgroup $\Z/2\subset \Z/4$ is $\{\pm I\}$.

We thus have  an exact sequence

\begin{equation}\label{c-normalizer}
0\lra \SO(2,\C)\lra N\SO(2,\C)\lra  \Z/2\lra 1.
\end{equation}
Of course this is the sequence defining the Weyl group of $\SL(2,\C)$.

Similarly, we also have that   
$N\SL(2,\R)$, the normalizer of $\SL(2,\R)$ in $\SL(2,\C)$, 
is given by

\begin{equation}\label{normalizer=sl2r}
0\lra \SL(2,\R)\lra N\SL(2,\R)\lra  \Z/2\lra 1.
\end{equation}

Recall from Section \ref{normalizers} that
$G_{\theta_s}=N\SO(2,\C)$, $G_{\sigma_s}=N\SL(2,\R)$, and 
hence  
$\Gamma_{\theta_s}=\Gamma_{\sigma_s}=\Z/2$. Since
$Z=\Z/2$ we have that the maps $c_{\theta_s}$  and
$\hat{c}_{\sigma_s}$ intervening  in Theorems
\ref{fixed-connected-a-alpha} and \ref{fixed-a-lambda-reps} are  
isomorphisms. On the other hand $\Gamma_\tau=\{1\}$.

For every $\alpha\in H^1(X,Z)=J_2$ we consider the involutions 
$\iota(\alpha,\pm)$ of 
$\cM(\SL(2,\C))$ defined by $(V,\varphi)\mapsto (V\otimes\alpha,\pm \varphi)$.
From Theorems \ref{fixed-connected-a-alpha} and \ref{fixed-a-lambda-reps} 
we see that the fixed point locus of 
$\iota(\alpha,+)$  (with $\alpha\neq 1$) is described by the moduli space
$\cM_\alpha(N\SO(2,\C),+)$, which in turn is homeomorhic to 
$\calR_\alpha(N\SO(2,\C))$; while the fixed points of $\iota(\alpha,-)$ are
described by the moduli space 
$\cM_\alpha(N\SO(2,\C),-)$,  homeomorhic to 
$\calR_\alpha(N\SL(2,\R))$. In the latter case $\alpha$ is allowed to be $1$, 
and  in that case the fixed point locus is described by 
$\cM(\SO(2,\C),-)$, homeomorhic to $\calR(\SL(2,\R))$.

One can construct  the moduli spaces
$\cM_\alpha(N\SO(2,\C),\pm)$ in terms of Prym varieties:
If $\alpha$ is a  non-trivial element of $J_2$, there is associated to it a
canonical $2$-sheeted \'etale cover $\pi: X_\alpha \to X$. Consider the 
norm homomorphism $\Nm:\Pic(X_\alpha)\to \Pic(X)$. Its kernel consists of 
two components and the one that contains the trivial bundle is the 
{\bf Prym variety} $P_\alpha$ associated to $\alpha$.
If $L$ is a line bundle on $X_\alpha$, its direct image $\pi_*L$ is a rank
two vector bundle $V$. Moreover $V$ is polystable \cite{narasimhan-ramanan}.
Since  $\det E=\Nm(L)\otimes \alpha$, we must take line bundles in
$S_\alpha=\Nm^{-1}(\alpha)$.  This consists of two cosets of $P_\alpha$, each
of which is left invariant under the Galois involution. 
In \cite{garcia-prada-ramanan-rank2} we give a description of the 
moduli spaces $\cM_\alpha(N\SO(2,\C),\pm)$ in terms of $S_\alpha$.

\subsection{$G=\SL(n,\C)$,\ $n>2$}

In this case $\Out(G)=\Z/2$. There are hence two cliques: $a=1$ and
$a=-1$. The classes in $\conj(G)/\sim$ corresponding to the trivial clique
$a=1$ are represented (see \cite{helgason} e.g.) by the conjugations
$\sigma_{p,q}$, with $0\leq p\leq q$ and $p+q=n$ given by 
\begin{equation}\nonumber
   \begin{aligned}
  \sigma_{p,q}:\SL(n,\C)   & \to \SL(n,\C) \\
    A  &\mapsto I_{p,q} (\overline{A}^t)^{-1} I_{p,q},
  \end{aligned}
\end{equation}
where
\begin{equation}\label{ipq}
I_{p,q} =
\begin{pmatrix}
  I_p & 0 \\
  0 & -I_q
\end{pmatrix}.
\end{equation}

One has $G^{\sigma_{p,q}}=\SU(p,q)$. In particular $\tau:=\sigma_{0,n}$ gives
the compact real form $\SU(n)$.
The elements in $\Aut_2(G)$ corresponding to $\sigma_{p,q}$ are given by
$\theta_{p,q}:=\tau\sigma_{p,q}$, and hence
$G^{\theta_{p,q}}=\SSS(\GL(p,\C)\times \GL(q,\C))$. 

From Theorem \ref{naht-real} we have homeomorphisms
$$
\cM(G^{\theta_{p,q}},+)\cong \calR(\SSS(\GL(p,\C)\times \GL(q,\C))
\;\;\;\mbox{and}\;\;\; 
\cM(G^{\theta_{p,q}},-)\cong \calR(\SU(p,q)).
$$
The moduli spaces $\cM(G^{\theta_{p,q}},-)$ and the correspondence with
$\calR(\SU(p,q))$ have been extensively studied in
\cite{bradlow-garcia-prada-gothen,bradlow-garcia-prada-gothen2}, 
where there is counting of the number of connected
components in terms of the so-called {\bf Toledo invariant}, an integer
invariant similar to the one appearing above for $\SU(1,1)$. Of course 
$\cM(G^{\theta_{0,n}},-)$ is the moduli space of polystable 
$\SL(n,\C)$-bundles and the homeomorphism with $\calR(\SU(n))$ is 
given by the Narasimhan--Seshadri theorem \cite{narasimhan-seshadri}.

The involutions  $\iota(-1,\pm)$ on $\cM(\SL(n,\C))$
for the outer clique $a=-1$ is given by
\begin{equation}\nonumber
   \begin{aligned}
\iota(-1,\pm): 
\cM(\SL(n,\C)) & \to \cM(\SL(n,\C)) \\
(V,\varphi) & \mapsto (V^*,\mp {\varphi}^t),
  \end{aligned}
\end{equation}
where $V^*$ is the dual vector bundle and ${\varphi}^t$ is the dual of 
$\varphi$ tensored with the identity of $K$.

To describe the fixed points of $\iota(-,\pm)$ given by Theorem 
\ref{fixed-connected-a},
we recall (see \cite{helgason} e.g.) that the classes in 
$\conj(G)/\sim$ corresponding to the outer clique $a=-1$
are represented by the conjugations
$\sigma_s(A)=\overline{A}$, corresponding to the {\bf split} real form, and
if $n=2m$, to $\sigma_*(A)=J_m\overline{A} J_m^{-1}$,
where 
$$
J_m =
\begin{pmatrix}
  0 & I_m \\
  -I_m  & 0
\end{pmatrix}.
$$
 
We have the corresponding elements in $\Aut_2(G)$ given by 
$\theta_s:=\tau\sigma$ and $\theta_*:=\tau\sigma_*$, and hence
$G^{\sigma_s}=\SL(n,\R)$, $G^{\theta_s}=\SO(n,\C)$, and if $n=2m$,
 $G^{\sigma_*}=\SU^*(2m)$, and $G^{\theta_*}=\Sp(2m,\C)$.

As above, from Theorem \ref{naht-real} we have homeomorphisms
$$
\cM(G^{\theta_s},+)\cong \calR(\SO(n,\C))
\;\;\;\mbox{and}\;\;\; 
\cM(G^{\theta_s},-)\cong \calR(\SL(n,\R)),
$$
and  
$$
\cM(G^{\theta_*},+)\cong \calR(\Sp(2m,\C))
\;\;\;\mbox{and}\;\;\; 
\cM(G^{\theta_*},-)\cong \calR(\SU^*(2m)),
$$

The correspondence $\cM(G^{\theta_s},-)\cong \calR(\SL(n,\R))$ is studied
by Hitchin in \cite{hitchin1992}, where he counts the number of connected 
components of
$\calR(\SL(n,\R))$ and introduces what he calls the higher
 Teichm\"uller components, now known as {\bf Hitchin components},  components
 analogous to the Teichm\"uller components for $\SL(2,\R)$ mentioned above, and
that exist for the split real form of every semisimple complex Lie group $G$.
The correspondence 
  $ \cM(G^{\theta_*},-)\cong \calR(\SU^*(2m))$ is studied in
  \cite{garcia-prada-oliveira-u*}, where it is
  shown that $\calR(\SU^*(2m))$ is connected.

The case of the involution $\iota(-,\pm)$ provides with  a very good example 
to ilustrate the need for restrecting to the smooth locus $\cM(G)_{ss}$  
of $\cM(G)$ in statement (2) of Theorem \ref{fixed-connected-a} (see Remark
\ref{singular-locus}). Indeed, if 
an $\SL(n,\C)$-Higgs bundle $(V,\varphi)$ is not stable (which in this case implies also
simple), the Higgs bundle is polystable and hence 
$(V,\varphi)=\oplus (V_i,\varphi_i)$ with $(V_i,\varphi_i)$ stable and 
$\deg V_i=0$, i.e. the structure group of $(V,\varphi)$  reduces
to $\SSS(\Pi_i \GL(n_i,\C))$ with $\sum n_i=n$. On each summand the 
involution $\iota(-,\pm)$ sends
$(V_i,\varphi_i)\mapsto (V_i^\ast,\mp \varphi_i^t)$, implying  in 
particular that, if $(V,\varphi)$ is a fixed point of the involution, 
the bundles $V_i$ reduce their structure group 
to $\OO(n_i,\C)$ or $\Sp(n_i,\C)$. But since there is no need for all the 
bundles $V_i$  simultaneously to be orthogonal  or  symplectic,
the object $(V,\varphi)$ may not be included in 
$\widetilde{\cM}(G^\theta,\pm)$ for $\theta=\theta_s$ or $\theta=\theta_*$.
In contrast with this, in the case of the involution $\iota(1,-)$, which sends
$(V,\varphi)\to (V,-\varphi)$, 
the inclusion (2) of  Theorem \ref{fixed-connected-a} extends to the 
whole moduli space $\cM(G)$ and not just the smooth locus.

Let $\alpha\in J_2(X)$. We consider now the involutions $\iota(1,\alpha,\pm)$
given  by 
$(V,\varphi)\mapsto (V\otimes \alpha,\pm \varphi)$. To describe the fixed
points, we have to compute the groups $G_{\theta_{p,q}}$, given by 
(\ref{g-theta}). 
A computation shows that
$\Gamma_{\theta_{p,q}}=\{1\}$ if $p\neq q$ and $\Gamma_{\theta_{p,q}}\cong
\Z/2$  if $p=q$ with $n=2p$. We have a situation similar to that of $\SL(2,\R)$
in the previous section.  
We thus have  exact sequences for the normalizers in $\SL(n,\C)$ of
$\SSS(\GL(p,\C)\times \GL(p,\C))$ and $\SU(p,p)$, respectively given by
 \begin{equation}\label{normalizer-glp-glp}
0\lra \SSS(\GL(p,\C)\times \GL(p,\C))
\lra 
N\SSS(\GL(p,\C)\times \GL(p,\C))
\lra  \Z/2\lra 1,
\end{equation}
and
\begin{equation}\label{normalizer-supp}
0\lra \SU(p,p)\lra N\SU(p,p)\lra  \Z/2\lra 1.
\end{equation}

Since $Z=\Z/2p$ we have that the map 
$c_{\theta_{p,p}}:H^1(X,\Gamma_{\sigma_{p,p}})\to H^1(X,Z)$ intervining  
in Theorem \ref{fixed-connected-a-alpha} is given by the injection
$J_2(X) \hookrightarrow J_{2p}(X)$. Similarly for the map
$\hat{c}_{\sigma_{p,p}}$ intervening  in Theorem \ref{fixed-a-lambda-reps}. 
From Theorems \ref{fixed-connected-a-alpha} and \ref{fixed-a-lambda-reps} 
we see that the fixed point locus of 
$\iota(1,\alpha,+)$  (with $\alpha\neq 1$) is described by the moduli space
$\cM_\alpha(N\SSS(\GL(p,\C)\times \GL(p,\C)),+)$, 
which in turn is homeomorhic to 
$\calR_\alpha(N\SSS(\GL(p,\C)\times \GL(p,\C))$. The fixed points of 
$\iota(1,\alpha,-)$ are
described by the moduli space
$\cM_\alpha(N\SSS(\GL(p,\C)\times \GL(p,\C)),-)$ homeomorphic to 
$\calR_\alpha(N\SU(p,p))$. 
In the latter case $\alpha$ is allowed to be $1$, 
and  in that case 
the fixed point locus is described by 
$\cM(\SSS(\GL(p,\C)\times \GL(p,\C)),-)$ homeomorphic to 
$\calR(\SU(p,p))$. 

As in the $n=2$ case, one can describe the moduli spaces 
$\cM_\alpha(N\SSS(\GL(p,\C)\times \GL(p,\C)),\pm)$ in terms of certain
objects in the $2$-sheeted \'etale cover $\pi: X_\alpha \to X$
defined by $\alpha$. This involves now {\bf generalised Prym varieties}
in the sense of Narasimhan--Ramanan 
(see \cite{narasimhan-ramanan}).

There are no fixed points for the involutions $\iota(-1,\alpha,\pm)$ since
$\Gamma_{\theta_s}=\Gamma_{\theta_*}=\{1\}$.

A direct approach to the study of the involutions
$\iota(a,\pm)$ for $\SL(n,\C)$ is carried out in
\cite{garcia-prada-2007}.

\section{$\Spin(8,\C)$-Higgs bundles and triality}\label{triality}

In this section we give an application of Theorems
\ref{fixed-connected}, \ref{fixed-connected-a} 
\ref{fixed-a-reps}  and  \ref{moduli-n-automorphism}  
to the case $G=\Spin(8,\C)$. This is the (simply connected) simple complex 
Lie group with the largest
group of outer automorphisms, namely $\Out(G)=S_3$ (see Table 
\ref{outer-groups}), thus exhibiting very interesting phenomena.  

\subsection{Involutions}
$G=\Spin(8,\C)$ fits in the exact sequence

\begin{equation}\label{spin}
1 \lra \Z/2 \lra \Spin(8,\C)\lra \SO(8,\C)  \lra 1,
\end{equation}
and hence its Lie algebra is  $\lieso(8,\C)$.
The set (\ref{cartan-classes}) of isomorphism classes  of conjugations 
for $\lieg=\lieso(8,\C)$ is represented (see e.g. \cite{helgason}) 
by the conjugations
$$
\sigma^{p,q}(A)=I_{p,q}\overline{A}I_{p,q} \;\;\mbox{with}\;\; 0\leq p\leq q\;\;\mbox{and}\;\; p+q=8,
$$
where $I_{p,q}$ is given by (\ref{ipq}), corresponding to the real forms
$\lieso(p,q)$, and
$$
\sigma^*(A)=J_4\overline{A}J_4^{-1},
$$
where 
$$
J_4= 
\begin{pmatrix}
  0 & I_4 \\
  -I_4 & 0
\end{pmatrix},
$$
corresponding to the real form $\lieso^*(8)$. 
The compact conjugation is given by $\tau=\sigma^{0,8}$ and the elements 
in $\Aut_2(G)$ corresponding to the conjugations above are 
$\theta^{p,q}=\tau\sigma^{p,q}$ and $\theta^*=\tau\sigma^*$ and thus 
given by
\begin{equation}\label{thetapq}
\theta^{p,q}(A)=I_{p,q} AI_{p,q} \;\;\mbox{with}\;\; 0\leq p\leq q\;\;\mbox{and}\;\; p+q=8,
\end{equation}
and
\begin{equation}\label{theta*}
\theta^*(A)=J_4AJ_4^{-1}.
\end{equation}
Of course, all these conjugations and involutions can be lifted to 
$\Spin(8,\C)$ (denoted in the same way), leading to real
forms $\Spin_0(p,q)$ with $p+q=8$ (see Remark \ref{spin-real-forms})  
and $\Spin^*(8)$.
One can show that $\Spin(8)$, $\Spin_0(2,6)$, $\Spin_0(4,4)$, and
$\Spin^*(8)$ are in the trivial clique, that is, are the real forms
of Hodge type, while  $\Spin_0(1,7)$ and $\Spin_0(3,5)$ are in a non-trivial
clique, say $a_1\in \Out_2(G)$. If $b\in \Out(G)=S^3$ is an element of order 
$3$, the group  $S^3$ can be generated with $a_1$ and $b$. The elements
$a_2:=ba_1b^{-1}$ and $a_3:=b^2a_1b^{-2}$ have also order $2$, and hence if
$B$ is a lift of $b$ to $\Aut_3(G)$, the real groups
$B(\Spin_0(1,7))$ and $B(\Spin_0(3,5))$ and  
$B^2(\Spin_0(1,7))$ and $B^2(\Spin_0(3,5))$ are isomorphic to  
$\Spin_0(1,7)$ and $\Spin_0(3,5)$, respectively, but of course by outer 
isomorphisms. Among the complex simple Lie algebras, this is the 
only case for which there is this type of phenomenon.
To distinguish these different subgroups of $\Spin(8,\C)$
we introduce the notation
$$
\Spin_0(1,7)_i:=B^{i-1}(\Spin_0(1,7))\;\;\mbox{with}\;\; i=1,2,3,
$$
and 
$$
\Spin_0(3,5)_i:=B^{i-1}(\Spin_0(3,5))\;\;\mbox{with}\;\; i=1,2,3.
$$
These correspond to conjugations and holomorphic involutions 
$\sigma^{1,7}_i:=B^{i-1}\sigma^{1,7}B^{1-i}$,
$\theta^{1,7}_i:=B^{i-1}\theta^{1,7}B^{1-i}$,
and 
$\sigma^{3,5}_i:=B^{i-1}\sigma^{3,5}B^{1-i}$,  
$\theta^{3,5}_i:=B^{i-1}\theta^{3,5}B^{1-i}$,
respectively, with $i=1,2,3$.

The subgroup $G^{\theta^{p,q}}$ of fixed points of the 
involution $\theta^{p,q}$ given by
(\ref{thetapq}) is $\Spin(p,\C)\times \Spin(q,\C)$
The decomposition $\lieg=\lieg_{p,q}^+\oplus \lieg_{p,q}^-$ in 
$(\pm 1)$-eingenspaces 
for $\theta^{p,q}$ is given by

$$
\lieg_{p,q}^+=\left\{\left(%
\begin{array}{cc}
  X & 0 \\
  0 & Y \\
\end{array}%
\right)\mid
X\in\mathfrak{so}(p,\C),Y\in\mathfrak{so}(q,\C)\right\},
$$ 
and
$$
\lieg^-_{p,q}
=\left\{\left(%
\begin{array}{cc}
  0 & Z\\
  -Z^t & 0 \\
\end{array}%
\right)\mid Z\text{ complex $(p\times q)$-matrix}\right\}.
$$ 

Clearly, for $p=1,q=7$ one has
$G^{\theta_i^{1,7}}=B^{i-1}(\Spin(1,\C)\times\Spin(7,\C))$ and
the $(\pm 1)$-eingenspace decompostion of $\lieg$ 
for $\theta_i^{1,7}$ is given 
$$
\lieg=B^{i-1}(\lieg^+_{1,7})\oplus B^{i-1}(\lieg^-_{1,7}),
$$
where, as usual, we are denoting by the same letter 
the element in $\Aut(\lieg)$ induced
by an element in $\Aut(G)$.
Similarly for the other outer case corresponding to 
$p=3,q=5$.

The subgroup 
$G^{\theta^*}$ of fixed points of the involution (\ref{theta*}) is the double 
cover determined by 
the exact sequence (\ref{spin}) of the subgroup $\GL(4,\C)\subset \SO(8,\C)$ 
given by elements
$$
 \begin{pmatrix}
  A & 0 \\
  & (A^t)^{-1}
\end{pmatrix},
$$
where $A\in \GL(4,\C)$.
The decomposition $\lieg=\lieg^+\oplus \lieg^-$ in 
$(\pm 1)$-eingenspaces 
for $\theta^*$ is given by
$\lieg^+= \liegl(4,\C)$ and $\lieg^-=\Lambda^2(\C^4)\oplus \Lambda^2(\C^4)^*$.

We have all the ingredients now to apply Theorems \ref{fixed-connected} 
and \ref{fixed-connected-a} to our situation.

\begin{theorem}
Let $G=\Spin(8,\C)$. Consider the involution
$$
   \begin{aligned}
\iota: \cM(G) & \to \cM(G) \\
(E,\varphi) & \mapsto (E,- \varphi).
  \end{aligned}
$$

Then
 
\begin{enumerate}

\item 
$$
\bigcup_{p=0,2,4;\; p+q=8}
\widetilde{\cM}(G^{\theta^{p,q}},-)\bigcup \widetilde{\cM}(G^{\theta^*},-)
\subset \cM(G)^\iota,
$$  

\item
$$
\cM(G)_{ss}^\iota
\subset 
\bigcup_{p=0,2,4;\; p+q=8}
\widetilde{\cM}(G^{\theta^{p,q}},-)\bigcup \widetilde{\cM}(G^{\theta^*},-)
$$
\end{enumerate}
\end{theorem}

\begin{theorem} Let $G=\Spin(8,\C)$ and 
let $1\neq a_i\in \Out_2(G)$ with $i=1,2,3$. Consider the  involutions 

$$
   \begin{aligned}
\iota(a_i,\pm): \cM(G) & \to \cM(G) \\
(E,\varphi) & \mapsto (a_i(E),\pm a_i(\varphi)).
  \end{aligned}
$$
Then
\begin{enumerate}

\item 
$$
\bigcup_{p=1,3;\; p+q=8}
\widetilde{\cM}(G^{\theta_i^{p,q}},\pm)
\subset \cM(G)^{\iota(a_i,\pm)},
$$  

\item
$$
\cM(G)_{ss}^{\iota(a_i,\pm)}
\subset 
\bigcup_{p=1,3;\; p+q=8}
\widetilde{\cM}(G^{\theta_i^{p,q}},\pm)
$$
\end{enumerate}
\end{theorem}

Of course we have the corresponding versions of Theorem \ref{fixed-a-reps} 
regarding the corresponding involutions in the moduli space
of representations $\calR(G)$.

It is clear that the subvarieties 
$\widetilde{\cM}(G^{\theta_i^{p,q}},\pm)$ for $p=1,3$ are transformed one into 
another by the action of the order $3$ element  $b\in \Out(G)$, namely,
$$
\widetilde{\cM}(G^{\theta_i^{p,q}},\pm)=b^{i-1}(\widetilde{\cM}(G^{\theta^{p,q}},\pm)),
$$
for $i=1,2,3$.

\subsection{Order $3$ automorphisms}
Let $b\in \Out_3(G)$, and let
$\zeta_k:=\exp(2\pi i\frac{ k}{3})$, with $k=0,1,2$. We study now
the order $3$ automorphism of $\cM(G)$ given by
$(E,\varphi)\mapsto (b(E),\zeta_kb(\varphi))$. 
To do this, we need to understand the $3$-clique $b$, that is 
we need to know the set $\cl_3^{-1}(b)\in \Aut_3(G)/\sim$.
This is studied in \cite{wolf-gray}. It turns out that $\cl_3^{-1}(b)$ consists
of two classes. One can choose representatives $\theta_1,\theta_2\in \Aut_3(G)$
of these two classes such that
$$
G^{\theta_1}=\PSL(3,\C)\subset\Spin(8,\C)
$$
and the decomposition of $\lieg$ in  $\zeta_k$-eigenspaces for $k=0,1,2$ is
$$
\lieg=\lieg^0\oplus \lieg^1 \oplus \lieg^2=
\liesl(3,\C)\oplus S^2(\C^3) \oplus S^2(\C^3)^*,
$$
where  $S^2(\C^3)$ is the $2$-symmetric tensor product of the fundamental
representation of $\SL(3,\C)$, which is thus $10$-dimensional;
and 
$$
G^{\theta_2}= G_2\subset\Spin(8,\C),
$$
and
$$
\lieg=\lieg^0\oplus \lieg^1 \oplus \lieg^2=
\lieg_2 \oplus \C^7\oplus \C_7,
$$
where $\C^7$ is the fundamental representation of $G_2$. 

We can now apply Theorem \ref{moduli-n-automorphism} to
obtain the following. 

\begin{theorem}
Let $G=\Spin(8,\C)$, $b$ be a non-trivial element of order $3$ in 
$\Out(G)=S^3$, and $\zeta_k:=\exp(2\pi i\frac{ k}{3})$, with $k=0,1,2$.
Consider the order $3$ automorphism
$$
\begin{aligned}
\iota(b,\zeta_k): \cM(G) & \to \cM(G) \\
(E,\varphi) & \mapsto (b(E),\zeta_k b(\varphi)).
  \end{aligned}
$$
Then
\begin{enumerate}

\item 
$$
\widetilde{\cM}(\PSL(3,\C),\zeta_k)\cup
\widetilde{\cM}(G_2,\zeta_k)
\subset \cM(G)^{\iota(b,\zeta_k)},
$$  

\item
$$
\cM(G)_{ss}^{\iota(b,\zeta_k)}
\subset 
\widetilde{\cM}(\PSL(3,\C),\zeta_k)\cup
\widetilde{\cM}(G_2,\zeta_k).
$$
\end{enumerate}

\end{theorem}

A direct approach to this case has been given in \cite{alvaro},
where the stability conditions for the $(\PSL(3,\C),\zeta_k)$- and
$(G_2,\zeta_k)$-Higgs bundles have been explictly worked out.

One has
$b^{-1}=a_iba_i$ for any $a\neq 1$ in $\Out_2(G)$, and hence if $A\in \Aut_2(G)$
is a lift of $a$,  the $3$-clique $b^{-1}$ has 
$\theta_1'=A\theta_1 A^{-1}$ and $\theta_2'=A\theta_2 A^{-1}$ as 
representatives of the two
classes in $\cl_3^{-1}(b^{-1})$, thus having  
$G^{\theta_1'}=A(\PSL(3,\C))$ and $G^{\theta_2'}=A(G_2)$.
The fixed points of $\iota(b,\zeta_k)$ are hence moved to the fixed
points of $\iota(b^{-1},\zeta_k)$ by the action of $b$ on $\cM(G)$.

\providecommand{\bysame}{\leavevmode\hbox to3em{\hrulefill}\thinspace}


\begin{thebibliography}{99}

\bibitem{adams} J. Adams, Galois and $\theta$ Cohomology of Real Groups,
preprint (2014), arXiv:1310.7917v3.

\bibitem{alvaro} \'A.  Ant\'on Sancho, {\em Higgs Bundles and Triality},
PhD Thesis, UCM-CSIC, 2009.

\bibitem{atiyah} M. F. Atiyah, Complex analytic connections in fibre bundles, 
{\it Trans. Amer. Math. Soc.} {\bf 85} (1957), 181--207.

\bibitem{baraglia-schaposnik}
D. Baraglia and L. P. Schaposnik, Real structures on moduli spaces of 
Higgs bundles, {\it Adv. Theo. Math. Phys.} {\bf 20} (2016), 525--551.

\bibitem{basu-garcia-prada}
S. Basu and O. Garc\'{\i}a-Prada, Finite group actions on Higgs bundle moduli 
spaces, in preparation.

\bibitem{biquard-garcia-prada-rubio}
O. Biquard, O. Garc\'{\i}a-Prada and R. Rubio,
Higgs bundles, the Toledo invariant and the Cayley correspondence,
{\em J. of Topology} {\bf 10} (2017), 795--826.

\bibitem{biswas-garcia-prada}
I. Biswas and O. Garc\'{\i}a-Prada,
 Antiholomorphic involutions of the moduli spaces of Higgs bundles,
{\it Jour. \'Ecole Polytech. -- Math.} {\bf 2} (2015), 35--54.

\bibitem{biswas-gomez}
{I. Biswas and T. L. G\'omez}, Connections and Higgs fields on a principal bundle,
 \textsl{Ann. Glob. Anal. Geom.} \textbf{33} (2008), 19--46.

\bibitem{borel-de-siebenthal}
A. Borel and J. de Siebenthal,  Les sous-groups ferm\'es de 
rang maximum de groups de Lie closed, \textsl{Commentari, Math. Helv.}
\textbf{22} (1949), 200--221.

\bibitem{bradlow-garcia-prada-gothen}
S.B. Bradlow, O. Garc\'{\i}a-Prada and P.B. Gothen, Surface group 
representations and $\mathrm{U}(p,q)$-{H}iggs
  bundles, \textsl{J. Differential Geom.} \textbf{64} (2003), 111--170.

\bibitem{bradlow-garcia-prada-gothen2}
\bysame, Maximal surface group representations  in 
isometry groups of classical Hermitian symmetric spaces, 
\textsl{Geom. Dedicata} \textbf{122} (2007), 185--213.

\bibitem{cartan}
\'E. Cartan, Les groupes r\'eels simples, finis et continus,
\textsl{Ann. \'Ec. Norm. Sup.} \textbf{31} (1914), 263--355.

\bibitem{collier}
B. Collier, {\em Finite Order Automorphisms of Higgs Bundles: Theory and 
Application}, PhD Thesis UIUC, 2016.

\bibitem{corlette}
K.~Corlette, Flat ${G}$-bundles with canonical metrics, \textsl{J. Differential
  Geom.} \textbf{28} (1988), 361--382.

\bibitem{donagi-pantev}
R. Donagi and T. Pantev, Langlands duality for Hitchin systems, 
\textsl{Invent. Math.} \textbf{189} (2012), 653--735.

\bibitem{donaldson}
S.~K. Donaldson, Twisted harmonic maps and the self-duality equations,
  \textsl{Proc. London Math. Soc.} (3) \textbf{55} (1987), 127--131.

\bibitem{garcia-prada-2007}
O.~Garc{\'\i}a-Prada,
Involutions of the moduli space of $\SL(n,C)$-Higgs bundles and real forms,
In Vector Bundles and Low Codimensional Subvarieties: State of the Art
and Recent Developments" Quaderni di Matematica, Editors: G. Casnati, F.
Catanese and R. Notari, 2007.

\bibitem{garcia-prada}\bysame, 
Higgs bundles and surface group representations, in 
\emph{Moduli Spaces and Vector Bundles}, LMS Lecture Notes Series 359,
265--310, Cambridge University Press, 2009.

\bibitem{garcia-prada-gothen-mundet}
O.~Garc{\'\i}a-Prada, P.~B. Gothen, and I.~Mundet i~Riera,
  The Hitchin--Kobayashi correspondence, Higgs pairs and surface
  group representations, 2009, \texttt{arXiv:0909.4487 [math.AG]}.

\bibitem{garcia-prada-gothen-mundet2}
\bysame,
Hitchin--Kobayashi correspondence for Higgs pairs, in preparation.

\bibitem{garcia-prada-oliveira-u*} 
O.~Garc{\'\i}a-Prada, and A. Oliveira,
Higgs bundles for the non-compact dual of the
unitary group {\em Illinois Journal of Mathematics}  {\bf 55} No. 3 (2011), 
1155--1181.

\bibitem{garcia-prada-oliveira} 
\bysame,
Connectedness of Higgs bundle moduli for complex reductive Lie groups,
{\em Asian Journal of Mathematics}, {\bf 21} (2017), 791--810. 

\bibitem{garcia-prada-ramanan-rank2} 
O.~Garc{\'\i}a-Prada, and S. Ramanan, Involutions of rank 2 Higgs bundle moduli 
spaces, in {\em Geometry and Physics: A Festschrift in honour of Nigel Hitchin}, Oxford University Press, 2018.




\bibitem{goldman}
W. Goldman, Topological components of spaces of representations,{\em Invent.
  Math.} \textbf{93} (1988), 557--607.

\bibitem{goto-kobayashi}
M. Goto and E.T. Kobayashi, 
On the subgroups of the centers of simply 
connected simple Lie groups --- classification of simple Lie groups in the
large, \textsl{Osaka J. Math.} {\bf 6} (1969), 251--281.

\bibitem{helgason}
S.~Helgason, \emph{Differential geometry, {L}ie groups, and symmetric spaces},
  Mathematics, vol.~80, Academic Press, San Diego, 1998.

\bibitem{hitchin1987}
N.~J. Hitchin, The self-duality equations on a {R}iemann surface, \textsl{Proc.
  London Math. Soc.} (3) \textbf{55} (1987), 59--126.

\bibitem{hitchin:duke}
\bysame, Stable bundles and integrable systems, \textsl{Duke
Math. J.} \textbf{54} (1987), 91--114.

\bibitem{hitchin1992}
\bysame, Lie groups and Teichm\"{u}ller space, \textsl{Topology} \textbf{31}
  (1992), 449--473.

\bibitem{hitchin2013}
\bysame, Higgs bundles and characteristic classes,
in {\em ``Arbeitstagung Bonn 2013: In Memory of Friedrich Hirzebruch"}, 
W.Ballmann et al (eds), Birkh\"auser Progresss in Mathematics 319, 2016.

\bibitem{kapustin-witten} A. Kapustin and E. Witten, Electric magnetic duality
and the geometric Langlands programme, \emph{Commun.  Number Theory Phys.}
{\bf 1} (2007), 1--236.

\bibitem{kobayashi:1987}
S. Kobayashi.
\newblock {\em Differential geometry of complex vector bundles}, volume~15 of
  {\em Publications of the Mathematical Society of Japan}.
\newblock Princeton University Press, Princeton, NJ, 1987.
\newblock Kan{\^o} Memorial Lectures, 5.

\bibitem{kobayashi-nomizu}
S. Kobayashi and  K. Nomizu, {\it Foundations of Differential Geometry},
Vols. I and II, Interscience, New York, 1963 and 1969.


\bibitem{labourie}
Fran\c{c}ois Labourie,
 Existence d'applications harmoniques tordues \`a valeurs dans les
vari\'et\'es \`a courbure n\'egative,
{\em  Proc. Amer. Math. Soc.} {\bf 111} (1991) 877--882.

\bibitem{li}
J. Li, The Space of Surface Group Representations, \textsl{Manuscripta Math.} 
\textbf{78} (1993), 223--243.

\bibitem{martin:2003}
{B. Martin}, Reductive subgroups of reductive groups in nonzero characteristic, \textsl{J. Algebra} \textbf{262} (2003), 265--286.

\bibitem{milnor}
J. Milnor, On the existence of a connection with curvature
zero \textsl{Comment. Math. Helv.}
\textbf{32} (1958) 215-223.

\bibitem{narasimhan-ramanan}
M.~S. Narasimhan and S. Ramanan, \emph{Generalised Prym varieties as fixed points}, 
Journal of the Indian Math. Soc. \textbf{39} (1975), 1--19.

\bibitem{narasimhan-seshadri}
M.~S. Narasimhan and C.~S. Seshadri, Stable and unitary vector bundles on
  a compact {R}iemann surface, \textsl{Ann. Math.} \textbf{82} (1965), 540--567.

\bibitem{nitsure}
     N. Nitsure,
     Moduli spaces of semistable pairs on a curve,
     {\em Proc. London Math. Soc.} {\bf 62} (1991), 275--300.

\bibitem{onishchik} A.L. Onishchik, \textsl{Lectures on Real Semisimple Lie
  Algebras  and their Representations}, EMS, 2004.

\bibitem{onishchik-vinberg} A.L. Onishchik and E.B. Vinberg (Eds.),
\textsl{Lie Groups and  Lie
Algebras III}, Encyclopedia of Mathematical Sciences, Vol. 41,
Springer--Verlag 1994.

\bibitem{ramanan} S. Ramanan, Some aspects of the
theory of Higgs pairs, in  {\em The Many Facets of Geometry. 
A tribute to Nigel Hitchin}, Oxford University Press, 2010, Editors:
O. Garc\'{\i}a-Prada, J.P.  Bourguignon and
S. Salamon.

\bibitem{ramanathan:1975}
{A. Ramanathan}, Stable principal bundles on a compact
Riemann surface, \textsl{Math. Ann.} \textbf{213} (1975), 129--152.

\bibitem{ramanathan:1996b}
{A. Ramanathan}, Moduli for principal bundles over algebraic
curves: II, \textsl{Proc. Indian Acad. Sci. (Math. Sci.)} \textbf{106}
(1996), 421--449.

\bibitem{schmitt:2008}
{A. H. W. Schmitt}, \textsl{Geometric Invariant Theory and 
Decorated Principal Bundles}, Zurich Lectures in Advanced Mathematics, 
European Mathematical Society, 2008.

\bibitem{serre} J.-P. Serre, \textsl{Cohomology Galoisienne}, LNM 5,
Springer--Verlag, 1964.

\bibitem{de-siebenthal}
J. de Siebenthal,
Sur les groupes de Lie compact non-connexes,
\textsl{Commentari Math. Helv.} \textbf{31} (1956) 41--89.


\bibitem{simpson}
C.~T. Simpson, Constructing variations of {H}odge structure using
  {Y}ang-{M}ills theory and applications to uniformization, 
\textsl{J. Amer. Math.  Soc.} \textbf{1} (1988), 867--918.

\bibitem{simpson:1992}
\bysame, Higgs bundles and local systems, 
\textsl{Inst. Hautes {\'E}tudes Sci.  Publ. Math.} \textbf{75} (1992), 5--95.


\bibitem{simpson:1994}
\bysame
Moduli of representations of the fundamental group of a smooth
  projective variety {I}, \textsl{Publ. Math., Inst. Hautes \'Etud. Sci.} \textbf{79}
  (1994), 47--129.

\bibitem{simpson:1995}
\bysame
Moduli of representations of the fundamental group of a smooth
  projective variety {II}, \textsl{Publ. Math., Inst. Hautes \'Etud. Sci.} \textbf{80}
  (1995), 5--79.

\bibitem{simpson:2009}
\bysame
Katz's middle convolution algorithm, \textsl{Pure Appl. Math. Q.} 
{\bf 5} (2009), no. 2. Special Issue: In honor of Friedrich Hirzebruch. 
Part 1, 781--852. 

\bibitem{wolf-gray}
J.A. Wolf, and A. Gray,  Homogeneous spaces defined by Lie group
automorphisms. I {\em J. Diff. Geom.} {\bf 2} (1968), 77--114.

\end{thebibliography}
\end{document}